\pdfoutput=1

\documentclass[a4paper, twoside]{article}

\usepackage{amssymb}
\usepackage{amsmath}
\usepackage{amsthm}

\usepackage[utf8]{inputenc}
\usepackage[T1]{fontenc}   
\usepackage{lmodern}       
\usepackage{mathrsfs}      
\usepackage{upgreek}       
\usepackage{microtype}     
\usepackage{xurl}          
\usepackage{hyperref}      
\usepackage[nottoc,numbib]{tocbibind} 

\newcounter{enumi_memory}  

\theoremstyle{definition}
\newtheorem{question}{Question}
\newtheorem*{intro-definition}{Definition}

\theoremstyle{plain}
\newtheorem{hypotheses}{Hypotheses}
\newtheorem*{Final-example}{Example}
\newtheorem{FinalThm}{Theorem}

\newtheorem{Final-corollary}{Corollary}

\swapnumbers
\theoremstyle{plain}
\newtheorem{theorem}{Theorem}[section]
\newtheorem{lemma}[theorem]{Lemma}
\newtheorem{corollary}[theorem]{Corollary}

\theoremstyle{remark}
\newtheorem{remark}[theorem]{Remark}

\theoremstyle{definition}
\newtheorem{definition}[theorem]{Definition}
\newtheorem{miniremark}[theorem]{}
\newtheorem{example}[theorem]{Example}

\newtheoremstyle{citing}
  {3pt}
  {3pt}
  {\itshape}
  {}
  {}
  {\textbf.}
  {.5em}
  {\thmnote{#3}}
\theoremstyle{citing}
\newtheorem*{citing}{}

\DeclareMathOperator{\without}{\sim}

\newcommand{\restrict}{\mathop{\llcorner}}

\DeclareMathOperator{\card}{card}

\DeclareMathOperator{\with}{:}
\DeclareMathOperator{\reach}{reach} 
\DeclareMathOperator{\Bdry}{Bdry}
\DeclareMathOperator{\Clos}{Clos}
\DeclareMathOperator{\Tan}{Tan}     
\DeclareMathOperator{\Nor}{Nor}		
\DeclareMathOperator{\spt}{spt}     
\DeclareMathOperator{\im}{im}       
\DeclareMathOperator{\Int}{Int}     
\DeclareMathOperator{\diam}{diam}   
\DeclareMathOperator{\Lip}{Lip}     
\DeclareMathOperator{\dmn}{dmn}     
\DeclareMathOperator{\dist}{dist}   
\DeclareMathOperator{\Hom}{Hom}     
\DeclareMathOperator{\Der}{D}       
\DeclareMathOperator{\weakD}{\mathbf{D}}  
\DeclareMathOperator{\boundary}{\partial} 
\DeclareMathOperator{\ap}{ap}       

\newcommand{\ud}{\,\mathrm{d}}

\title{A priori bounds for geodesic diameter. Part III. \\ A Sobolev-Poincaré
inequality and applications to a variety of geometric variational problems}

\author{Ulrich Menne \and Christian Scharrer}

\begin{document}

\maketitle 
 
\begin{abstract}
	Based on a novel type of Sobolev-Poincaré inequality (for generalised
	weakly differentiable functions on varifolds), we establish a
	finite upper bound of the geodesic diameter of generalised compact
	connected surfaces-with-boundary of arbitrary dimension in Euclidean
	space in terms of the mean curvatures of the surface and its boundary.
	Our varifold setting includes smooth immersions, surfaces with finite
	Willmore energy, two-convex hypersurfaces in level-set mean curvature
	flow, integral currents with prescribed mean curvature vector, area
	minimising integral chains with coefficients in a complete normed
	commutative group, varifold solutions to Plateau's problem furnished
	by min-max methods or by Brakke flow, and compact
	sets solving Plateau problems based on Čech homology.  Due to the
	generally inevitable presence of singularities, path-connectedness was
	previously known neither for the class of varifolds (even in the
	absence of boundary) nor for the solutions to the Plateau problems
	considered.
\end{abstract}

\paragraph{MSC-classes 2020} 53A07 (Primary); 46E35,  49Q05, 49Q15,  53A10,
53C22 (Secondary).

\paragraph{Keywords} Sobolev-Poincaré inequality $\cdot$
area-stationary set $\cdot$ geodesic distance $\cdot$
diameter $\cdot$ varifold $\cdot$ indecomposability $\cdot$ mean curvature
$\cdot$ boundary $\cdot$ density $\cdot$ Plateau problem $\cdot$ integral
current $\cdot$ integral $G$ chain $\cdot$ Čech homology.

\tableofcontents

\section{Introduction}

Throughout this introduction, we suppose \emph{$m$ and $n$ are integers, $2
\leq m \leq n$, and $B$ is a nonempty compact $m-1$ dimensional submanifold of
class $2$ of $\mathbf R^n$}.

\begin{intro-definition} [see \ref{def:geodesic_distance}]
	For closed subsets $A$ of $\mathbf R^n$, their \emph{geodesic
	diameter} is the supremum of all numbers $\sigma (a,x)$ corresponding
	to $a,x \in A$, where $\sigma (a,x)$ is the infimum of the set of
	lengths of continuous paths in $A$ connecting $a$ and $x$.
\end{intro-definition}

We derive upper bounds for the geodesic diameter of sets $A$ associated with
solutions to a variety of geometric variational problems.
They imply finiteness and are given solely in terms of the boundary data and
the dimensions involved.

\subsection{Plateau problems}

We illustrate our varifold-theoretic results by their implication on the two
most prominent formulations of Plateau's problem in geometric measure theory.

\begin{FinalThm} [see \ref{theorem:diameter_bound_mass_minimising_chains} and
\ref{remark:diameter_bound_mass_minimising_currents}]
\label{Thm:Plateau:chains}
	Suppose $S \in \mathbf I_m ( \mathbf R^n )$, $\boundary S$ is
	indecomposable, $\| \boundary S \| = \mathscr H^{m-1} \restrict B$,
	\begin{equation*}
		\| S \| ( \mathbf R^n ) \leq \| T \| ( \mathbf R^n ) \quad
		\text{whenever $T \in \mathbf I_m ( \mathbf R^n )$ and
		$\boundary T = \boundary S$},
	\end{equation*}
	and $d$ denotes the geodesic diameter of $\spt \| S \|$.

	Then, for some positive finite number $\Gamma$ determined by $m$,
	there holds
	\begin{equation*}
		{\textstyle d \leq \Gamma \int_B | \mathbf h ( B, b) |^{m-2}
		\ud \mathscr H^{m-1} \, b};
	\end{equation*}
	here, by convention, we stipulate $\int_B | \mathbf h (B,b) |^0 \ud
	\mathscr H^1 \, b = \mathscr H^1 (B)$ regarding $m = 2$.
\end{FinalThm}
By the fundamental results of H.\ Federer and W.\ Fleming, there exists an
absolutely area minimising integral current $S$ satisfying the hypotheses of
Theorem \ref{Thm:Plateau:chains} whenever $B$ is connected and orientable.
Even in case $A = \spt \| S \|$ is an $m$ dimensional
submanifold-with-boundary of class $2$ and $\partial A = B$, no a priori
estimate for the geodesic diameter was known prior to this work.  If $m<n-1$,
even finiteness of $d$ is new; in fact, it was only known that, in general,
$A$ is connected (an elementary fact shown by the second paper
\cite{arXiv:2209.05955v2} of our series, see
\ref{remark:connected-integral-minimisers} and
\ref{remark:connected-minimisers}) and that, in case $B$ is of class $4$, much
more profoundly $A \without B$ is connected (by results in \cite{MR0397520}
and \cite{MR1777737} due to W.\ Allard and F.\ Almgren, respectively, see
\ref{remark:earlier-results}) and the geodesic distance on $A \without B$ in
the sense of \cite[Definition 6.6]{MR3626845} is real valued and continuous
with respect to the Euclidean metric (by \cite[Theorem 6.8\,(1)]{MR3626845} by
the first author).  If $m = n-1$, the properties of the geodesic distance on
$A \without B$ may be combined with the studies of R.\ Hardt and L.\ Simon in
\cite{MR554379} to yield finiteness of $d$ without providing any a priori
estimate on $d$, see \ref{remark:earlier-results}.  Theorem
\ref{Thm:Plateau:chains} transfers unchanged to the context of flat chains
modulo $\nu$; in fact, in \ref{theorem:diameter_bound_mass_minimising_chains},
we provide a formulation of Theorem \ref{Thm:Plateau:chains} in terms of the
integral chains with coefficients in a complete normed commutative group
constructed in the first paper of our series \cite{arXiv:2206.14046v2}, see
\ref{remark:classic-Abelian-groups}.

The role of the indecomposability hypothesis on $\boundary S$ is merely to
guarantee the indecomposability of $S$, see
\ref{theorem:diameter_bound_mass_minimising_chains} and
\ref{remark:diameter_bound_mass_minimising_currents}.  Thus, we obtain a
nonexistence criterion for indecomposable solutions to Plateau's problem, see
\ref{remark:nonexistence}.

Next, we describe the theorem which results from combining our result on
area-stationary sets in \ref{corollary:Reifenberg-Plateau-diameter-sets} with
the studies of the Reifenberg-type Plateau problem by F.\ Almgren, H.\ Pugh,
and C.\ Labourie in \cite{MR0420406,MR3998213,MR4489608}.
\begin{FinalThm} [see \ref{remark:Reifenberg-Plateau}]
\label{FinalThm:Reifenberg-Plateau}
	Suppose $B$ is connected, $G$ is a commutative group, $L$ is a
	subgroup of the $(m-1)$-th Čech homology group of $B$ with
	coefficients in $G$, $\mathscr{\check C} (B,L,G)$ denotes the family
	of compact subsets of $\mathbf R^n$ spanning $L$,
	\begin{equation*}
		E \in \mathscr{\check C} (B,L,G), \quad \mathscr H^m (E) =
		\inf \big \{ \mathscr H^m (F) \with F \in \mathscr{\check C}
		(B,L,G) \big \},
	\end{equation*}
	$A = \spt ( \mathscr H^m \restrict E )$, and $d$ is the geodesic
	diameter of $A$.

	Then, for some positive finite number $\Gamma$ determined by $n$,
	there holds
	\begin{equation*}
		{\textstyle d \leq \Gamma \reach(B)^{-m} \mathscr H^{m-1}
		(B)^{m/(m-1)} \int_B | \mathbf h(B,b)|^{m-2} \ud
		\mathscr H^{m-1} \, b};
	\end{equation*}
	here, by convention, we stipulate $\int_B | \mathbf h (B,b) |^0 \ud
	\mathscr H^1 \, b = \mathscr H^1 (B)$ regarding $m = 2$.
\end{FinalThm}
This yields an a priori bound on $d$ determined by the boundary $B$ and the
dimension $n$.  As before, no bound was known even in the smooth case and, due
to inevitable singularities, only points in $A \without B$ in the same
connected component of $A$ were known to admit a connecting path in $A$ of
finite length.

\subsection{The applications to geometric variational problems}

The scope of our present results is much wider than the
relatively regular setting of minimisers of Plateau problems: It encompasses
the Willmore energy with clamped boundary condition studied by M.\ Novaga and
M.\ Pozzetta, see \ref{remark:clamped-Willmore-II}, level-set mean curvature
flow of two-convex hypersurfaces as studied by P.\ Gianniotis and R.\
Haslhofer in \cite{MR4176547}, see \ref{remark:two-convex-MCF},
$\lambda$-minimising currents (which in turn include integral currents with
prescribed mean curvature vector and codimension-one area minimising integral
currents with prescribed volume, as studied by F.\ Duzaar, K.\ Steffen, and
M.\ Fuchs in
\cite{MR1243155,MR1037996,MR1200740,MR1220033,MR1156438,MR1383909}, see
\ref{remark:Duzaar-Steffen-Fuchs}), integral varifolds stationary in $\mathbf
R^n \without B$ as furnished either by Brakke flow with fixed boundary studied
by S.\ Stuvard and Y.\ Tonegawa in \cite{MR4204569}, see
\ref{remark:Brakke-flow} and \ref{remark:Brakke-flow-II}, or by min-max
methods studied by C.\ De Lellis, J.\ Ramic, and R.\ Montezuma in
\cite{MR3893761,MR4160865}, see \ref{remark:min-max}.

\subsection{The general results in the varifold-setting}

Here, we discuss our key results---Theorems \ref{Thm:new-Sobolev-Poincare},
\ref{Thm:varifold-density-lower-bound}, and
\ref{Thm:varifold-diameter-estimate}---in the varifold-setting.  The following
subsection will then include corresponding statements---Corollaries
\ref{Final-corollary:new-Sobolev-Poincare},
\ref{Final-corollary:lower-density-bound}, and
\ref{Final-corollary:diameter-bound-immersions}---formulated in a purely
differential-geometric setting as well as two further theorems in the varifold
setting---Theorems \ref{FinalThm:geom-var-I} and \ref{FinalThm:geom-var-II}.
The latter are tailored for applications to geometric variational problems.

\begin{hypotheses} [First variation] \label{hyp:delta_v}
	Suppose $V$ is an $m$ dimensional varifold in an open subset $U$ of
	$\mathbf R^n$ and the first variation,
	\begin{equation*}
		\updelta V \in \mathscr D' ( U, \mathbf R^n ),
	\end{equation*}
	of $V$ is representable by integration (equivalently, the variation
	measure, $\| \updelta V \|$, of $\updelta V$ is a Radon measure).
\end{hypotheses}

We recall that an $m$ dimensional varifold $V$ in $U$ is a Radon measure over
the Cartesian product of $U$ with the Grassmann manifold $\mathbf G (n,m)$
consisting of all (unoriented) $m$ dimensional vector subspaces of $\mathbf
R^n$, that the projection of $V$ onto the first factor is termed the weight of
$V$ and is denoted by $\| V \|$, and that $\mathbf T ( V )$ is the class
consisting of all real valued generalised $V$ weakly differentiable functions,
see \cite[Definition 4.2]{MR3777387} by the authors, which includes all
locally Lipschitzian functions $f : U \to \mathbf R$, see \cite[Lemma
4.6\,(1)]{MR3777387}.

\begin{intro-definition} [see \protect{\cite[Definition 5.1]{MR3528825}}
by the first author]
	Suppose $V$ satisfies the Hypotheses~\ref{hyp:delta_v} and $E$ is
	measurable with respect to $\| V \|$ and $\| \updelta V \|$.

	Then, the \emph{distributional $V$ boundary} of $E$ is defined by
	\begin{equation*}
		V \boundary E = ( \updelta V ) \restrict E - \updelta ( V
		\restrict E \times \mathbf G (n,m) ) \in \mathscr D' ( U,
		\mathbf R^n ).
	\end{equation*}
\end{intro-definition}

If $V$ and $E$ are suitably regular (see H.\ Federer's characterisation of
sets of locally finite perimeter in \cite[4.5.11]{MR41:1976}), then $V
\boundary E$ may be expressed in terms of the exterior normal of $E$ (with
respect to $V$), see \cite[Theorem 5.9]{MR3528825}.  The notion of boundary
allows us to formulate indecomposability of a varifold by considering, for a
given class of functions, how many of its superlevel sets
split the varifold in a nontrivial way and yet have no distributional
boundary.

\begin{intro-definition}
	[\protect{see \cite[7.1]{arXiv:2209.05955v2}}]
\label{def:indecomposability}

	Suppose that $V$ satisfies the Hypotheses \ref{hyp:delta_v} and
	that $\Psi \subset \mathbf T ( V )$.

	Then, $V$ is called \emph{indecomposable of type $\Psi$} if and only
	if, whenever $f \in \Psi$, the set of $y \in \mathbf R$, such that
	$E(y) = \{x \with f(x)>y\}$ satisfies
	\begin{equation*}
		\|V\|( E(y) )>0, \quad \|V\| ( U \without E(y) )>0, \quad V
		\boundary E (y) = 0
	\end{equation*}
	has $\mathscr L^1$ measure zero.
\end{intro-definition}

It is crucial for the present final paper of our series,
that indecomposability of type $\mathscr D (U,\mathbf R)$ is strictly weaker
than indecomposability as defined in \cite[Definition 6.2]{MR3528825}.  An
in-depth comparison of notions of indecomposability has been carried out in
the second paper \cite{arXiv:2209.05955v2}; for instance, two touching spheres
yield a varifold which is decomposable but indecomposable of type $\mathscr D
( U, \mathbf R)$, see Example 2 and Corollary 2 therein.  The formulation of
our main results involves two further sets of hypotheses whose meaning and
significance we shall discuss next.

\begin{hypotheses} [Density and mean curvature] \label{hyp:mean_curvature}
	Suppose $V$ is an $m$ dimensional varifold in an open subset $U$ of
	$\mathbf R^n$, $\| \updelta V \|$ is a Radon measure absolutely
	continuous with respect to $\| V \|$, $\boldsymbol \Uptheta^m ( \| V
	\|, x ) \geq 1$ for $\| V \|$ almost all $x$, $1 \leq p \leq \infty$,
	and the generalised mean curvature vector $\mathbf h (V,\cdot)$ of $V$
	belongs to $\mathbf L_p^\mathrm{loc} ( \| V \|, \mathbf R^n )$.
\end{hypotheses}

Unlike in the differential-geometric case, $\mathbf h (V, \cdot )$ may have a
nontrivial tangential component related to variations of $\boldsymbol
\Uptheta^m ( \| V \|, \cdot)$.  In this regard, considering the example of a
weighted properly embedded smooth submanifold (see \cite[Remark 7.6, Lemma
15.2]{MR3528825}), the following question seems natural; if $V$ is integral,
an affirmative answer follows from \cite[Theorem 4.8]{MR3023856} of the first
author.

\begin{question}
	Suppose $V$ satisfies the Hypotheses \ref{hyp:delta_v}, $\tau = \Tan^m
	( \| V \|, \cdot )_\natural$ is the tangent plane%
	\begin{footnote}
		{For $\| V \|$ almost all $x$, the closed cone $\Tan^m ( \| V
		\|, x )$ is an $m$ dimensional plane and $\tau (x)$ is the
		orthogonal projection retracting $\mathbf R^n$ onto $\Tan^m (
		\| V \|, x )$.}
	\end{footnote}%
	function, and $\Theta (x) \geq 1$ for $\| V \|$ almost all $x$, where
	$\Theta = \boldsymbol \Uptheta^m ( \| V \|, \cdot )$.  Does it follow
	that both functions $\tau$ and $\Theta$ are $( \| V \|, m )$
	approximately differentiable at $\| V \|$ almost all $x$, and, if so,
	does there hold---denoting $(\|V \|, m )$ approximate derivatives by
	the prefix ``$\ap$''---the equation
	\begin{equation*}
		\mathbf h ( V, x ) \bullet u = T ( \ap \Der \tau (x) \circ
		\tau (x) ) \bullet u + ( \ap \Der ( \log \circ \, \Theta ) (x)
		\circ \tau (x) ) (u) \quad \text{for $u \in \mathbf R^n$}
	\end{equation*}
	for $\| V \|$ almost all $x$, where the trace operator $T$ is as in
	\cite[15.1]{MR3528825}?
\end{question}

Concerning the significance of the Hypotheses \ref{hyp:mean_curvature}, we
recall that, if $p \geq m$, then $\spt \| V \|$ is in many ways well-behaved:
For instance, there holds
\begin{equation*}
	\boldsymbol \Uptheta^m_\ast ( \| V \|, x ) \geq 1 \quad \text{for $x
	\in \spt \| V \|$}
\end{equation*}
by \cite[Remark 2.7]{MR2537022} of the first author---in particular, $\spt \|
V \|$ has locally finite $\mathscr H^m$ measure---, the set $\spt \| V \|$ is
locally connected (see \cite[Corollary 6.14\,(3)]{MR3528825}), decompositions
of $V$ are locally finite (see \cite[Remark 6.11]{MR3528825}) and non-uniquely
refine the decomposition of $\spt \| V \|$ into connected components (see
\cite[Remark 6.13, Corollary 6.14\,(1)]{MR3528825}), connected components of
$\spt \| V \|$ are locally connected by paths of finite length (see
\cite[Theorem 14.2]{MR3528825}), and the resulting geodesic distance thereon
is a continuous Sobolev function with bounded generalised weak derivative, see
\cite[Theorem 6.8\,(1)]{MR3626845}.  A substantial challenge for the present
development arises from the fact that, if $p < m$, then $\spt \| V \|$ has
substantially less geometric significance: Whenever $X$ is an open subset of
$U$, there exists a varifold $V$ such that $\spt \| V \|$ equals the closure
of $X$ relative to $U$, see \cite[Example 14.1]{MR3528825}.  However, one is
at least assured that $\mathscr H^{m-p}$ almost all $x \in
\spt \| V \|$ satisfy the dichotomy
\begin{equation*}
	\text{either $\boldsymbol \Uptheta^m_\ast ( \| V \|, x ) \geq 1$}
	\quad \text{or $\boldsymbol \Uptheta^m ( \| V \|, x ) = 0$}
\end{equation*}
by \cite[Remark 2.11]{MR2537022}.

The next set of hypotheses concerns the formulation of a boundary condition
for varifolds.  This is complicated by the absence of a boundary operator as
is available for currents.  In this regard, the distribution $B \in \mathscr
D' ( U, \mathbf R^n)$ defined by
\begin{equation*}
	B ( \theta ) = ( \updelta V ) ( \theta ) + {\textstyle\int} \mathbf h
	( V,x ) \bullet \theta ( x)  \ud \| V \| \, x \quad \text{for $\theta
	\in \mathscr D ( U, \mathbf R^n )$}
\end{equation*}
may act as a replacement whenever $V$ satisfies the Hypotheses
\ref{hyp:delta_v}.  According to W.\ Allard in \cite[4.3]{MR0307015}, $\| B \|
= \| \updelta V \|_{\mathrm{sing}}$ \emph{in some sense} is the boundary of
$V$; here we have employed $\| \updelta V \|_{\mathrm{sing}}$ to denote the
unique Radon measure over $U$ such that
\begin{equation*}
	\| \updelta V \| = \| \updelta V \|_{\| V \|} + \| \updelta V
	\|_{\mathrm{sing}}.
\end{equation*}
Yet, it appears more accurate to consider $\| B \|$ as stemming from two
ingredients: firstly indeed, the \emph{geometric boundary} of $V$ but, on top
of that, the singular part of the distributional derivative of the tangent
plane function of $V$; in \cite[Example 15]{MR3701146} by the first author,
this is illustrated by a varifold $V$ associated with a properly embedded
submanifold-with-boundary $M$ of class $1$ of $\mathbf R^n$ such that the
support of $\| B ||$ is not contained in $\partial M$.  $\big($A more basic
example with $M$ of class $0$ but not of class $1$ is given by the varifold
associated with the boundary of an $m+1$ dimensional cube, see Footnote
\ref{footnote:convex-geometry} on page
\pageref{footnote:convex-geometry}.$\big)$

\begin{hypotheses} [Density and boundary] \label{hyp:boundary}
	Suppose $V$ and $W$ are $m$ and $m-1$ dimensional varifolds in an open
	subset $U$ of $\mathbf R^n$, respectively, $\| \updelta V \|$ and $\|
	\updelta W \|$ are Radon measures, $\boldsymbol \Uptheta^m ( \| V \|,
	x ) \geq 1$ for $\| V \|$ almost all $x$, $\boldsymbol \Uptheta^{m-1}
	( \| W \|, x) \geq 1$ for $\| W \|$ almost all $x$,
	\begin{gather*}
		\text{$W = 0$ if $m=2$}, \quad \text{$\| \updelta V \| \leq \|
		V \| \restrict | \mathbf h ( V, \cdot ) | + \| W \|$ if
		$m>2$}, \\
		\text{$\| \updelta W \|$ is absolutely continuous with respect
		to $\| W \|$ if $m > 3$}.
	\end{gather*}
\end{hypotheses}

The displayed inequality is equivalent to requiring $\| B \| \leq \| W \|$.
It implies that the geometric boundary of $V$ is contained in $W$ but not
necessarily equal to $W$.  There are good geometric reasons to consider the
stronger condition
\begin{equation*}
	| B ( \theta ) | \leq {\textstyle\int} | S_\natural^\perp ( \theta
	(x)) | \ud W \, (x,S) \quad \text{for $\theta \in \mathscr D ( U,
	\mathbf R^n )$}
\end{equation*}
which may be seen as the boundary part of F.\ Almgren's concept of
\emph{regular pair $(V,W)$} described in \cite[Subsection {4-3}]{MR0190856}.
This stronger condition is also employed by T.\ Ekholm, B.\ White, and D.\
Wienholtz in \cite[Section 7]{MR1888799}.  For our present purposes, the
weaker condition will be sufficient.  Finally, the last condition in the
Hypotheses \ref{hyp:boundary} excludes the presence of boundary for $W$.%
\footnote{\label{footnote:convex-geometry}%
	This condition is natural from the
	differential-geometric point of view.  However,
	following F.\ Almgren's original approach to compactness (see
	\cite[Theorem 10.8]{Almgren:Vari}), one might also study tuples $(V_0,
	\ldots, V_m)$ consisting of $i$ dimensional varifolds $V_i$ such that
	$V_{i-1}$ controls the boundary behaviour of $V_i$ for $i>0$.  This
	would include $m$ dimensional cubes, for instance.}

To formulate the next theorem, we recall that $V \weakD f$ denotes the
generalised $V$ weak derivative of $f$ whenever $f \in \mathbf T ( V )$, see
\cite[Definition 4.2]{MR3777387}, and that $\mathbf T_{\Bdry U} (V)$ denotes
the subclass of those members of $\mathbf T (V)$ which are nonnegative and
have zero boundary values on $\Bdry U$, see \cite[Definition
4.16]{MR3777387}.  It constitutes the varifold formulation of the novel type
of Sobolev-Poincaré inequality formulated in the context of
submanifolds-with-boundary of class $2$ in Corollary
\ref{Final-corollary:new-Sobolev-Poincare} below.

\begin{FinalThm} [see \ref{thm:new_sobolev_inequality}]
	\label{Thm:new-Sobolev-Poincare}

	Suppose $V$ and $W$ satisfy the Hypotheses \ref{hyp:boundary},
	\begin{gather*}
		\mathbf h (V, \cdot ) \in \mathbf L_{m-1}^{\mathrm{loc}} ( \|
		V \|, \mathbf R^n ), \quad \text{if $m>3$
		then $\mathbf h ( W, \cdot ) \in \mathbf L_{m-2}^\mathrm{loc}
		( \| W \|, \mathbf R^n)$}, \\
		f \in \mathbf T_{\Bdry U} ( V ) \cap \mathbf T_{\Bdry U} (W),
		\\
		\text{$| V \weakD f (x) | \leq 1$ for $\| V \|$ almost
		all $x$}, \quad \text{$| W \weakD f (x) | \leq 1$ for $\| W
		\|$ almost all $x$},
	\end{gather*}
	and $E = \{ x \with f(x) > 0 \}$.

	Then, there exists a Borel subset $Y$ of $\mathbf R$ such that
	\begin{equation*}
		f(x) \in Y \quad \text{for $\| V \|$ almost all $x$}
	\end{equation*}
	and such that, for some positive finite number $\Gamma$ determined
	by $m$,
	\begin{enumerate}
		\item if $m = 2$, then $\mathscr L^1 ( Y ) \leq \Gamma \big (
		\| V \| (E)^{1/2} + \| \updelta V \| ( E ) \big )$;
		\item if $m = 3$, then
		\begin{equation*}
			\mathscr L^1 (Y) \leq \Gamma \big ( \| V \| (E)^{1/3}
			+ {\textstyle\int_E} | \mathbf h ( V, \cdot ) |^2 \ud
			\| V \| + \| W \| (E)^{1/2} + \| \updelta W \| (E)
			\big );
		\end{equation*}
		\item if $m > 3$, then
		\begin{multline*}
			\mathscr L^1 ( Y ) \leq \Gamma \big ( \| V \| (
			E)^{1/m} + {\textstyle\int_E} | \mathbf h ( V, \cdot )
			|^{m-1} \ud \| V \|  \\
			+ \| W \| ( E)^{1/(m-1)} + {\textstyle\int_E} |
			\mathbf h (W,\cdot ) |^{m-2} \ud \| W \| \big ).
		\end{multline*}
	\end{enumerate}
\end{FinalThm}

The utility of these estimates stems from \cite[Theorem
B]{arXiv:2209.05955v2}: namely, if $V$ is indecomposable of type $\{ f \}$,
then $\spt f_\# \| V \|$ is an interval and satisfies the bound
\begin{equation*}
	\diam \spt f_\# \| V \| \leq \mathscr L^1 ( Y ).
\end{equation*}
Moreover, if $f$ is continuous, then $f [ \spt \| V \| ] \subset \spt f_\# \|
V \|$, see \cite[7.13\,(2)]{arXiv:2209.05955v2}. The resulting oscillation
estimate, is the key to establish the next two theorems.

\begin{FinalThm} [see
	\ref{thm:few_special_points}\,\eqref{item:few_special_points:p}]
	\label{Thm:varifold-density-lower-bound}

	Suppose $V$ and $W$ satisfy the Hypotheses \ref{hyp:boundary},
	if $m=2$ then $\| \updelta V \|$ is absolutely
	continuous with respect to $\| V \|$,\footnote{More
	generally, absolute continuity may be relaxed to $\| \updelta V \|
	\leq \| V \| \restrict | \mathbf h ( V, \cdot ) | + \| W \|$ and ``$W
	= 0$ if $m = 2$'' may be omitted from the Hypotheses
	\ref{hyp:boundary} for the present theorem.} $\| \updelta W \|$ is
	absolutely continuous with respect to $\| W \|$, $m-1 \leq p < m$,
	$\mathbf h (V, \cdot) \in \mathbf L_p^\mathrm{loc} ( \| V \|, \mathbf
	R^n)$, if $m>2$ then $\mathbf h (W, \cdot ) \in \mathbf
	L_{p-1}^{\mathrm{loc}} ( \| W \|, \mathbf R^n )$, and $V$ is
	indecomposable of type $\mathscr D ( U, \mathbf R )$.

	Then, there holds
	\begin{equation*}
		\text{either $\boldsymbol \Uptheta^m_\ast ( \| V \|, x ) \geq
		1$} \quad \text{or $\boldsymbol \Uptheta^{m-1}_\ast ( \| W \|,
		x ) \geq 1$}
	\end{equation*}
	for $\mathscr H^{m-p}$ almost all $x \in \spt \| V \|$; in particular,
	$\mathscr H^m \restrict \spt \| V \| \leq \| V \|$.
\end{FinalThm}

The result is already significant in case $W = 0$ because it
implies that indecomposability of type $\mathscr D ( U, \mathbf R )$ allows
to discard the alternative $\boldsymbol \Uptheta^m ( \| V \|, x ) = 0 $ from
the afore-mentioned dichotomy. In general, this alternative
may not be omitted as is shown by suitable decomposable varifolds, see
\cite[2.11]{MR2537022}.

\begin{FinalThm} [see \ref{thm:diameter_bound}]
	\label{Thm:varifold-diameter-estimate}

	Suppose $V$ and $W$ satisfy the Hypotheses \ref{hyp:boundary} with $U
	= \mathbf R^n$, $V$ is indecomposable of type $\mathscr D ( \mathbf
	R^n, \mathbf R )$, we have $(\| V \| + \| W \| ) ( \mathbf R^n ) <
	\infty$, and $d$ denotes the geodesic diameter of $\spt \| V \|$.

	Then, for some positive finite number $\Gamma$ determined by $m$,
	there holds
	\begin{enumerate}
		\item if $m = 2$, then $d \leq \Gamma \| \updelta V \| (
		\mathbf R^n )$;
		\item if $m=3$, then $d \leq \Gamma \big ( {\textstyle\int} |
		\mathbf h ( V, \cdot ) |^2 \ud \| V \| + \| \updelta W \|
		( \mathbf R^n ) \big )$; and,
		\item if $m>3$, then $d \leq \Gamma \big ( {\textstyle\int} |
		\mathbf h ( V, \cdot ) |^{m-1} \ud \| V \| + {\textstyle\int}
		| \mathbf h (W, \cdot ) |^{m-2} \ud \| W \| \big )$.
	\end{enumerate}
\end{FinalThm}

In particular, if the sum on the right hand side of the inequality is finite,
then $\spt \| V \|$ is a compact subset of $\mathbf R^n$ and any two points of
$\spt \| V \|$ may be connected by a path of finite length in $\spt \| V \|$.
In analogy with the properties described for the case $p =
m$ of the Hypotheses \ref{hyp:mean_curvature}, an array of further questions
arises.  In the absence of boundary, the most immediate ones read as follows.

\begin{question} \label{question:further_lines_of_study}
	Suppose $V$ satisfies the Hypotheses \ref{hyp:boundary} with $W=0$,
	$V$ is indecomposable of type $\mathscr D ( U, \mathbf R )$, and
	$\mathbf h (V,\cdot) \in \mathbf L_{m-1}^{\mathrm{loc}} ( \| V \|,
	\mathbf R^n)$.\footnote{If $m>2$, the first and last
	condition are equivalent to the Hypotheses \ref{hyp:mean_curvature}
	with $p = m-1$; if $m=2$, they do not require $\| \updelta V \|$ to be
	absolutely continuous with respect to $\| V \|$.}
	\begin{enumerate}
		\item Is $\spt \| V \|$ locally connected?
		\item If so, is $\spt \| V \|$ locally connected by paths of
		finite length?
		\item If so, is the geodesic distance induced on connected
		components of $\spt \| V \|$ a Sobolev function with bounded
		generalised weak derivative and what are the continuity
		properties of this particular (or, any such) function?
	\end{enumerate}
\end{question}

The last item thereof relates to the possible study of
\emph{intermediate conditions on the mean curvature}, that is, to $1 < p < m$
in the Hypotheses \ref{hyp:mean_curvature} (see \cite[p.\,990]{MR3528825}); a
special case of that item was already raised as fifth question in the MSc
thesis of the second author supervised by the first author, see \cite[Section
A]{scharrer:MSc}.

\subsection{The challenges and their resolution}

The starting point of our line of research (see \cite{MR3777387} and the
previous two parts \cite{arXiv:2206.14046v2} and
\cite{arXiv:2209.05955v2} of our series) culminating in the present
final paper was the following a priori bound for geodesic
diameter by P.\ Topping.

\begingroup
	\begin{citing} [\textbf{Theorem} (see \protect{\cite[Theorem
	1.1]{MR2410779}})]
	
		Suppose that $M$ is a compact connected $m$ dimensional
		manifold (without boundary) of class $2$, that $F : M \to
		\mathbf R^n$ is an immersion of class $2$, that $g$ is the
		Riemannian metric on $M$ induced by $F$, and that $\sigma$ is
		the Riemannian distance associated with $(M,g)$.
		
		Then, for some positive finite number $\Gamma$ determined by
		$m$, there holds
		\begin{equation*}
			\diam_\sigma M \leq \Gamma {\textstyle\int_M} |
			\mathbf h ( F, x ) |^{m-1} \ud \mathscr
			H^m_\sigma \, x,
		\end{equation*}
		where the vector $\mathbf h (F, x )$ in $\mathbf R^n$ denotes
		the mean curvature of $F$ at $x$ in $M$.
	\end{citing}
\endgroup

\subsubsection{Aim 1: the generalisation to a varifold-setting}

The varifold-setting is the natural one to model generalised surfaces with
mean curvature.  The corresponding generalisation involved two challenges:
\renewcommand{\theenumi}{\roman{enumi}}
\begin{enumerate}
	\item \label{item:challenges:connectedness} how to rephrase the
	connectedness hypothesis for varifolds; and,
	\item \label{item:challenges:summability} how to handle the low
	summability of the mean curvature.
	\setcounter{enumi_memory}{\value{enumi}}
\end{enumerate}
\renewcommand{\theenumi}{\arabic{enumi}}
Regarding \eqref{item:challenges:connectedness}, the first notion of
connectedness developed for varifolds---termed \emph{indecomposability}---was
available from the theory of generalised weakly differentiable functions on
varifolds, see \cite[Definition 6.2]{MR3528825}.  Proceeding to
\eqref{item:challenges:summability}, we recall that the Hypotheses
\ref{hyp:mean_curvature} with $U = \mathbf R^n$ and $p = m-1$ do not guarantee
the required geometric significance of $\spt \| V \|$.

The key to resolve these two challenges is to treat them simultaneously by
means of the insight that \emph{indecomposability has a strong regularising
effect}; in particular, the known examples exhibiting undesirable behaviour of
$\spt \| V \|$, such as \cite[Example 14.1]{MR3528825} or the earlier one of a
similar structure in \cite[Example 1.2]{MR2537022}, are decomposable.  The
analysis carried out in \cite{scharrer:MSc}, not only yields
a complete generalisation of the above estimate of the geodesic diameter to
the varifold setting (i.e., Theorem \ref{Thm:varifold-diameter-estimate} with
$W = 0$) but it also yields a lower density bound for the varifolds involved
(i.e., Theorem \ref{Thm:varifold-density-lower-bound} with $W = 0$ and
$p=m-1$) ensuring the geometric significance of $\spt \| V \|$ via
\begin{equation*}
	\boldsymbol \Uptheta_\ast^m ( \| V \|, x ) \geq 1 \quad \text{for
	$\mathscr H^1$ almost all $x \in \spt \| V \|$}.
\end{equation*}
We note that both theorems describe a one-dimensional property of $\spt \| V
\|$: the length of geodesics therein and a lower density bound $\mathscr H^1$
almost everywhere.

\subsubsection{Aim 2: the treatment of surfaces with boundary}

To allow for boundary is evidently a prerequisite for applications to
geometric variational problems such as the Plateau problem.  An initial step
had been made by S.-H.\ Paeng: If $m=2$ and $M$ is a manifold-with-boundary,
then the estimate in the preceding theorem may be replaced by
\begin{equation*} \diam_\sigma M \leq \Gamma \big ( {\textstyle\int_M} |
\mathbf h(F,x)| \ud \mathscr H^2_\sigma \, x + \mathscr H^1_\sigma ( \partial
M ) \big ), \end{equation*} \emph{provided $(M,g)$ is convex}, see
\cite[Theorem 2\,(a)]{MR3183369}.  Therefore, to achieve the second aim, we
had to resolve the following three additional challenges:
\renewcommand{\theenumi}{\roman{enumi}}
\begin{enumerate}
	\setcounter{enumi}{\value{enumi_memory}}
	\item \label{item:challenges:convexity} for $m=2$, how to remove 
	the convexity hypothesis;
	\item \label{item:challenges:smooth-boundary} for $m>2$, how to take
	the geometry of the boundary into account; and,
	\item \label{item:challenges:varifold-boundary} how to phrase the
	boundary condition in a varifold-setting.
	\setcounter{enumi_memory}{\value{enumi}}
\end{enumerate}
\renewcommand{\theenumi}{\arabic{enumi}}
In \cite{MR3183369}, the convexity hypothesis is mainly used to ensure that
interior points of length-minimising geodesics cannot meet $\partial M$.
Viewing this as regularity consideration, a method sufficiently robust for the
varifold-setting is likely to accommodate \eqref{item:challenges:convexity} as
well.  For \eqref{item:challenges:smooth-boundary}, the $m-1$ dimensional
Hausdorff measure of $\partial M$ (raised to the appropriate power) does not
yield a valid estimate, see \ref{remark:zemas}.  Instead, the mean curvature
of $F | \partial M$ turns out to be an adequate choice; in particular, the
boundary is naturally represented by an $m-1$ dimensional varifold.  Regarding
\eqref{item:challenges:varifold-boundary}, we recall that the difficulty stems
from the lack of a boundary operator for varifolds and is resolved by
employing the conditions described in the Hypotheses \ref{hyp:boundary} above
as a substitute.

The second aim was then achieved in the first version of the present
publication (see
\href{https://arxiv.org/abs/1709.05504v1}{\path{arXiv:1709.05504v1}}) which
contained two additional insights with respect to \cite{scharrer:MSc}:
Firstly, it introduced the notion of indecomposability of type $\mathscr D (
U, \mathbf R )$.  This connectedness property is strictly weaker than
indecomposability but yet strong enough for the deduction of the intended
geometric consequences; in fact, this weakening of our hypotheses has later
turned out to be crucial for the applicability of our theory to geometric
variational problems (see Aim 3 below).  The second insight was the novel type
of Sobolev-Poincaré inequality formulated in Theorem
\ref{Thm:new-Sobolev-Poincare} which is foundational for both the
varifold-results on density and those on geodesic diameter in Theorems
\ref{Thm:varifold-density-lower-bound} and
\ref{Thm:varifold-diameter-estimate}.  To discuss the nature of Theorem
\ref{Thm:new-Sobolev-Poincare}, we shall now describe the corollary resulting
from it in the special case of properly embedded
\emph{connected} submanifolds.

\begin{Final-corollary} [Novel type of Sobolev-Poincaré inequality, see
	\ref{corollary:new_sobolev_inequality}]
	\label{Final-corollary:new-Sobolev-Poincare}

	Suppose $U$ is an open subset of $\mathbf R^n$, $M$ is a properly
	embedded, connected $m$ dimensional submanifold-with-boundary of class
	$2$ of $U$, $f : M \to \mathbf R$ is a function of class $1$ relative
	to $M$, $\spt f$ is compact, and
	\begin{equation*}
		E = M \cap \{ x \with f(x) \neq 0 \}, \quad \kappa = \sup \{ |
		\Der f (x) | \with x \in M \without \partial M \}.
	\end{equation*}

	Then, for some positive finite number $\Gamma$ determined by $m$,
	there holds
	\begin{multline*}
		\diam f[M] \leq \Gamma \big ( \mathscr H^m ( E \cap M)^{1/m} +
		{\textstyle\int_{E \cap M}} | \mathbf h ( M, x ) |^{m-1} \ud
		\mathscr H^m \, x \\
		+ \mathscr H^{m-1} (E \cap \partial M)^{1/(m-1)} +
		{\textstyle\int_{E \cap \partial M}} | \mathbf h ( \partial
		M, x) |^{m-2} \ud \mathscr H^{m-1} \, x \big) \kappa;
	\end{multline*}
	here, the summand $\int_{E \cap \partial M} | \mathbf h ( \partial M,
	x ) |^{m-2} \ud \mathscr H^{m-1} \, x$ shall be omitted if $m=2$.
\end{Final-corollary}

In case $M$ is compact, $\spt f$ is automatically compact and, applying the
isoperimetric inequality to $M$ and $\partial M$ yields the following
\emph{Poincaré inequality} (see \ref{corollary:poincare-inequality}) with a
positive finite number $\Delta$ determined by $m$:
\begin{equation*}
	\diam f[M] \leq \Delta \big ( {\textstyle\int_M} | \mathbf h ( M,x )
	|^{m-1} \ud \mathscr H^m \, x + {\textstyle\int_{\partial M}} |
	\mathbf h ( \partial M, x ) |^{m-2} \ud \mathscr H^{m-1} \, x \big )
	\kappa;
\end{equation*}
regarding $m=2$, we stipulate $\int_{\partial M} | \mathbf h ( \partial M, x)
|^0 \ud \mathscr H^1 \, x = \mathscr H^1 ( \partial M )$ by convention.
Earlier estimates of the essential oscillation of $f$, see \cite[Theorems
10.1\,(1d), 10.7\,(4), and 10.9\,(4)]{MR3528825}, differ in several aspects:
Firstly, they are limited to the case $\partial M = \varnothing$ but do not
require connectedness of $M$.  Secondly, they are local in nature rather than
global.  Thirdly, they involve the $m$-th power of the mean curvature (in
fact, an integral smallness condition thereon) instead of the power $m-1$.
Finally, they allow for $q$-th power integrals of $\Der f$ for some $m < q <
\infty$ which in fact turns out to be impossible in our setting, see
\ref{example:cylinder}.  The preceding Poincaré inequality is readily seen to
be equivalent to an a priori bound for the geodesic diameter of $M$.
Applying differential-topologic density results (to remove the embeddedness
hypothesis) then yields the following extension of P.\ Topping's result to
immersions of manifolds-with-boundary which corresponds to Theorem
\ref{Thm:varifold-diameter-estimate} in the varifold-setting.

\begin{Final-corollary} [Geodesic diameter bound for immersions, see
	\ref{corollary:smooth_diameter_bound}]
	\label{Final-corollary:diameter-bound-immersions}
	
	Suppose $M$ is a compact connected $m$ dimensional
	manifold-with-boundary of class $2$, $F : M \to \mathbf R^n$ is an
	immersion of class $2$, $g$ is the Riemannian metric on $M$ induced by
	$F$, and $\sigma$ is the Riemannian distance associated with $(M,g)$.
	
	Then, for some positive finite number $\Gamma$ determined by $m$,
	there holds
	\begin{equation*}
		\diam_\sigma M \leq \Gamma \big ( {\textstyle\int_M} | \mathbf
		h ( F, x ) |^{m-1} \ud \mathscr H^m_\sigma \, x +
		{\textstyle\int_{\partial M}} | \mathbf h ( F | \partial M, x)
		|^{m-2} \ud \mathscr H^{m-1}_\sigma \, x \big );
	\end{equation*}
	regarding $m=2$, we stipulate $\int_{\partial M} | \mathbf h (\partial
	M, x )|^0 \ud \mathscr H^1_\sigma \, x = \mathscr H^1_\sigma (
	\partial M )$.
\end{Final-corollary}

For $m \geq 3$, this leads to the following question which is open even if
$\mathbf h ( F, \cdot ) = 0$ and $F$ is an embedding.  If the answer were in
the affirmative, then Corollary
\ref{Final-corollary:diameter-bound-immersions} would be a consequence of the
resulting statement applied to both $F$ and $F | \partial M$.

\begin{question}
	May the summand $\int_{\partial M} | \mathbf h ( F, x ) |^{m-2} \ud
	\mathscr H^{m-1}_\sigma \, x$ in Corollary
	\ref{Final-corollary:diameter-bound-immersions} be replaced by the sum
	of the geodesic diameters of the connected components of $\partial M$
	computed with respect to the induced Riemannian distance on $\partial
	M$?
\end{question}

\subsubsection{Interlude: lower density ratio bounds}

The insights discussed so far allow us to deduce lower bounds on the density
based on the connectedness hypothesis.  To illustrate this, we state here the
underlying conditional lower density ratio bound in the case
of properly embedded submanifolds regarding the varifold-result on the density
in Theorem \ref{Thm:varifold-density-lower-bound}.

\begin{Final-corollary} [Conditional lower density ratio bound, see
	\ref{lemma:lower_density_ratio_bound}]
	\label{Final-corollary:lower-density-bound}

	Suppose $U$ is an open subset of $\mathbf R^n$, $M$ is a properly
	embedded, connected $m$ dimensional submanifold-with-boundary of $U$
	of class $2$, $a \in M$, $0 < r < \infty$, $\mathbf B (a,r) \subset
	U$, and $M \without \mathbf U (a,r) \neq \varnothing$.

	Then, for some positive finite number $\Gamma$ determined by $m$,
	there holds
	\begin{multline*}
		\Gamma^{-1} r \leq \mathscr H^m ( \mathbf U (a,r) \cap
		M)^{1/m} + {\textstyle\int_{\mathbf U (a,r) \cap M}} | \mathbf
		h (M, \cdot )|^{m-1} \ud \mathscr H^m \\
		+ \mathscr H^{m-1} ( \mathbf U (a,r) \cap \partial
		M)^{1/(m-1)} + {\textstyle\int_{\mathbf U(a,r) \cap \partial
		M}} | \mathbf h ( \partial M, \cdot ) |^{m-2} \ud \mathscr
		H^{m-1};
	\end{multline*}
	here, the summand $\int_{\mathbf U (a,r) \cap \partial M} | \mathbf h
	( \partial M, \cdot ) |^{m-2} \ud \mathscr H^{m-1}$ shall be omitted
	if $m=2$.
\end{Final-corollary}

This is a consequence of Corollary \ref{Final-corollary:new-Sobolev-Poincare}
generalised to Lipschitzian functions and applied with $f(x) = \sup \{
r-|x-a|,0 \}$ for $x \in U$.  Since $M \without \mathbf U (a,r) \neq
\varnothing$, our novel type of Sobolev-Poincaré inequality for this $f$ takes
the flavour of a \emph{Sobolev inequality} because $\sup \im f = \diam f[M]$.
Small spheres show that the hypothesis $M \without \mathbf U (a,r) \neq
\varnothing$ cannot be omitted.  Therefore, whereas $p = m$ is the critical
exponent in general, \emph{the critical exponent for the summability of the
mean curvature vector in the connected case is $p = m-1$}.  This is the key to
the regularising effects of indecomposability.  For further illustration, we
consider the following class of submanifolds (not necessarily properly
embedded).

\begin{hypotheses} \label{hyp:connected-submanifold}
	Suppose $1 \leq p < \infty$, $M$ is a connected $m$ dimensional
	submanifold of class $\infty$ of $\mathbf R^n$ which meets every
	compact subset of $\mathbf R^n$ in a set of finite $\mathscr H^m$
	measure such that its second fundamental form $\mathbf b ( M, \cdot )$
	satisfies $\int_M \| \mathbf b (M,x) \|^p \ud \mathscr H^m \, x <
	\infty$ and the divergence theorem holds in the sense that
	\begin{equation*}
		{\textstyle \int_M \Tan (M,x)_\natural \bullet \Der \theta (x)
		\ud \mathscr H^m \, x = - \int_M \mathbf h (M,x) \bullet
		\theta (x) \ud \mathscr H^m \, x}
	\end{equation*}
	for $\theta \in \mathscr D ( \mathbf R^n, \mathbf R^n )$, and $A =
	( \Clos M ) \cap \{ a \with \boldsymbol \Uptheta^m ( \mathscr H^m
	\restrict M, a ) = 0 \}$.
\end{hypotheses}

In view of
\cite[3.3, 6.1]{arXiv:2209.05955v2}, Theorem
\ref{Thm:varifold-density-lower-bound} guarantees that
\begin{equation*}
	\mathscr H^{m-p} ( A) = 0  \quad \text{in case $m-1 \leq p < m$}.
\end{equation*}
The next example shows not only that this bound is sharp but also that there
is no corresponding result in the range $1 \leq p < m-1$.  \emph{This exhibits
a discontinuity in $p$ regarding the optimal upper bound on the Hausdorff
dimension of $A$.}

\begingroup \hypertarget{Example}{}
\begin{citing} [\textbf{Example} (Sharpness of the lower density bounds, see
	\ref{thm:optimality_exceptional_set} and \ref{thm:example_p_small})]

	Suppose $m < n$.  Then, the following two statements hold.
	\begin{enumerate}
		\item If $m-1 < q < m$, then there exists $M$
		satisfying the Hypotheses \ref{hyp:connected-submanifold} for
		$1 \leq p < q$ such that $\mathscr H^{m-q} ( A ) > 0$ for the
		associated set $A$.
		\item If $m \geq 3$, then there exists $M$ satisfying
		the Hypotheses \ref{hyp:connected-submanifold} for $1 \leq p <
		m-1$ such that $\mathscr H^m ( A ) > 0$ for the associated set
		$A$.
	\end{enumerate}
\end{citing}
\endgroup

\subsubsection{Aim 3: applicability to geometric variational problems}

For this purpose, we were guided by the following two model cases:
\renewcommand{\theenumi}{\Alph{enumi}}
\begin{enumerate}
	\item \label{item:Plateau:chains} The Plateau problem for integral
	currents (more generally, $G$ chains).
	\item \label{item:Plateau:sets} The Plateau problem for sets using
	Čech homology.
\end{enumerate}
\renewcommand{\theenumi}{\arabic{enumi}}
The solution to each of these problems gives rise to an associated
varifold $V$ which is stationary away from the boundary.  This entails the
challenge:
\renewcommand{\theenumi}{\roman{enumi}}
\begin{enumerate}
	\setcounter{enumi}{\value{enumi_memory}}
	\item \label{item:challenges:indecomposability} Does the natural
	connectedness property of the solution ($G$ chain or set) entail a
	suitable indecomposability property of the associated varifold?
	\setcounter{enumi_memory}{\value{enumi}}
\end{enumerate}
Introducing indecomposability of type $\mathscr D ( U, \mathbf R )$,
\eqref{item:challenges:indecomposability} has been resolved in
\cite{arXiv:2209.05955v2}, see Theorem G therein for Case
\eqref{item:Plateau:chains} and Corollary 2 therein for Case
\eqref{item:Plateau:sets}; in particular, for Case \eqref{item:Plateau:sets},
it suffices to verify connectedness of $\spt \| V \|$.  On the other hand, for
Case \eqref{item:Plateau:sets}, two considerations which do not pose a
particular difficulty in Case \eqref{item:Plateau:chains}, see
\ref{lemma:first-variation-minimiser} and
\ref{remark:diameter_bound_mass_minimising_currents}, add two final challenges
to our list:
\begin{enumerate}
	\setcounter{enumi}{\value{enumi_memory}}
	\item \label{item:challenges:estimate} Is there an a
	priori estimate for $\| \updelta V \|$ in terms of the boundary data?
	\item \label{item:challenges:solution} Does connectedness of the
	boundary imply connectedness of $\spt \| V \|$?
\end{enumerate}
\renewcommand{\theenumi}{\arabic{enumi}}
For these challenges in the present Case \eqref{item:Plateau:sets}, a key
ingredient is provided by C.\ Labourie with \cite[Lemmata 1.2.2 and
2.2.1]{MR4489608} which ensure
\begin{equation*}
	B \subset \spt \| V \|,
\end{equation*}
see \ref{remark:Reifenberg-Plateau}.  The treatment of
\eqref{item:challenges:estimate} then uses the following theorem which rests
on a refinement of W.\ Allard's estimates regarding boundary behaviour in
\cite{MR0397520}.  It employs H.\ Federer's concept of reach from
\cite[4.1]{MR0110078}.
\begin{FinalThm} [see \ref{thm:estimate_variation_boundary}]
	\label{FinalThm:geom-var-I}

	Suppose $R = \reach (B)$, $V$ is an $m$ dimensional varifold in
	$\mathbf R^n$, $\| V \| ( B ) = 0$, $\spt \updelta V \subset B \subset
	\spt \| V \|$, $\boldsymbol \Uptheta^m ( \| V \|, x ) \geq 1$ for $\|
	V \|$ almost all $x$, and
	\begin{equation*}
		M = R^{-m} \sup \{ \| V \| \, \mathbf U (b,R/2) \with b \in B
		\}.
	\end{equation*}

	Then, for some positive finite number $\Gamma$ determined by $m$,
	there holds
	\begin{equation*}
		\| \updelta V \| \leq \Gamma M \mathscr H^{m-1} \restrict B.
	\end{equation*}
\end{FinalThm}
For Case \eqref{item:Plateau:sets}, the $\| V \|$ measure of
$\mathbf R^n$, and thus the number $M$, can readily be estimated in terms of
$B$ and $m$, see \ref{remark:Reifenberg-Plateau}.  Regarding
\eqref{item:challenges:solution}, we can rely on the study of connected
components of $\spt \| V \|$ from \cite[Corollary 6.14]{MR3528825} in
combination with the isoperimetric inequality to deduce connectedness of $\spt
\| V \|$ from that of its subset $B$, see \ref{lemma:plateau_connectedness}.
In combination with Theorem \ref{Thm:varifold-diameter-estimate}, this yields
the following theorem which is also applicable to varifolds constructed by
min-max methods or Brakke flows, see \ref{remark:min-max} and
\ref{remark:Brakke-flow-II}.

\begin{FinalThm} [see \ref{thm:Reifenberg-Plateau-diameter}]
	\label{FinalThm:geom-var-II}

	Suppose $B$ is connected, $V$ is an $m$ dimensional varifold in
	$\mathbf R^n$, $\| V \| ( \mathbf R^n ) < \infty$, $1 \leq \lambda <
	\infty$,
	\begin{equation*}
		B \subset \spt \| V \|, \quad \| \updelta V \| \leq \lambda
		\mathscr H^{m-1} \restrict B,
	\end{equation*}
	$\boldsymbol \Uptheta^m ( \| V \|, x ) \geq 1$ for $\| V \|$ almost
	all $x$, and $d$ is the geodesic diameter of $\spt \| V \|$.

	Then, there holds
	\begin{equation*}
		d \leq \Gamma_{\textup{\ref{thm:diameter_bound}}} (m) \lambda
		{\textstyle\int_B} | \mathbf h ( B, b ) |^{m-2} \ud \mathscr
		H^{m-1} \, b;
	\end{equation*}
	here, by convention, we stipulate $\int_B | \mathbf h (B,b) |^0 \ud
	\mathscr H^1 \, b = \mathscr H^1 (B)$ regarding $m=2$.
\end{FinalThm}

\subsection{Acknowledgements}

The authors thank Theodora Bourni, Guy David, Camillo De Lellis, Michael
Eichmair, Robert Haslhofer, Camille Labourie, Sławomir Kolasiński, Melanie
Rupf{}lin, Nicolau Sarquis Aiex, Richard Schoen, Felix Schulze, Yoshihiro
Tonegawa, Peter Topping, and Konstantinos Zemas for discussions related to the
present development over the years.  Apart of an extended research visit of
the first author at the University of Zurich for which he is grateful to
Camillo De Lellis, the initial version of this paper (see
\href{https://arxiv.org/abs/1709.05504v1}{\path{arXiv:1709.05504v1}}) was
carried out while the authors were affiliated to both, the Max Planck
Institute for Gravitational Physics (Albert Einstein Institute) and the
University of Potsdam.  During subsequent revisions the first author was
affiliated first with both, the Universities of Leipzig and the Max Planck
Institute for Mathematics in the Sciences, before---at the present affiliation
in Taiwan (R.\ O.\ C.)---he was supported by the grants with nos.\
108-2115-M-003-016-MY3, MOST 110-2115-M-003-017, MOST 111-2115-M-003-014, NSTC
112-2115-M-003-001, and NSTC 113-2918-I-003-002, by the National Science and
Technology Council (formerly termed Ministry of Science and Technology) and as
Center Scientist by the National Center for Theoretical Sciences, whereas the
second author was affiliated with the University of Warwick---supported by the
EPSRC as part of the MASDOC DTC with Grant No.\ EP/HO23364/1---, the Max
Planck Institute for Mathematics, and the present institution. The final touch
was made during an extended research visit of the first author at the
University of Cambridge for which he is grateful to Neshan Wickramasekera.

\section{Notation} \label{sec:notation}

\paragraph{Basic sources} As in Parts I and II (see
\cite{arXiv:2206.14046v2,arXiv:2209.05955v2}), our notation follows
\cite{MR3528825} and is thus largely consistent with H.\ Federer's terminology
for geometric measure theory listed in \cite[pp.\,669--676]{MR41:1976} and W.\
Allard's notation for varifolds introduced in \cite{MR0307015}.  This
includes, for distributions $S$, the variation measure $\| S \|$, see
\cite[2.18]{MR3528825}; for certain varifolds $V$ and sets $E$, the notion of
\emph{distributional boundary}, $V \boundary E$, of $E$ with respect to $V$,
see \cite[5.1]{MR3528825}; for certain varifolds $V$, concepts relating to
generalised $V$ weak differentiability---that is, the space $\mathbf T (V)$ of
generalised $V$ weakly differentiable real valued functions $f$, the
\emph{generalised $V$ weak derivative}, $V \weakD f$, for such $f$, and the
subspace $\mathbf T_G (V)$ of \emph{nonnegative} members of $\mathbf T (V)$
realising the concept of \emph{zero boundary values} on an open subset $G$ of
$\Bdry U$---, see \cite[8.3, 9.1]{MR3528825}.

\paragraph{Review}  Here, we list symbols not already reviewed in
\cite[Introduction, Section 1]{MR3528825}: the Grassmann manifold, $\mathbf
G(n,m)$, of all $m$ dimensional subspaces of $\mathbf R^n$, see
\cite[1.6.2]{MR41:1976}; the norms associated with inner products, $| \cdot
|$, see \cite[1.7.1]{MR41:1976};  the seminorm, $\| \cdot \|$, on $\Hom (V,W)$
associated with normed spaces $V$ and $W$, see \cite[1.10.5]{MR41:1976}; the
infimum and supremum, $\inf S$ and $\sup S$, of a subset $S$ of the extended
real numbers, see \cite[2.1.1]{MR41:1976}; the class $\mathbf 2^X$ of all
subsets of $X$, see \cite[2.1.2]{MR41:1976}; the restriction, $\phi \restrict
A$, of a measure $\phi$ to a set $A$, see \cite[2.1.2]{MR41:1976}; the
support, abbreviated $\spt \phi$, of a measure $\phi$, see
\cite[2.2.1]{MR41:1976}; the least Lipschitz constant, $\Lip f$, of a map $f$
between metric spaces, see \cite[2.2.7]{MR41:1976}; the Lebesgue spaces,
$\mathbf L_p ( \phi, Y )$, see \cite[2.4.12]{MR41:1976}; the support, denoted
$\spt f$, of a member $f$ of $\mathscr K (X)$, see \cite[2.5.13]{MR41:1976};
the variation, $\mathbf V_a^b g$, of $g$ from $a$ to $b$ for maps of the real
line into a complete metric space, see \cite[2.5.16]{MR41:1976}; the $n$
dimensional Lebesgue measure, $\mathscr L^n$, see \cite[2.6.5]{MR41:1976}; the
diameter of $S$, $\diam S$, see \cite[2.8.8]{MR41:1976}; the absolutely
continuous part $\psi_\phi$ of a measure $\psi$ with respect to $\phi$, see
\cite[2.9.1]{MR41:1976};%
\begin{footnote}%
	{As was stipulated in \cite{MR3528825} for similar notions,
	$\psi_\phi$ will also be employed in the case where one of the
	measures $\psi$ or $\phi$ on $X$ fails to be finite on bounded sets
	but there exist open sets $U_1, U_2, U_3, \ldots$ covering $X$ at
	which $\phi$ and $\psi$ are finite.}
\end{footnote}%
the derivative, $g'$, of a function on the real line, see
\cite[2.9.19]{MR41:1976}; the number $\boldsymbol \upalpha (m)$ and the $m$
dimensional Hausdorff measure, $\mathscr H^m$, for $0 \leq m < \infty$, see
\cite[2.10.2]{MR41:1976}; the $m$ dimensional densities, $\boldsymbol
\Uptheta^{\ast m} ( \phi, a)$, $\boldsymbol \Uptheta_\ast^m ( \phi, a )$, and
$\boldsymbol \Uptheta^m (\phi, a )$, see \cite[2.10.19]{MR41:1976}; the
differential, $\Der f$, see \cite[3.1.1]{MR41:1976}; the closed cones of
tangent and normal vectors, $\Tan (S,a)$ and $\Nor (S,a)$, see
\cite[3.1.21]{MR41:1976}; the unit sphere, $\mathbf S^{n-1}$, in $\mathbf
R^n$, see \cite[3.2.13]{MR41:1976}; the vector space, $\mathscr E (U,Y)$, of
functions of class $\infty$ and the support, $\spt T$, of a distribution $T$
in $U$ of type $Y$, see \cite[4.1.1]{MR41:1976}; the chain complex of
\emph{integral currents} in $\mathbf R^n$ with $m$-th chain group $\mathbf I_m
( \mathbf R^n )$ and boundary operator $\boundary$, see \cite[4.1.7,
4.1.24]{MR41:1976}; the member $\boldsymbol [ u,v \boldsymbol ]$ of $\mathbf
I_1 ( \mathbf R^n )$ associated with the line segment from $u$ to $v$, see
\cite[4.1.8]{MR41:1976}; the weight, $\| V \|$, of a varifold $V$, see
\cite[3.1]{MR0307015}; the image, $f_\# V$, of a varifold $V$ under a map $f$
of class $\infty$, see \cite[3.2]{MR0307015}; and, the first variation,
$\updelta V$, of a varifold $V$, see \cite[4.2]{MR0307015}.

\paragraph{Modification}  For the push forward, $f_\# \phi$, of a measure
$\phi$ by a function $f$, we use the definition in
\cite[3.9]{arXiv:2206.14046v2} which extends \cite[2.1.2]{MR41:1976}.

\paragraph{Amendments}

Modelled on \cite[4.1.7]{MR41:1976}, we employ the restriction notation, $\phi
\restrict f$, for the measure $\phi$ weighted by $f$ introduced in
\cite[3.6]{arXiv:2206.14046v2}; this weighted measure was discussed---without
name---in \cite[2.4.10]{MR41:1976}.  For subsets $A$ of Euclidean space, we
adopt the concept of reach and the symbol $\reach (A)$ from
\cite[4.1]{MR0110078} as well as those of \emph{approximate differentiability
of order $2$} with the corresponding \emph{approximate mean curvature},
denoted by $\ap \mathbf h (A, \cdot )$, from \cite[3.8]{MR3978264} and
\cite[6.9]{arXiv:2209.05955v2}.  For certain immersions $F$ into an open
subset~$U$ of $\mathbf R^n$, we use the concepts of \emph{mean curvature
vector}, denoted $\mathbf h (F,\cdot)$, of \emph{varifold associated with
$(F,U)$}, and of \emph{Riemannian distance} associated with $F$ as laid down
in \cite[6.10, 6.13, 10.1]{arXiv:2209.05955v2}, respectively.
The terms \emph{immersion} and \emph{embedding} are employed in accordance
with \cite[p.\,21]{MR1336822}.  Whenever $k$ is a positive integer or $k =
\infty$, we mean by a [\emph{sub}]\emph{manifold-with-boundary of class $k$} a
Hausdorff topological space with a countable base of its topology that is, in
the terminology of \cite[pp.\,29--30]{MR1336822}, a $C^k$ [sub]manifold.  For
manifolds-with-boundary $M$ of class $k$, we similarly adapt the notion of
\emph{chart of class $k$} and \emph{Riemannian metric of class $k-1$} from
\cite[p.\,29, p.\,95]{MR1336822} and denote by $\partial M$ its
\emph{boundary} as in \cite[p.\,30]{MR1336822}.  Whenever $G$ is a complete
normed commutative group as defined in \cite[3.1]{arXiv:2206.14046v2}, we
employ the following notation regarding the group $\mathscr R_m^{\textup{loc}}
( \mathbf R^n, G )$ of $m$ dimensional \emph{locally rectifiable $G$ chains
$S$} in $\mathbf R^n$, see \cite[4.5]{arXiv:2206.14046v2}: the notion of
weight measure, $\| S \|$, see \cite[4.5]{arXiv:2206.14046v2}; the
homomorphisms
\begin{equation*}
	f_\# : \mathscr R_m^{\textup{loc}} ( \mathbf R^n, G ) \to \mathscr
	R_m^{\textup{loc}} ( \mathbf R^\nu, G )
\end{equation*}
associated with locally Lipschitzian maps $f : \mathbf R^n \to \mathbf
R^\nu$, see \cite[4.6]{arXiv:2206.14046v2}; the Cartesian product
\begin{equation*}
	\times : \mathscr R_m^{\textup{loc}} ( \mathbf R^n, \mathbf Z ) \times
	\mathscr R_\mu^{\textup{loc}} ( \mathbf R^\nu, G ) \to \mathscr
	R_{m+\mu}^{\textup{loc}} ( \mathbf R^n \times \mathbf R^\nu, G ),
\end{equation*}
see \cite[4.7]{arXiv:2206.14046v2}; the isomorphism $\iota_{\mathbf R^n,m} :
\mathscr R_m^{\textup{loc}} (\mathbf R^n) \to \mathscr R_m^{\textup{loc}} (
\mathbf R^n, \mathbf Z )$, see \cite[5.1]{arXiv:2206.14046v2}; the chain
complex of \emph{integral $G$ chains} in $\mathbf R^n$ with $m$-th chain group
$\mathbf I_m ( \mathbf R^n, G)$---identified with a subgroup of $\mathscr
R_m^{\textup{loc}} ( \mathbf R^n, G)$---and boundary operator $\boundary_G$,
see \cite[5.11, 5.13, 5.17]{arXiv:2206.14046v2}; and,
\emph{indecomposability} of members of $\mathbf I_m ( \mathbf R^n, G )$, see
\cite[6.6]{arXiv:2209.05955v2}.  Finally, for certain varifolds, we make use
of the concept of \emph{indecomposability of type $\Psi$} introduced in
\cite[7.1]{arXiv:2209.05955v2}.

\paragraph{Definitions in the text}

Following \cite[6.6]{MR3626845}, the terms \emph{geodesic distance} and
\emph{geodesic diameter} are laid down in \ref{def:geodesic_distance}.  The
locally convex space $\mathscr C^k (M,Y)$ is defined in \ref{def:ck_space}.

\section{The special case of extrinsic diameter and no
boundary}

To highlight some of the principal ideas, we provide the short proof of a
special case of Theorem \ref{Thm:varifold-diameter-estimate} in this section;
the general case will be treated in \ref{thm:diameter_bound}.

\begin{theorem} \label{thm:topping}
	Suppose $m$ and $n$ are integers, $2 \leq m \leq n$, $V \in \mathbf
	V_m ( \mathbf R^n )$, $\spt \| V \|$ is compact, $\| \updelta V \|$ is
	a Radon measure, $\boldsymbol \Uptheta^m ( \| V \|, x ) \geq 1$ for
	$\| V \|$ almost all $x$, $V$ is indecomposable, and
	\begin{enumerate}
		\item if $m = 2$, then $\psi = \| \updelta V \|$, and
		\item if $m > 2$, then $\| \updelta V \|$ is absolutely
		continuous with respect to $\| V \|$ and $\psi = \| V \|
		\restrict | \mathbf h (V, \cdot ) |^{m-1}$.
	\end{enumerate}

	Then, there holds
	\begin{equation*}
		\diam \spt \| V \| \leq \Gamma \psi ( \mathbf R^n ),
	\end{equation*}
	where $\Gamma$ is a positive, finite number determined by $m$.
\end{theorem}

\begin{proof}
	By \cite[3.5\,(1b), 5.5\,(1)]{MR0307015}, $V$ is rectifiable and $\| V
	\| = \mathscr H^m \restrict \boldsymbol \Uptheta^m ( \| V \|, \cdot
	)$.

	Assume $\psi ( \mathbf R^n ) < \infty$. Let $A$ denote the set of all
	$a \in \spt \| V \|$ such that
	\begin{equation*}
		\limsup_{s \to 0+} \| V \| ( \mathbf B (a,s) )^{(1/m) - 1} \|
		\updelta V \| \, \mathbf B (a,s) < ( 2 \boldsymbol \upgamma
		(m) )^{-1}.
	\end{equation*}
	Note $\| V \| ( \mathbf R^n \without A ) = 0$ by \cite[2.8.18,
	2.9.5]{MR41:1976}, hence $A$ is dense in $\spt \| V \|$ and
	\begin{equation*}
		\diam \spt \| V \| = \sup \{ \diam p [ A ] \with p \in \mathbf
		O^\ast ( n , 1 ) \}.
	\end{equation*}

	Next, suppose $p \in \mathbf O^\ast ( n,1 )$ and let $\phi = p_\#
	\psi$. Then, for each $b \in p [ A ]$, there exists $0 < r < \infty$
	such that
	\begin{equation*}
		r \leq \Delta \phi \, \mathbf B (b,r),
	\end{equation*}
	where $\Delta = m ( 2 \boldsymbol \upgamma (m) )^m$; in fact, one
	may choose $a \in A$ with $p(a) = b$ and take
	\begin{equation*}
		r = \inf \big \{ s \with \| \updelta V \| \, \mathbf B (a,s) >
		( 2 \boldsymbol \upgamma (m)  )^{-1} \| V \| ( \mathbf B (a,s)
		)^{1-1/m} \big \},
	\end{equation*}
	hence $0 < r < \infty$ and $\| V \| \, \mathbf B (a,r) \geq ( 2 m
	\boldsymbol \upgamma (m) )^{-m} r^m$ by \cite[2.5]{MR2537022} implying
	\begin{gather*}
		\begin{aligned}
			( 2 \boldsymbol \upgamma (m) )^{-1} \| V \| (
			\mathbf B (a,r) )^{1-1/m} & \leq \| \updelta V \| \,
			\mathbf B (a,r) \\
			& \leq \| V \| ( \mathbf B (a,r) )^{1-1/(m-1)} \psi (
			\mathbf B (a,r) )^{1/(m-1)},
		\end{aligned} \\
		( 2 m \boldsymbol \upgamma (m)  )^{-1} r \leq \| V \| (
		\mathbf B (a,r) )^{1/m} \leq ( 2 \boldsymbol \upgamma (m)
		)^{m-1} \psi \, \mathbf B (a,r) .
	\end{gather*}
	Consequently,
	\begin{equation*}
		\mathscr L^1 ( p [ A ] ) \leq 2 \Delta \boldsymbol \upbeta (1)
		\psi ( \mathbf R^n ).
	\end{equation*}
	
	The proof will be concluded by showing
	\begin{gather*}
		\diam p [ A ] = \mathscr L^1 ( p [ A ] ).
	\end{gather*}
	If this were not the case, there would exist $b \in \mathbf R$ such
	that
	\begin{gather*}
		\inf p [ A ] < b < \sup p [ A ] , \quad \boldsymbol \Uptheta^1
		( p_\# \| V \|, b ) = 0
	\end{gather*}
	by \cite[2.2.17, 2.10.19\,(4)]{MR41:1976} since $(p_\# \| V \|) (
	\mathbf R \without p [ A ] ) = 0$, hence one could use
	\cite[4.10\,(1)]{MR0307015} or \cite[8.7, 8.29]{MR3528825} to infer
	\begin{gather*}
		\updelta ( V \restrict \{ (a,S) \with p(a) > b \} ) = (
		\updelta V ) \restrict \{ a \with p (a) > b \}
	\end{gather*}
	which would be incompatible with the indecomposability hypothesis on
	$V$.
\end{proof}

\begin{remark}
	Inspection of the final argument shows that the indecomposability
	hypothesis on $V$ may be weakened to indecomposability of type
	$\mathbf O^\ast (n,1)$.
\end{remark}

\begin{remark}
	In case $m = 2$, a related inequality for submanifolds involving the
	second fundamental form is provided in \cite[Lemma 1.2]{MR1243525}.
\end{remark}

\section{Sobolev-Poincaré inequality}

The purpose of this section is to establish (see
\ref{thm:new_sobolev_inequality}) our new Sobolev-Poincaré inequality, Theorem
\ref{Thm:new-Sobolev-Poincare}; the key ingredient therein is a monotonicity
lemma (see \ref{lemma:density_ratio}) based on the isoperimetric inequality.
Following this, the setting of connected submanifolds, Corollary
\ref{Final-corollary:new-Sobolev-Poincare}, results as corollary (see
\ref{corollary:new_sobolev_inequality}).  We also include a simpler version of
our Sobolev-Poincaré inequality for varifolds that are suitably indecomposable
and have no boundary (see \ref{corollary:sobolev_inequality}).

The lemma below will presently be applied only with $q =
\infty$.  We include the case $q < \infty$ since it yields a new proof of
previous Sobolev inequalities (see
\ref{remark:previous_sobolev_inequalities}).

\begin{lemma} \label{lemma:density_ratio}
	Suppose $m$ and $n$ are positive integers, $m \leq n$, $U$ is an open
	subset of $\mathbf R^n$, $V \in \mathbf V_m ( U )$, $\| \updelta V
	\|$ is a Radon measure, $\boldsymbol \Uptheta^m ( \| V \|, x ) \geq 1$
	for $\| V\|$ almost all $x$, $f \in \mathbf T_{\Bdry U} ( V )$, $0
	\leq s \leq r \leq \| V \|_{(\infty)} (f)$,
	\begin{align*}
		\text{either}, & \quad \text{$m=q=1$ and $\lambda = 1$}, \\
		\text{or}, & \quad \text{$m<q \leq \infty$ and $\lambda =
		\left ((1/m-1/q) \big / (1-1/q) \right )^{1-1/q}$},
	\end{align*}
	$V \weakD f \in \mathbf L_q ( \| V \|, \Hom ( \mathbf R^n, \mathbf R
	))$, $0 < \epsilon < \boldsymbol \upgamma ( m)^{-1}$,
	\begin{gather*}
		\| V \| \, \{ x \with f(x) \geq y \} < \infty, \quad \|
		\updelta V \| \, \{ x \with f(x) \geq y \} \leq \epsilon \,
		\| V \| ( \{ x \with f(x) \geq y \} )^{1-1/m}
	\end{gather*}
	for $s < y < r$, and $\delta = \boldsymbol \upgamma (m)^{-1} -
	\epsilon$.
	
	Then, the quantities
	\begin{align*}
	\| V \| ( \{ x \with f(x) \geq y \})^{1/m-1/q} ( \| V \|
	\restrict \{ x \with f(x) \geq y\})_{(q)}( V \weakD f ) +
	\delta \lambda y, & \quad \text{if $q < \infty$}, \\
	\| V \| ( \{ x \with f(x) \geq y \})^{1/m} \| V \|_{(\infty)}(
	V \weakD f ) + \delta \lambda y, & \quad \text{if $q =
		\infty$},
	\end{align*}
	are nonincreasing in $y$, for $s < y < r$.
\end{lemma}

\begin{proof}
	We treat the case $m < q < \infty$.  The cases $m=q=1$ and $q =
	\infty$ follow by a similar but simpler argument.  Abbreviate $\alpha
	= 1-1/q$ and $\beta = 1-1/m$.  Let $i : U \to \mathbf R^n$ denote the
	inclusion, define
	\begin{gather*}
	E(\upsilon) = \{ x \with f(x) \geq -\upsilon \}, \quad g
	(\upsilon) = \| V \| ( E (\upsilon) ), \quad G ( \upsilon ) =
	g( \upsilon )^{1-\beta/\alpha}, \\
	{\textstyle h (\upsilon) = \int_{E(\upsilon)} | V \weakD f |^q
		\ud \| V \|}, \quad W_\upsilon = i_\# ( V \restrict
	E(\upsilon) \times \mathbf G (n,m) ) \in \mathbf V_m ( \mathbf
	R^n )
	\end{gather*}
	for $-r < \upsilon < -s$, and notice that
	\begin{equation*}
	\boldsymbol \Uptheta^m (\| W_\upsilon \|, x ) \geq 1 \quad
	\text{for $\| W_\upsilon \|$ almost all $x$}
	\end{equation*}
	by \cite[2.8.9, 2.8.18, 2.9.11]{MR41:1976}.  Furthermore, as $( \| V
	\| + \| \updelta V \| ) \, \{ x \with f(x) = \upsilon \} = 0$ for all
	but countably many $\upsilon$, we have (see \cite[9.1]{MR3528825})
	\begin{equation*}
		\| \updelta W_\upsilon \| \leq i_\# \big ( \| \updelta V \|
		\restrict E(\upsilon) + \| V \boundary E (\upsilon) \| \big)
		\quad \text{for $\mathscr L^1$ almost all $-r < \upsilon <
		-s$}.
	\end{equation*}
	
	For such $\upsilon$, the isoperimetric inequality (see \cite[3.5,
	3.7]{MR3777387}) yields
	\begin{equation*}
		\boldsymbol \upgamma (m )^{-1} g(\upsilon)^{1-1/m} \leq  \|
		\updelta V \| ( E(\upsilon) )  + \| V \boundary E(\upsilon) \|
		( U ).
	\end{equation*}
	In view of \cite[8.29]{MR3528825} and \cite[2.9.19]{MR41:1976}, we
	deduce the inequalities
	\begin{gather*}
		0 < ( \boldsymbol \upgamma (m)^{-1} - \epsilon )
		g(\upsilon)^{1-1/m} \leq \| V \boundary E(\upsilon) \| ( U )
		\leq g'(\upsilon)^{1-1/q} h'(\upsilon)^{1/q}, \\
		\delta \lambda = ( \boldsymbol \upgamma(m)^{-1} - \epsilon )
		(1-\beta/\alpha)^\alpha  \leq (g^{1-\beta/\alpha} )'
		(\upsilon)^\alpha h' (\upsilon)^{1-\alpha} =
		G'(\upsilon)^\alpha h'(\upsilon)^{1-\alpha}
	\end{gather*}
	for $\mathscr L^1$ almost all $- r < \upsilon < -s$.  Therefore,
	noting $0 < \int_{-r}^\upsilon h' \ud \mathscr L^1 \leq h(\upsilon)$
	for $- r < \upsilon < -s$ by \cite[2.9.19]{MR41:1976}, we obtain
	(using the inequality relating geometric and arithmetic means)
	\begin{equation*}
		\delta \lambda \leq \alpha G'(\upsilon) G(\upsilon)^{\alpha-1}
		h (\upsilon)^{1-\alpha} + (1-\alpha) h'(\upsilon)
		h(\upsilon)^{-\alpha} G(\upsilon)^\alpha = (G^\alpha
		h^{1-\alpha})'(\upsilon)
	\end{equation*}
	for $\mathscr L^1$ almost all $-r < \upsilon < -s$, whence the
	conclusion follows by integration with respect to $\mathscr L^1$
	using \cite[2.9.19]{MR41:1976}.
\end{proof}

\begin{remark}
	For $q = \infty$, the pattern of the preceding proof is
	that of \cite[8.3]{MR0307015}.
\end{remark}

\begin{remark} \label{remark:previous_sobolev_inequalities}
	The preceding lemma in particular entails the
	estimates \cite[10.1\,(2b)\,(2d)]{MR3528825} with a different,
	somewhat more explicit constant.
\end{remark}

We next gather the set of conditions on density and first variation that we
assume for both varifolds occurring in the Sobolev-Poincaré estimate
in \ref{thm:new_sobolev_inequality}.

\begin{miniremark} \label{miniremark:new_situation}
	Suppose $U$ is an open subset of $\mathbf R^n$, $V$ is a varifold
	in $U$, $m = \dim V$, $\| \updelta V \|$ is a Radon measure,
	$\boldsymbol \Uptheta^m ( \| V \|, x ) \geq 1$ for $\| V \|$ almost
	all $x$, and, either $m=1$ and $\phi = \| V \|$, $m=2$ and $\phi = \|
	\updelta V \|$, or $m > 2$, $\mathbf h ( V, \cdot) \in \mathbf
	L_{m-1}^\mathrm{loc} ( \| V \|, \mathbf R^n )$, and $\phi = \| V \|
	\restrict | \mathbf h ( V, \cdot ) |^{m-1}$.
\end{miniremark}

\begin{theorem} \label{thm:new_sobolev_inequality}
	Suppose $m$ and $n$ are
	integers, $2 \leq m \leq n$, $U$ is an open subset of $\mathbf R^n$,
	$V_1 \in \mathbf V_m ( U )$ and $V_2 \in \mathbf V_{m-1} ( U )$
	satisfy the conditions of \ref{miniremark:new_situation},
	\begin{gather*}
	V_2 = 0 \quad \text{if $m=2$}, \qquad \| \updelta V_1 \| \leq
	\| V_1 \| \restrict | \mathbf h ( V_1, \cdot ) | + \| V_2 \|
	\quad \text{if $m>2$}, \\
	\text{$\| \updelta V_2 \|$ is absolutely continuous with
		respect to $\| V_2 \|$} \quad \text{if $m > 3$}, \\
	f \in \mathbf T_{\Bdry U} ( V_i ), \qquad | V_i \weakD f
	(x) | \leq 1 \quad \text{for $\| V_i \|$ almost all $x$}, \\
	\text{$\phi_i$ are associated with $V_i$ as
		in \ref{miniremark:new_situation}},
	\end{gather*}
	for $i \in \{ 1,2 \}$,
	and $E = \{ x \with f(x) > 0 \}$.
	
	Then, there exists a Borel subset $Y$ of $\mathbf R$ such that
	\begin{gather*}
	f(x) \in Y \quad \text{for $\| V_1 \|$ almost all $x$}, \\
	\mathscr L^1 ( Y ) \leq \Gamma \big ( \| V_1 \| ( E )^{1/m} +
	\phi_1 (E) + \| V_2 \| ( E
	)^{1/(m-1)} + \phi_2 ( E ) \big ),
	\end{gather*}
	where $\Gamma$ is a positive finite number determined
	by $m$.
\end{theorem}

\begin{proof}
	We assume $( \| V_i \| + \phi_i ) (E) < \infty$ for $i \in \{ 1,2 \}$;
	in particular, we have $\| \updelta V_i \| ( E ) < \infty$.  Define $I
	= \mathbf R \cap \{ y \with y > 0 \}$,
	\begin{equation*}
	\mu_i = f_\# \| V_i \|, \quad \nu_i = f_\# \| \updelta V_i \|,
	\quad \text{and} \quad \omega_i = f_\# \phi_i
	\end{equation*}
	for $i \in \{ 1,2 \}$.  Let $\alpha = \omega_1$ if $m=2$ and $\alpha =
	f_\# ( \| V_1 \| \restrict | \mathbf h (V_1,\cdot) | )$ if $m>2$.
	With
	\begin{equation*}
	\Delta_1 = 2 m \boldsymbol \upgamma (m), \quad \Delta_2 = \sup
	\big \{ 2(m-1) \boldsymbol \upgamma (m-1), 2 \Delta_1 ( 2
	\boldsymbol \upgamma ( m ))^{1/(m-1)} \big \},
	\end{equation*}
	we define
	\begin{equation*}
	\lambda_i = \Delta_i \, \mu_i ( I )^{1/\dim V_i} \quad
	\text{for $i \in \{ 1, 2 \}$}
	\end{equation*}
	and functions $r_i : \mathbf R \to \mathbf R$, for $i \in
	\{ 1, 2 \}$, by
	\begin{equation*}
	r_i (b) = \sup \big \{ s \with \text{$0 \leq s < b$ and
		$\Delta_i \, \mu_i ( \mathbf B (b,s) )^{1/\dim V_i} \geq s$}
	\big \} \quad \text{whenever $b \in \mathbf R$}.
	\end{equation*}
	Since the sets $( \mathbf R \times \mathbf R ) \cap \big \{ (b,s)
	\with \text{$0 \leq s < b$ and $\Delta_i \, \mu_i ( \mathbf B
	(b,s))^{1/\dim V_i} \geq s$} \big \}$ are relatively closed in $(
	\mathbf R \times \mathbf R ) \cap \{ (b,s) \with s < b \}$, we may
	deduce that $r_i$ are Borel functions for $i \in \{ 1,2 \}$.  We also
	note $r_i(b) \leq \lambda_i$ for $b \in \mathbf R$ and $i \in \{ 1,2
	\}$.  Let $C = \{ b \with \mu_1 \, \{ b \} > 0 \}$ and notice that
	$C$ is countable.  Moreover, we define
	\begin{equation*}
		Q_i = \mathbf R \cap \{ b \with r_i ( b ) > 0 \} \quad
		\text{for $i \in \{ 1, 2 \}$}, \qquad B = \{ b \with r_1 (b) >
		r_2 ( b ) \}.
	\end{equation*}
	Our two estimates below rest on the basic fact that
	\begin{equation*}
		\nu_i \, \mathbf B (b,r_i(b)) \geq (2 \boldsymbol
		\upgamma(\dim V_i))^{-1} \mu_i ( \mathbf B
		(b,r_i(b)))^{1-1/\dim V_i}
	\end{equation*}
	whenever $\lambda_i < b \in Q_i$ and $i \in \{ 1,2 \}$; in fact, we
	note \cite[8.12, 8.13, 9.9]{MR3528825} and apply, for small $s > 0$,
	\ref{lemma:density_ratio} with $m$, $V$, $s$, $r$, $q$, $f(x)$, and
	$\epsilon$ replaced by $\dim V_i$, $V_i$, $0$, $s$, $\infty$, $\sup
	\{ r_i(b) + s - | f(x)-b|, 0 \}$, and $( 2 \boldsymbol \upgamma(\dim
	V_i))^{-1}$.
	
	Next, the following two estimates will be proven
	\begin{equation*}
		\mathscr L^1 ( Q_2 \cap \{ b \with b > \lambda_2 \} ) \leq
		\Delta_3 \, \omega_2 (I), \quad \mathscr L^1 ( B \cap \{ b
		\with b > \lambda_1 \} ) \leq \Delta_4 \, \omega_1 (I),
	\end{equation*}
	where $\Delta_3 = 2^{m+1} \Delta_2 \boldsymbol \upgamma (m-1)^{m-2}$
	and $\Delta_4 = 2^{m(m-1)+3} \boldsymbol \upgamma (m)^{m-1} \Delta_1$.
	Whenever $\lambda_2 < b \in Q_2$, the basic fact and Hölder's
	inequality yield
	\begin{equation*}
		\Delta_2^{-1} r_2 (b) \leq \mu_2 ( \mathbf B ( b, r_2
		(b)))^{1/(m-1)} \leq ( 2 \boldsymbol \upgamma (m-1) )^{m-2}
		\omega_2 \, \mathbf B (b,r_2(b)),
	\end{equation*}
	whence we deduce the first estimate by Vitali's covering theorem
	(see \cite[2.8.5, 2.8.8]{MR41:1976} with $\delta = \diam$ and
	$\tau=3/2$).  To similarly prove the second estimate suppose
	$\lambda_1 < b \in B$.  We first notice that $b > r_1(b) > r_2(b)$
	yields
	\begin{align*}
		\mu_2 ( \mathbf B (b,r_1(b)) )^{1/(m-1)} & \leq \Delta_2^{-1}
		r_1(b) \\
		& \leq 2^{-1} (2 \boldsymbol \upgamma (m) )^{1/(1-m)} \mu_1 (
		\mathbf B (b,r_1(b))^{1/m} < \infty.
	\end{align*}
	Combining this estimate with the following consequence of the basic
	fact,
	\begin{align*}
		& \mu_1 ( \mathbf B  ( b,r_1(b)) )^{1/m} \leq ( 2 \boldsymbol
		\upgamma (m) )^{1/(m-1)} \nu_1 ( \mathbf B
		(b,r_1(b)))^{1/(m-1)} \\
		& \qquad \leq ( 2 \boldsymbol \upgamma (m) )^{1/(m-1)} \big (
		\alpha ( \mathbf B (b,r_1(b)))^{1/(m-1)} + \mu_2 ( \mathbf B
		(b,r_1(b)))^{1/(m-1)} \big ),
	\end{align*}
	we first obtain
	\begin{equation*}
		2^{-1} \mu_1 ( \mathbf B ( b,r_1(b)))^{1/m} \leq ( 2
		\boldsymbol \upgamma (m))^{1/(m-1)} \alpha ( \mathbf B (
		b,r_1(b)))^{1/(m-1)},
	\end{equation*}
	and then, using Hölder's inequality,
	\begin{equation*}
		\Delta_1^{-1} r_1(b) \leq \mu_1 ( \mathbf B (b,r_1(b)) )^{1/m}
		\leq 2^{m(m-1)} \boldsymbol \upgamma (m)^{m-1} \, \omega_1 \,
		\mathbf B (b,r_1(b)).
	\end{equation*}
	Vitali's covering theorem now yields the second estimate.
	
	We now define $Y = C \cup Q_1$. Since $Q_1 \subset
	B \cup Q_2 \subset I$, the preceding two estimates imply that the
	asserted property of $Y$ may be established by proving
	\begin{equation*}
		\mu_1 ( \mathbf R \without ( C \cup Q_1 )) = 0.
	\end{equation*}
	For this purpose, we define
	\begin{equation*}
		\Upsilon = ( \spt \mu_1 ) \cap \left \{ y \with \text{$y > 0$
		and $\limsup_{s \to 0+} \frac{\nu_1 \, \mathbf B (y,s)}{\mu_1
		( \mathbf B (y,s) )^{1-1/m}} \geq (2 \boldsymbol
		\upgamma(m))^{-1}$} \right \}
	\end{equation*}
	and notice that $\mu_1 ( \Upsilon \without C ) = 0$ by \cite[2.8.9,
	2.8.18, 2.9.5]{MR41:1976}.  The assertion then follows verifying $(
	\spt \mu_1 ) \without \Upsilon \subset Q_1 \cup \{ 0 \}$; in fact, for
	$0 < b \in ( \spt \mu_1 ) \without \Upsilon$ and small $s>0$, we
	note \cite[8.12, 8.13, 9.9]{MR3528825} and apply
	\ref{lemma:density_ratio} with $V$, $s$, $r$, $q$, $f(x)$, and
	$\epsilon$ replaced by $V_1$, $0$, $s$, $\infty$, $\sup \{ s-|f(x)-b|,
	0 \}$, and $(2 \boldsymbol \upgamma(m))^{-1}$, to infer $r_1 (b) > 0$.
\end{proof}

\begin{remark} \label{remark:sobolev_indecomposable}
	We recall from \cite[7.12, 7.13\,(1)]{arXiv:2209.05955v2} that, under
	the hypotheses of the preceding theorem, if $V_1$ is indecomposable of
	type $\{ f \}$, then
	\begin{equation*}
		\diam \spt f_\# \| V_1 \| \leq \mathscr L^1 (Y)
	\end{equation*}
	and if additionally $\| V_1 \| \, \{ x \with f(x) \leq y \} > 0$ for
	$0 < y < \infty$, then
	\begin{equation*}
		\diam \spt f_\# \| V_1 \| = \| V_1 \|_{(\infty)} (f).
	\end{equation*}
\end{remark}

\begin{remark} \label{remark:sobolev_poincare}
	We notice that $\mathbf T_{\Bdry U} (V_i) = \mathbf T (V_i) \cap \{ f
	\with f \geq 0 \}$ in case $U = \mathbf R^n$ by \cite[9.2]{MR3528825}.
\end{remark}

In the special case of \emph{connected} properly embedded
submanifolds of class $2$ and a function of class $1$ thereon, the following
corollary, Corollary \ref{Final-corollary:new-Sobolev-Poincare}, results.

\begin{corollary} \label{corollary:new_sobolev_inequality}
	Suppose $m$ and $n$ are integers, $2 \leq m \leq n$, $U$ is an open
	subset of $\mathbf R^n$, $M$ is a properly embedded, connected $m$
	dimensional submanifold-with-boundary of $U$ of class $2$, $f : M \to
	\mathbf R$ is of class $1$ relative to $M$, $\spt f$ is compact, $|
	\Der f (x) | \leq 1$ for $x \in M \without \partial M$, and $E = M
	\cap \{ x \with f(x) \neq 0 \}$.

	Then, there holds
	\begin{multline*}
		\diam f[M] \leq 2
		\Gamma_{\textup{\ref{thm:new_sobolev_inequality}}} (m) \big (
		\mathscr H^m ( E \cap M)^{1/m} + {\textstyle\int_{E \cap M}} |
		\mathbf h ( M, x ) |^{m-1} \ud \mathscr H^m \, x \\
		+ \mathscr H^{m-1} (E \cap \partial M)^{1/(m-1)} +
		{\textstyle\int_{E \cap \partial M}} | \mathbf h ( \partial
		M, x) |^{m-2} \ud \mathscr H^{m-1} \, x \big);
	\end{multline*}
	here, the summand $\int_{E \cap \partial M} | \mathbf h ( \partial M,
	x ) |^{m-2} \ud \mathscr H^{m-1} \, x$ shall be omitted if $m=2$.
\end{corollary}

\begin{proof}
	We recall \cite[6.14]{arXiv:2209.05955v2}. We define $V_1 \in
	\mathbf{RV}_m ( U )$ and $V_2 \in \mathbf{RV}_{m-1} ( U )$ such that
	\begin{equation*}
		\| V_1 \| = \mathscr H^m \restrict M, \quad \text{$V_2 = 0$ if
		$m=2$}, \quad \text{$\| V_2 \| = \mathscr H^{m-1} \restrict
		\partial M$ if $m>2$};
	\end{equation*}
	hence, $V_1$ is indecomposable of type $\mathscr D ( U, \mathbf R )$
	by \cite[7.9]{arXiv:2209.05955v2}.
	Employing \cite[3.1.22]{MR41:1976}, we construct a
	function $g : U \to \mathbf R$ of class $1$ with compact support such
	that $g|M = f$.  From \cite[8.7, 8.16, 9.2, 9.4]{MR3528825}, we infer
	that $|g| \in \mathbf T_{\Bdry U} ( V_i )$ with
	\begin{equation*}
		\big | V_i \weakD |g| (x) \big | \leq 1 \quad \text{for
		$\| V_i \|$ almost all $x$}
	\end{equation*}
	for $i \in \{ 1,2 \}$.  In view of
	\cite[7.13\,(2)]{arXiv:2209.05955v2} the conclusion now follows from
	\ref{thm:new_sobolev_inequality} and
	\ref{remark:sobolev_indecomposable} applied with $f$ replaced by $|g|$
	because $\diam f[M] \leq 2 \diam g[M]$, if $0 \in \im f$, and $\diam
	f[M] = \diam g[M]$ as $f[M]$ is an interval, if $0 \notin \im f$.
\end{proof}

To prepare for the case without boundary, we collect another
set of hypotheses on density and first variation.  It will be assumed to hold
in \ref{corollary:sobolev_inequality}.

\begin{miniremark} \label{miniremark:situation}
	Suppose $U$ is an open subset of $\mathbf R^n$, $V$ is a varifold
	in $U$, $2 \leq m = \dim V$, $\| \updelta V \|$ is a Radon measure,
	$\boldsymbol \Uptheta^m ( \| V \|, x ) \geq 1$ for $\| V \|$ almost
	all $x$, and, either $m=2$ and $\psi = \| \updelta V \|$, or $m > 2$,
	$\| \updelta V \|$ is absolutely continuous with respect to $\| V \|$,
	$\mathbf h ( V, \cdot) \in \mathbf L_{m-1}^\mathrm{loc} ( \| V \|,
	\mathbf R^n )$, and $\psi = \| V \| \restrict | \mathbf h ( V, \cdot )
	|^{m-1}$.
\end{miniremark}

\begin{corollary} \label{corollary:sobolev_inequality}
	Suppose $U$, $V$, and $\psi$ are as in \ref{miniremark:situation},
	$\Gamma = 2^{m+3} m \boldsymbol \upgamma (m)$, $f \in \mathbf T_{\Bdry
		U} (V)$, $V$ is indecomposable of type $\{ f \}$, and
	\begin{equation*}
	| V \weakD f(x) | \leq 1 \quad \text{for $\| V \|$ almost
		all $x$}.
	\end{equation*}
	
	Then, there holds
	\begin{equation*}
	\diam \spt f_\# \| V \| \leq \Gamma \big ( \| V \| ( \{ x
	\with f(x)>0 \})^{1/m} + \boldsymbol \upgamma (m)^{m-1} \psi \,
	\{ x \with f (x) > 0 \} \big).
	\end{equation*}
\end{corollary}

\begin{proof}
	With a possibly larger number $\Gamma$, this follows from
	\ref{thm:new_sobolev_inequality} and \ref{remark:sobolev_indecomposable}
	with $V_1 = V$ and $V_2 = 0$.  We verify the eligibility of the
	present number $\Gamma$ by noting that, for $V_2=0$, we can take
	$\Delta_4 = 2^{m+2} \Delta_1 \boldsymbol \upgamma (m)^{m-1}$ in the
	proof of \ref{thm:new_sobolev_inequality}; in fact,
	\begin{equation*}
	\Delta_1^{-1} r_1 (b) \leq \mu_1 ( \mathbf B ( b, r_1
	(b)))^{1/m} \leq ( 2 \boldsymbol \upgamma (m) )^{m-1}
	\omega_1 \, \mathbf B (b,r_1(b))
	\end{equation*}
	whenever $\lambda_1 < b \in B$
	by the basic fact and Hölder's inequality.
\end{proof}

\begin{remark} \label{remark:diameter_seminorm}
	As in \ref{remark:sobolev_indecomposable}, we note
	that, in case $\| V \| \, \{ x \with f(x) \leq y \}
	> 0$ whenever $0 < y < \infty$, we
	have $\diam \spt f_\# \| V \| = \| V \|_{(\infty)} ( f)$.
\end{remark}

\section{Examples} \label{section:density_bounds}

In the present section, we construct (see
\ref{thm:optimality_exceptional_set} and \ref{thm:example_p_small}) the
\hyperlink{Example}{Example} mentioned the introduction which shows the
sharpness of Theorem \ref{Thm:varifold-density-lower-bound}.  As preparations,
we list an arithmetic formula and terminology for cylinders (see
\ref{miniremark:geometric_series}--\ref{miniremark:cylinders}) and then
indicate a procedure to smooth out corners (see
\ref{lemma:corners}--\ref{remark:corners}).

\begin{miniremark} \label{miniremark:geometric_series}
	If 
	$0 \leq x < 1$ and $i$ is a nonnegative integer, then
	\begin{equation*}
	\sum_{j=i}^\infty (j+1) x^j = (1-x)^{-2} \big ( (i+1) x^i - i
	x^{i+1} \big ).
	\end{equation*}
\end{miniremark}

\begin{miniremark} \label{miniremark:cylinders}
	Whenever $u \in \mathbf S^{n-1}$, $a \in \mathbf R^n$, $0 < r <
	\infty$, and $0 \leq h \leq \infty$, we define
	\begin{align*}
	Z(a,r,u,h) & = \mathbf R^n \cap \big \{ x \with |x-a|^2 =
	((x-a) \bullet u )^2 + r^2, 0 \leq (x-a) \bullet u \leq h
	\big \}.
	\end{align*}
\end{miniremark}

\begin{lemma} \label{lemma:corners}
	Suppose $n$ is an integer, $n \geq 2$, $Y$ is an $n - 1$ dimensional
	submanifold-with-boundary of class $\infty$ of $\mathbf R^{n-1}$,
	$\partial Y$ is connected and compact, $\epsilon > 0$, and,
	identifying $\mathbf R^n \simeq \mathbf R^{n-1} \times \mathbf R$, the
	subsets $Q$ and $U$ of $\mathbf R^n$ satisfy
	\begin{equation*}
		Q  \simeq Y \times \{ t \with 0 \leq t < \infty \}, \quad U
		\simeq ( Y \cap \{ y \with \dist (y, \partial Y ) < \epsilon
		\} ) \times \{ z \with 0 \leq z < \epsilon \}.
	\end{equation*}
	
	Then, there exists a properly embedded $n$ dimensional
	submanifold-with-boundary $M$ of class $\infty$ of $\mathbf R^n$ such
	that $\partial M$ is connected and $M \without U = Q \without U$.
\end{lemma}

\begin{proof}
	As $\partial Y$ is compact, we employ \cite[3.9,
	3.12]{arXiv:2209.05955v2} to construct $\delta > 0$ such that the
	function $f : G \to \mathbf R$, with $G = \mathbf R^{n-1} \cap \{ y
	\with \dist (y, \partial Y) < \delta \}$ and
	\begin{equation*}
		f(y) = \dist (y,\partial Y) \quad \text{if $y \in Y$}, \qquad
		f(y) = - \dist (y,\partial Y) \quad \text{else},
	\end{equation*}
	for $y \in G$, is of class $\infty$ and satisfies $| \Der f(y) | = 1$
	for $y \in G$.  Then, defining $g : G \times \mathbf R \to \mathbf R
	\times \mathbf R$ by $g(y,z) = (f(y),z)$ for $y \in G$ and $z \in
	\mathbf R$ and noting $\im \Der g (y,z) = \mathbf R \times \mathbf R$
	for $y \in G$ and $z \in \mathbf R$, the assertion reduces (e.g.,
	by \cite[3.1.18]{MR41:1976}) to the case $n=2$ and $Y = \{ y \with 0
	\leq y < \infty \}$ which is elementary.
\end{proof}

\begin{remark} \label{remark:corners}
	By induction on $n$, the preceding lemma implies the following
	proposition: \emph{If $2 \leq n \in \mathbf Z$, $- \infty < a_k < b_k
	< \infty$ for $k = 1, \ldots, n$, $\epsilon > 0$, and
	\begin{gather*}
		Q = \mathbf R^n \cap \{ x \with \textup{$a_k \leq x \bullet
		e_k \leq b_k$ for $k = 1, \ldots, n$} \}, \\
		U = Q \cap \big \{ x \with \dist (x, Q \cap \{ \chi \with
		\card \{ k \with \textup{$\chi \bullet e_k = a_k$ or $\chi
		\bullet e_k = b_k$} \} \geq 2 \}) < \epsilon \big \},
	\end{gather*}
	where $e_1, \ldots, e_n$ form the standard base of $\mathbf R^n$, then
	there exists a properly embedded $n$ dimensional
	submanifold-with-boundary $M$ of $\mathbf R^n$ of class $\infty$ such
	that $\partial M$ is connected and $M \without U = Q \without U$.}
\end{remark}

In the first example below related to the sharpness of
Theorem \ref{Thm:varifold-density-lower-bound}, we are able to control the
second fundamental form $\mathbf b (M,\cdot)$ instead of merely $\mathbf h
(M,\cdot)$.

\begin{theorem} \label{thm:optimality_exceptional_set}
	Whenever $m$ and $n$ are positive integers, $2 \leq m < n$, and $m-1 <
	q < m$, there exists a bounded, connected $m$ dimensional
	submanifold $M$ of class $\infty$ of $\mathbf R^n$ such that
	\begin{gather*}
		{\textstyle \int_M \| \mathbf b (M,x) \|^p \ud \mathscr H^m \,
		x } < \infty \quad \text{whenever $1 \leq p < q$}, \\
		{\textstyle \int_M \Tan (M,x)_\natural \bullet \Der \theta (x)
		\ud \mathscr H^m \, x = - \int_M \mathbf h (M,x) \bullet
		\theta (x) \ud \mathscr H^m \, x}
	\end{gather*}
	for $\theta \in \mathscr D ( \mathbf R^n, \mathbf R^n )$, and such
	that $A = ( \Clos M ) \without M$ satisfies
	\begin{equation*}
		\mathscr H^{m-q} ( A ) = \boldsymbol \upalpha (m-q) 2^{1-m+q},
		\qquad \boldsymbol \Uptheta^m ( \mathscr H^m \restrict M, a )
		= 0 \quad \text{for $a \in A$}.
	\end{equation*}
\end{theorem}

\begin{proof}
	We assume $n = m+1$ and let $d = m-q$.  Whenever $J$ is a compact
	subinterval of $\mathbf{R}$, we denote by $\Phi( J )$ the family
	consisting of the two disjoint subintervals
	\begin{align*}
		\big \{ t \with \inf J \leq t \leq \inf J + 2^{-1/d} \diam J
		\big \}, \quad \big \{ t \with \sup J - 2^{-1/d} \diam J \leq
		t \leq \sup J \big \}
	\end{align*}
	of $J$.  Letting $G_0 = \{ \{ t \with -1/2 \leq t \leq 1/2 \} \}$, we
	define $G_i = \textstyle \bigcup \{ \Phi(S) \with S \in G_{i-1} \}$
	for every positive integer $i$ and $C = \bigcap_{i=0}^\infty \bigcup
	G_i$.  By \cite[2.10.28]{MR41:1976}, there holds
	\begin{equation*}
		\mathscr H^d ( C ) = \boldsymbol \upalpha (d) 2^{-d}.
	\end{equation*}
	For every nonnegative integer $i$, we let $r_i = 2^{-i/d}$ and $s_i =
	\sum_{j=i}^\infty (j+1) r_j$, hence $\diam J = r_i$ whenever $J \in
	G_i$ and, using \ref{miniremark:geometric_series}, we compute
	\begin{equation*}
		s_i = ( 1 - 2^{-1/d} )^{-1} \big ( i + (1-2^{-1/d})^{-1} \big)
		r_i.
	\end{equation*}
	
	Suppose $e_1, \ldots, e_n$ form the standard base of $\mathbf R^n$.
	We observe that \ref{lemma:corners} may be employed to construct a
	subset $N$ of the cube
	\begin{equation*}
	\mathbf R^n \cap \big \{ x \with \text{$0 < x \bullet e_n <1$
		and $| x \bullet e_k | < {\textstyle \frac 12} $ for $k = 1,
		\ldots, n-1$} \big \}
	\end{equation*}
	such that its union with (see \ref{miniremark:cylinders})
	\begin{equation*}
	Z \big ( {\textstyle \frac {r_1-1}2} e_1 + e_n,
	{\textstyle\frac{r_1}4}, e_n, \infty \big ) \cup Z \big (
	{\textstyle\frac{1-r_1}2} e_1 + e_n, {\textstyle \frac{r_1}4},
	e_n, \infty \big ) \cup Z \big ( 0, {\textstyle \frac 14},
	-e_n, \infty )
	\end{equation*}
	is a properly embedded, connected $m$ dimensional submanifold of
	class $\infty$ of $\mathbf R^n$.  In particular, there exists $0 \leq
	\kappa < \infty$ satisfying
	\begin{equation*}
	\mathscr H^m ( N ) \leq \kappa,\quad \sup \im \| \mathbf b (N,
	\cdot ) \| \leq \kappa.
	\end{equation*}
	With $N (a,r) = \mathbf R^n \cap \{ x \with r^{-1} (x-a) \in N
	\}$ for $a \in \mathbf R^n$ and $0 < r < \infty$, we use
	\begin{align*}
	X_i & = \big \{ Z \big ( {\textstyle \frac{\sup J+\inf J }2}
	e_1 + (s_0-s_i)e_n, {\textstyle \frac{r_i}4}, e_n, ir_i )
	\with J \in G_i \big \}, \\
	\Psi_i & = \big \{ N \big ( {\textstyle \frac{\sup J +\inf
			J}2} e_1 + (s_0-s_i+ir_i) e_n, r_i \big ) \with J \in G_i \big
	\},
	\end{align*}
	$H_i = \bigcup_{j=0}^i \bigcup ( X_j \cup \Psi_j )$, and $M_i = \{ x
	\with \text{$x \in H_i$ or $x-2(x\bullet e_n)e_n \in H_i$} \}$, to
	define
	\begin{equation*}
	{\textstyle M = \bigcup \{ M_i \with \text{$i$ is a
			nonnegative integer} \}}.
	\end{equation*}
	Clearly, $M$ is a connected $m$ dimensional submanifold of
	class $\infty$ of $\mathbf R^n$ and
	\begin{equation*}
	A = \mathbf R^n \cap \{ x \with \text{$x \bullet e_1 \in C$,
		$|x \bullet e_n| = s_0$, and $x \bullet e_k = 0$ for $k = 2,
		\ldots, n-1$} \},
	\end{equation*}
	where $A = ( \Clos M) \without M$.
	
	Next, we will show the following assertion.  \emph{There holds
	\begin{equation*}
		\mathscr H^m ( M \cap \mathbf B (a,r) ) \leq 2^{2+m/d}
		(1-2^{1-m/d})^{-2} ( m \boldsymbol \upalpha (m) + \kappa )
		(i+1)^{1+d-m} r^m
	\end{equation*}
	whenever $a \in A$, $i$ is a nonnegative integer, and $s_{i+1} \leq r
	\leq s_i$.}  For this purpose, we let $I = \{ t \with |t-a \bullet
	e_1| \leq r \}$ and firstly estimate
	\begin{equation*}
		\lambda_i = \card \{ J \with I \cap J \neq \varnothing, J \in
		G_i \} \leq 4 (1-2^{-1/d})^{-2} (i+1)^d;
	\end{equation*}
	in fact, $G_i$ is special for $\{ t \with -1/2 \leq t \leq 1/2 \}$
	by \cite[2.10.28\,(2)\,(4)]{MR41:1976}, hence
	\begin{equation*}
		\card \{ J \with I \supset J \in G_i \} r_i^d \leq (2r)^d \leq
		2 (1-2^{-1/d})^{-2} (i+1)^d r_i^d.
	\end{equation*}
	Since $2^{-i} \sum_{j=i}^\infty 2^j ( j+1 ) r_j^m \leq ( 1-2^{1-m/d}
	)^{-2} (i+1) r_i^m$ by \ref{miniremark:geometric_series}, we
	estimate
	\begin{equation*}
		\mathscr H^m ( M \cap \mathbf B (a,r) ) \leq \lambda_i ( m
		\boldsymbol \upalpha (m) + \kappa ) (1-2^{1-m/d})^{-2} (i+1)
		r_i^m.
	\end{equation*}
	Hence, together with the first estimate and $r_i \leq 2^{1/d}
	(1-2^{-1/d}) (i+1)^{-1} s_{i+1}$, the assertion follows.
	
	The assertion of the preceding paragraph implies
	\begin{equation*}
		\boldsymbol \Uptheta^m ( \mathscr H^m \restrict M, a ) = 0
		\quad \text{whenever $a \in A$}
	\end{equation*}
	and, as $\Clos M$ is compact, also $\mathscr H^m (M) < \infty$.
	Noting
	\begin{align*}
		{\textstyle \int_{Z(a,r,u,h)} \| \mathbf b ( Z(a,r,u,h), x )
		\|^p \ud \mathscr H^m \, x } & \leq m \boldsymbol \upalpha (m)
		r^{m-p} (h/r), \\
		{\textstyle \int_{N(a,r)} \| \mathbf b ( N(a,r), x ) \|^p \ud
		\mathscr H^m \, x} & \leq \kappa^{1+p} r^{m-p}
	\end{align*}
	for $a \in \mathbf R^n$, $0 < r < \infty$, $u \in \mathbf S^{n-1}$,
	and $0 \leq h \leq \infty$, we estimate
	\begin{equation*}
	{\textstyle \int_M \| \mathbf b ( M,x ) \|^p \ud \mathscr H^m
		\, x \leq 2 ( \kappa^{1+p} + 4 m \boldsymbol \upalpha ( m ) )
		\sum_{i=0}^\infty (i+1) 2^{i(1+(p-m)/d)} < \infty}
	\end{equation*}
	whenever $1 \leq p < q$.  Since $\mathscr H^{m-1} ( \partial M_i )
	\leq m \boldsymbol \upalpha (m) 2^{i ( 1 + (1-m)/d )}$ for every
	positive integer $i$, the conclusion now readily follows.
\end{proof}

\begin{remark} \label{remark:third_question}
	The preceding theorem answers the third question posed
	in \cite[Section A]{scharrer:MSc}.
\end{remark}

In the second example below related to the sharpness of
Theorem \ref{Thm:varifold-density-lower-bound}, we are again able to control
the second fundamental form $\mathbf b (M,\cdot)$ instead of $\mathbf h
(M,\cdot)$.

\begin{theorem} \label{thm:example_p_small}
	Whenever $m$ and $n$ are positive integers and $3 \leq m < n$, there
	exists a bounded, connected $m$ dimensional submanifold $M$ 
	of class $\infty$ of $\mathbf R^n$ such that
	\begin{gather*}
	{\textstyle \int_M \| \mathbf b (M,x) \|^p \ud \mathscr H^m \,
		x } < \infty \quad \text{whenever $1 \leq p < m-1$}, \\
	{\textstyle \int_M \Tan (M,x)_\natural \bullet \Der \theta (x)
		\ud \mathscr H^m \, x = - \int_M \mathbf h (M,x) \bullet
		\theta (x) \ud \mathscr H^m \, x}
	\end{gather*}
	for $\theta \in \mathscr D ( \mathbf R^n, \mathbf R^n )$, and such
	that
	\begin{equation*}
	\mathscr H^m ( ( \Clos M ) \without M ) = 1, \qquad
	\boldsymbol \Uptheta^m ( \mathscr H^m \restrict M, a ) = 0 \quad
	\text{for $a \in (\Clos M) \without M$}.
	\end{equation*}
\end{theorem}

\begin{proof}
	We assume $n = m+1$.  With each $A \subset \mathbf R^n$, we associate
	sets
	\begin{equation*}
	A(a,r) = \mathbf R^n \cap \big \{ x \with r^{-1} (x-a) \in A
	\big \} \quad \text{for $a \in \mathbf R^n$ and $0 < r <
		\infty$}.
	\end{equation*}
	Let $e_1, \ldots, e_n$ denote the standard base vectors of $\mathbf
	R^n$.  We define
	\begin{align*}
	D & = \mathbf R^n \cap \{ x \with \text{$x \bullet e_n = 0$
		and $|x| < 1$} \}, \\
	H & = \mathbf R^n \cap \{ u \with \text{for some $k \in \{ 1,
		\ldots, m \}$, $u = e_k$ or $u = -e_k$} \}.
	\end{align*}
	We define $\gamma : \mathbf R^n \to \mathbf 2^H$ by
	\begin{equation*}
	\gamma (x) = H \without \{ u \with \text{for some $k$, $x
		\bullet e_k = 1$ and $u = e_k$ or $x \bullet e_k = 0$ and $u =
		-e_k$} \}
	\end{equation*}
	for $x \in \mathbf R^n$.  Notice that \ref{lemma:corners} may be used
	to construct a subset $R$ of the cylinder
	\begin{equation*}
	\mathbf R^n \cap \big \{ {\textstyle x \with \text{$0 \leq
			x \bullet e_n < \frac 14$ and $|x-(x \bullet e_n)e_n| < \frac
			14$} } \big \}
	\end{equation*}
	such that $R \cup \big ( \mathbf R^n \cap \{ x \with x \bullet e_n = 0
	\} \without D ( 0, \frac 14) \big) \cup Z ( \frac{e_n}4, \frac 18,
	e_n, \infty )$, see \ref{miniremark:cylinders}, is a properly
	embedded, connected $m$ dimensional submanifold of $\mathbf R^n$ of
	class $\infty$.  Let $S$ denote the reflection $\{ x \with
	x-2(x\bullet e_n) e_n \in R \}$ of $R$ along $\mathbf R^n \cap \{ x
	\with x \bullet e_n = 0 \}$.  Considering the submanifold furnished
	by \ref{remark:corners} applied with $\epsilon = \frac 1{16}$ and
	$(a_k,b_k)$ replaced by $(-\frac 38, \frac 38)$ if $k<n$ and $(- \frac
	14,\frac 14)$ if $k = n$, we observe that \ref{lemma:corners} may also
	be employed to construct, for $G \subset H$, a subset $N_G$ of the
	cuboid
	\begin{equation*}
	\mathbf R^n \cap \big \{ {\textstyle x \with \text{$|x \bullet
			e_n| \leq \frac 14$ and $|x \bullet e_k| < \frac 12$ for $k =
			1, \ldots, m$} } \big \}
	\end{equation*}
	such that $N_G$ contains $D( \frac{e_n}4,\frac 14) \cup D
	(-\frac{e_n}4, \frac 14)$ and such
	that
	\begin{equation*}
	N_G \cup \bigcup {\textstyle \big \{ Z ( \frac u2, \frac 18,
		u, \infty ) \with u \in G \big \}}
	\end{equation*}
	is a properly embedded, connected $m$ dimensional submanifold
	of $\mathbf R^n$ of class $\infty$.  Clearly, there exists $0 \leq
	\kappa < \infty$ satisfying
	\begin{gather*}
	\mathscr H^m ( R ) \leq \kappa, \quad \sup \im \| \mathbf b
	(R,\cdot) \|
	\leq \kappa, \quad
	\mathscr H^m ( N_G ) \leq \kappa, \quad \sup \im \| \mathbf b
	( N_G, \cdot) \| \leq \kappa
	\end{gather*}
	whenever $G \subset H$.
	
	In this paragraph, we define various objects for each positive
	integer $i$.  Let $r_i = 2^{-i(i+1)}$ and define $C_i$ to consist of
	those $x \in \mathbf R^n$ such that
	\begin{equation*}
	x \bullet e_n = 2^{-i}, \quad 0 \leq x \bullet e_k \leq 1,
	\quad \text{and} \quad 2^{i-1} x \bullet e_k \in \mathbf Z
	\end{equation*}
	for $k = 1, \ldots, m$.  We have $\card C_i = (2^{i-1}+1)^m$.  Then,
	noting $\frac{r_i}2 < 2^{-i}$, we define $X_i(u)$, for $u \in H$, to
	be the family consisting of the sets
	\begin{equation*}
	Z \big ( {\textstyle x+\frac{r_i}2u, \frac{r_i}8,
		u, 2^{-i}-\frac{r_i}2 } \big )
	\end{equation*}
	corresponding to $x \in C_i$ with $u \in \gamma (x)$.  We have $\card
	X_i (u) = (2^{i-1}+1)^{m-1}2^{i-1}$.  With $C_0 = \varnothing$, we
	define $\Psi_i$ to be the family consisting of the sets
	\begin{equation*}
	{\textstyle N_{\gamma(x)} (x,r_i) \without D ( x-\frac{r_i}4
		e_n,\frac{r_{i+1}}4)}
	\end{equation*}
	corresponding to $x \in C_i$ with $x+2^{-i}e_n \notin C_{i-1}$ as
	well as the sets
	\begin{equation*}
	{\textstyle N_{\gamma(x)} (x,r_i) \without \big ( D (
		x+\frac{r_i}4 e_n,\frac{r_i}4) \cup D (x- \frac{r_i}4e_n,
		\frac{r_{i+1}}4 ) \big )}
	\end{equation*}
	corresponding to $x \in C_i$ with $x+2^{-i}e_n \in C_{i-1}$.  Noting
	$\frac{r_i}4 + \frac{3r_{i+1}}4 < r_i \leq 2^{-i-1}$, we also define
	$\Omega_i$ to be the family consisting of the sets
	\begin{align*}
	S \big ( {\textstyle x-\frac{r_i}4e_n,r_{i+1} } \big ) & \cup
	Z \big ({ \textstyle x- ( \frac{r_i}4+\frac{r_{i+1}}4)e_n,
		\frac{r_{i+1}}8, -e_n, 2^{-i-1} - \frac{r_i}4-\frac{3r_{i+1}}4
	} \big ) \\
	& \cup R \big ( {\textstyle x- (2^{-i-1}-\frac{r_{i+1}}4) e_n,
		r_{i+1}} \big )
	\end{align*}
	corresponding to $x \in C_i$.  Clearly, we have $\card \Psi_i = \card
	C_i = \card \Omega_i$.  Finally, we let $M_i = \bigcup_{j=1}^i
	\bigcup_{u \in H} \bigcup (  X_j (u) \cup \Psi_j \cup \Omega_j )$.
	
	Now, we define $M = \bigcup_{i=1}^\infty M_i$ and notice that $M$ is a
	bounded, connected $m$ dimensional submanifold of class $\infty$
	of $\mathbf R^n$ such that
	\begin{equation*}
	( \Clos M ) \without M = \mathbf R^n \cap \{ x \with \text{$x
		\bullet e_n = 0$ and $0 \leq x_k \leq 1$ for $k = 1, \ldots,
		m$} \}.
	\end{equation*}
	Since we have $\mathscr H^{m-1} ( \partial D (a,r)) = m \boldsymbol
	\upalpha (m) r^{m-1}$ and
	\begin{gather*}
	\mathscr H^m ( Z (a,r,u,h) ) = m \boldsymbol \upalpha ( m )
	r^{m-1} h, \quad \sup \im \| \mathbf b ( Z(a,r,u,h), \cdot )
	\| = r^{-1}, \\
	\mathscr H^m ( R(a,r) ) \leq \kappa r^m, \quad \sup \im \|
	\mathbf b (R(a,r), \cdot ) \| \leq \kappa r^{-1}, \\
	\mathscr H^m ( N_G(a,r) ) \leq \kappa r^m, \quad \sup \im \|
	\mathbf b (N_G(a,r),\cdot ) \| \leq \kappa r^{-1}
	\end{gather*}
	whenever $a \in \mathbf R^n$, $0 < r < \infty$, $u \in \mathbf
	S^{n-1}$, $0 < h < \infty$, and $G \subset H$, one may use the fact
	that $\sum_{i=1}^\infty 2^{i\lambda} r_i^\epsilon < \infty$
	whenever $\lambda \in \mathbf R$ and $\epsilon > 0$ to conclude
	\begin{gather*}
	\lim_{i \to \infty} 2^{im} \mathscr H^m ( M \without M_{i-1})
	= 0, \quad \lim_{i \to \infty} \mathscr H^{m-1} ( \partial M_i
	) = 0, \\
	{\textstyle \int_M \| \mathbf b (M,x) \|^p \ud \mathscr H^m \,
		x < \infty} \quad \text{for $1 \leq p < m-1$},
	\end{gather*}
	whence we readily deduce the asserted conclusion.
\end{proof}

\begin{remark}
	The construction bears some similarities with \cite[1.2]{MR2537022}.
\end{remark}

\section{Lower density bounds}

In this section, we provide (see
\ref{thm:few_special_points}) Theorem \ref{Thm:varifold-density-lower-bound}.
The key to this are conditional lower density ratio bounds (see
\ref{lemma:lower_density_ratio_bound}--\ref{remark:classic_swing}) which are
in turn based on the Sobolev-Poincaré inequality (see
\ref{thm:new_sobolev_inequality}).  Moreover, to treat small positive density
ratios, a compactness lemma (see \ref{lemma:density_lower_bound}) is
employed.

\begin{lemma} \label{lemma:density_lower_bound}
	Suppose $1 \leq M < \infty$.
	
	Then, there exists a positive, finite number $\Gamma$ with the
	following property.
	
	If $m$ and $n$ are positive integers, $m \leq n \leq M$, $a \in
	\mathbf R^n$, $0 < r < \infty$, $V \in \mathbf V_m ( \mathbf U
	(a,r))$, $\| \updelta V \| \, \mathbf U (a,r) \leq \Gamma^{-1}
	r^{m-1}$,
	\begin{equation*}
	\| V \| \, \mathbf B (a,s) \geq M^{-1} \boldsymbol \upalpha (m)
	s^m \quad \text{whenever $0 < s < r$},
	\end{equation*}
	and $\boldsymbol \Uptheta^m ( \| V \|,x ) \geq 1$ for $\| V \|$ almost
	all $x$, then
	\begin{equation*}
		\| V \| \, \mathbf U (a,r) \geq (1-M^{-1}) \boldsymbol
		\upalpha (m) r^m.
	\end{equation*}
\end{lemma}

\begin{proof}
	If the lemma were false for some $M$, there would exist sequences
	$\Gamma_i$ with $\Gamma_i \to \infty$ as $i \to \infty$ and sequences
	$m_i$, $n_i$, $a_i$, $r_i$, and $V_i$ showing that $\Gamma = \Gamma_i$
	does not have the asserted property.
	
	We could assume for some positive integers $m$ and $n$ that $m \leq
	n \leq M$, $m=m_i$, $n=n_i$, $a_i=0$, and $r_i=1$ whenever $i$ is a
	positive integer.  Defining $V \in \mathbf V_m ( \mathbf R^n \cap
	\mathbf U (0,1) )$ to be the limit of some subsequence of $V_i$, we
	would obtain
	\begin{equation*}
		\| V \| \, \mathbf U (0,1) \leq (1-M^{-1}) \boldsymbol
		\upalpha (m), \quad 0 \in \spt \| V \|, \quad \updelta V = 0.
	\end{equation*}
	Finally, using \cite[5.6, 8.6, 5.1\,(2)]{MR0307015},
	we would then conclude that
	\begin{gather*}
		\boldsymbol \Uptheta^m ( \| V \|, x ) \geq 1 \quad \text{for
		$\| V \|$ almost all $x$}, \\
		\boldsymbol \Uptheta^m ( \| V \|, 0 ) \geq 1, \quad \| V \| \,
		\mathbf U (0,1) \geq \boldsymbol \upalpha (m),
	\end{gather*}
	a contradiction.
\end{proof}

\begin{remark}
	The pattern of the preceding proof is that of \cite[7.3]{MR3528825}.
\end{remark}

The conditional lower density bounds follow rather immediately from the
Sobolev-Poincaré inequality (see \ref{thm:new_sobolev_inequality}) and its
corollary (see \ref{corollary:sobolev_inequality}), respectively.

\begin{lemma} \label{lemma:lower_density_ratio_bound}
	Suppose $m$ and $n$ are positive inters, $2 \leq m \leq n$, $U$ is an
	open subset of $\mathbf R^n$, $V_1 \in \mathbf V_m ( U )$ and $V_2 \in
	\mathbf V_{m-1} ( U )$ satisfy the conditions
	of \ref{miniremark:new_situation}, $V_1$ is indecomposable of
	type $\mathscr D ( U, \mathbf R )$,
	\begin{gather*}
	V_2 = 0 \quad \text{if $m=2$}, \qquad \text{$\| \updelta V_1
		\| \leq \| V_1 \| \restrict | \mathbf h ( V_1, \cdot ) | + \|
		V_2 \|$} \quad \text{if $m>2$}, \\
	\text{$\| \updelta V_2 \|$ is absolutely continuous with
		respect to $\| V_2 \|$} \quad \text{if $m>3$}, \\
	\text{$\phi_i$ are associated with $V_1$ as
		in \ref{miniremark:new_situation} for $i \in \{ 1,2 \}$},
	\end{gather*}
	$a \in \spt \| V_1 \|$, $0 < r < \infty$, $\mathbf B (a,r) \subset U$,
	and $( \spt \| V_1 \| ) \without \mathbf U (a,r) \neq \varnothing$.
	
	Then, there holds
	\begin{equation*}
		\Gamma_{\textup{\ref{thm:new_sobolev_inequality}}} (m)^{-1} r
		\leq \| V_1 \| ( \mathbf U (a,r) )^{1/m} + \phi_1 \mathbf U
		(a,r) + \| V_2 \| ( \mathbf U(a,r) )^{1/(m-1)} + \phi_2
		\mathbf U (a,r).
	\end{equation*}
\end{lemma}

\begin{proof}
	In view of \cite[4.6\,(1)]{MR3777387}, \cite[9.2, 9.4]{MR3528825}, and
	\cite[7.4]{arXiv:2209.05955v2}, one may apply
	\ref{thm:new_sobolev_inequality} and
	\ref{remark:sobolev_indecomposable} with $f(x)$ replaced by $\sup \{
	r-|x-a|, 0 \}$.
\end{proof}

We also include a version without boundary with more explicit constants.

\begin{lemma} \label{lemma:classic_swing}
	Suppose $U$, $V$, and $\psi$ are as in \ref{miniremark:situation}, $a
	\in \spt \| V \|$, $0 < r < \infty$, $\mathbf B (a,r) \subset U$, $(
	\spt \| V \|) \without \mathbf U (a,r) \neq \varnothing$, and $V$ is
	indecomposable of type $\mathscr D ( U, \mathbf R )$.
	
	Then,
	\begin{equation*}
		2^{-m-3} m^{-1} \boldsymbol \upgamma (m)^{-1} r \leq \| V \| (
		\mathbf U (a,r) )^{1/m} + \boldsymbol \upgamma (m)^{m-1} \psi
		\, \mathbf U (a,r) .
	\end{equation*}
\end{lemma}

\begin{proof}
	In view of \cite[4.6\,(1)]{MR3777387}, \cite[9.2, 9.4]{MR3528825}, and
	\cite[7.4]{arXiv:2209.05955v2}, one may apply
	\ref{corollary:sobolev_inequality} and \ref{remark:diameter_seminorm}
	with $f(x)$ replaced by $\sup \{ r-|x-a|, 0 \}$.
\end{proof}

\begin{remark}
	If either $m < n$ or $m=n=2$, considering small spheres or small
	disks, respectively, shows that neither the nonemptyness hypothesis
	nor the indecomposability hypothesis may be omitted.
\end{remark}

\begin{remark} \label{remark:classic_swing}
	If $m=1$ and $V$ otherwise is as in \ref{miniremark:situation}, then
	$r \leq \| V \| \, \mathbf U (a,r)$; in fact, the indecomposability
	hypothesis implies $\{ x \with |x-a| = s \} \cap \spt \| V \| \neq
	\varnothing$ for $0 < s < r$, whence the inequality follows, since
	$\mathscr H^1 \restrict \spt \| V \| \leq \| V \|$ by
	\cite[3.5\,(1b)]{MR0307015} and \cite[4.8\,(4)]{MR3528825}.
\end{remark}

Theorem \ref{Thm:varifold-density-lower-bound} is contained
in the first item of the next theorem.  The remaining items discuss---for
special dimensions---slightly more general boundary conditions.

\begin{theorem} \label{thm:few_special_points}
	Suppose $m$ and $n$ are positive inters, $2 \leq m \leq n$, $U$ is an
	open subset of $\mathbf R^n$, $V_1 \in \mathbf V_m ( U )$, $V_2 \in
	\mathbf V_{m-1} ( U )$,
	\begin{gather*}
	\boldsymbol \Uptheta^{\dim V_i} ( \| V_i \|, x ) \geq 1 \quad
	\text{for $\| V_i \|$ almost all $x$ and $i \in \{ 1,2 \}$},
	\\
	\text{$\| \updelta V_1 \|$ is a Radon measure}, \quad
	\text{$\| \updelta V_2 \|$ is a Radon measure},
	\end{gather*}
	$V_1$ is indecomposable of type $\mathscr D ( U, \mathbf R )$,
	and $\lambda = 2^{-4} \boldsymbol \upalpha (2)^{-1} \boldsymbol \upgamma
	(2)^{-2}$.
	
	Then, the following three statements hold.
	\begin{enumerate}
		\item \label{item:few_special_points:p} If $m-1 \leq p < m$,
		$\| \updelta V_1 \| \leq \| V_1 \| \restrict | \mathbf h (
		V_1, \cdot ) | + \| V_2 \|$, $\| \updelta V_2 \|$ is
		absolutely continuous with respect to $\| V_2 \|$, $\mathbf h
		(V_1, \cdot) \in \mathbf L_p^{\mathrm{loc}} ( \| V_1 \|,
		\mathbf R^n )$, and, in case $m>2$, additionally $\mathbf h (
		V_2, \cdot ) \in \mathbf L_{p-1}^{\mathrm{loc}} ( \| V_2 \|,
		\mathbf R^n )$, then
		\begin{equation*}
		\mathscr H^{m-p} \big ( \spt \| V_1 \| \cap \big \{ x
		\with \textup{$\boldsymbol \Uptheta^m_\ast ( \| V_1 \|,
			x ) < 1$ and $\boldsymbol \Uptheta^{m-1}_\ast ( \| V_2
			\|, x ) < 1$} \big \} \big) = 0.
		\end{equation*}
		\item \label{item:few_special_points:3} If $m =3$, $\|
		\updelta V_1 \| \leq \| V_1 \| \restrict | \mathbf h ( V_1,
		\cdot ) | + \| V_2 \|$, $\mathbf h (V_1,\cdot) \in \mathbf
		L_2^{\mathrm{loc}} ( \| V_1 \|, \mathbf R^n )$, and
		\begin{equation*}
		A = \spt \| V_1 \| \cap \big \{ x \with
		\textup{$\boldsymbol \Uptheta^3_\ast ( \| V_1 \|, x ) <
			1$ and $\boldsymbol \Uptheta_\ast^2 ( \| V_2 \|, x ) <
			\lambda$} \big \},
		\end{equation*}
		then
		\begin{equation*}
			\mathscr H^1 \restrict A \leq \sup \big \{ 2^7
			\boldsymbol \upgamma (2)^2, 2^2
			\Gamma_{\textup{\ref{thm:new_sobolev_inequality}}} (3)
			\big \} \| \updelta V_2 \| \restrict A.
		\end{equation*}
		\item \label{item:few_special_points:2} If $m = 2$, $V_2 = 0$,
		and $A = \spt \| V_1 \| \cap \big \{ x \with \boldsymbol
		\Uptheta^2_\ast ( \| V_1 \|, x ) < \lambda \big \}$, then
		\begin{equation*}
		\mathscr H^1 \restrict A \leq 2^8 \boldsymbol \upgamma
		(2)^2 \| \updelta V_1 \| \restrict A.
		\end{equation*}
	\end{enumerate}
	In particular, in all cases, $\mathscr H^m \restrict \spt \| V_1 \|
	\leq \| V_1 \|$.
\end{theorem}

\begin{proof}
	Firstly, we notice that in case of \eqref{item:few_special_points:p}
	we may assume $V_2 = 0$ if $m=2$; in fact, since $\boldsymbol
	\Uptheta^1 ( \| V_2 \|, x) \geq 1$ for $x \in \spt \| V_2 \|$ by
	\cite[4.8\,(4)]{MR3528825}, we may otherwise replace $U$ by $U
	\without \spt \| V_2 \|$ by \cite[7.5]{arXiv:2209.05955v2}.
	
	Secondly, we notice for $i \in \{ 1,2 \}$ that $V_i \in
	\mathbf{RV}_{\dim V_i} ( U )$ by \cite[5.5\,(1)]{MR0307015} and that
	\begin{equation*}
		\| V_i \| = \mathscr H^{\dim V_i} \restrict \boldsymbol
		\Uptheta^{\dim V_i} ( \| V_i \|, \cdot )
	\end{equation*}
	by \cite[3.5\,(1b)]{MR0307015}.  Taking $p = m-1$ in case of
	\eqref{item:few_special_points:3}
	or \eqref{item:few_special_points:2}, we define
	\begin{gather*}
		\text{$\psi_1 = \| V_1 \| \restrict | \mathbf h ( V_1, \cdot )
		|^p$ in case of \eqref{item:few_special_points:p} or
		\eqref{item:few_special_points:3}},  \quad \text{$\psi_1 = \|
		\updelta V_1 \|$ in case of
		\eqref{item:few_special_points:2}}, \\
		\text{$\psi_2 = \| V_2 \| \restrict | \mathbf h ( V_2, \cdot )
		|^{p-1}$ in case of \eqref{item:few_special_points:p}}, \quad
		\text{$\psi_2 = \| \updelta V_2 \|$ in case of
		\eqref{item:few_special_points:3} or
		\eqref{item:few_special_points:2}}.
	\end{gather*}
	Taking $\epsilon = \inf \big \{ 2^{-7} \boldsymbol \upgamma (2)^{-2},
	2^{-2} \Gamma_{\textup{\ref{thm:new_sobolev_inequality}}} (3)^{-1} \big
	\}$ and
	\begin{gather*}
		\text{$\lambda_1 = 1$ in case of
		\eqref{item:few_special_points:p} or
		\eqref{item:few_special_points:3}}, \quad \text{$\lambda_1 =
		\lambda$ in case of \eqref{item:few_special_points:2}}, \\
		\text{$\lambda_2 = 1$ in case of
		\eqref{item:few_special_points:p} or
		\eqref{item:few_special_points:2}}, \quad \text{$\lambda_2 =
		\lambda$ in case of \eqref{item:few_special_points:3}}, \\
		\text{$\delta = 2^{-m} \boldsymbol \upalpha
		(m)^{-1}\Gamma_{\textup{\ref{thm:new_sobolev_inequality}}}
		(m)^{-m}$ in case of \eqref{item:few_special_points:p}
		or \eqref{item:few_special_points:3}}, \quad \text{$\delta =
		0$ in case of \eqref{item:few_special_points:2}}, \\
		\text{$\epsilon_1 = 0$ in case of
		\eqref{item:few_special_points:p} or
		\eqref{item:few_special_points:3}}, \quad \text{$\epsilon_1 =
		2^{-8} \boldsymbol \upgamma (2)^{-2}$ in case of
		\eqref{item:few_special_points:2}}, \\
		\text{$\epsilon_2 = 0$ in case of
		\eqref{item:few_special_points:p} or
		\eqref{item:few_special_points:2}}, \quad \text{$\epsilon_2 =
		\epsilon$ in case of \eqref{item:few_special_points:3}},
	\end{gather*}
	we furthermore define
	\begin{gather*}
	A_1 = \spt \| V_1 \| \cap \big \{ x \with \boldsymbol
	\Uptheta^m_\ast ( \| V_1 \|, x ) < \lambda_1 \big \}, \quad A_2
	= \big \{ x \with \boldsymbol \Uptheta^{m-1}_\ast ( \| V_2 \|, x
	) < \lambda_2 \big \}, \\
	Q_1 = \big \{ x \with \boldsymbol \Uptheta^m_\ast ( \| V_1 \|, x
	) > \delta \big \}, \quad Q_2 = \big \{ x \with \boldsymbol
	\Uptheta^{\ast m-1} ( \| V_2 \|, x ) > 0 \big \}, \\
	X_i = \big \{ x \with \boldsymbol \Uptheta^{\ast m-p} (
	\psi_i, x) > \epsilon_i \big \} \quad \text{for $i \in \{
		1,2 \}$}.
	\end{gather*}
	
	Clearly, we have $X_2 = \varnothing$ in case
	of \eqref{item:few_special_points:2}.  Moreover, we observe that
	\cite[2.10.19\,(3)]{MR41:1976} may be employed (cf.~\cite[p.\,152,
	l.\,9--16]{MR0435361}) to conclude
	\begin{gather*}
	\text{$2^{-8} \boldsymbol \upgamma(2)^{-2} \, \mathscr H^1
		\restrict X_1 \leq \psi_1$ in case
		of \eqref{item:few_special_points:2}}, \quad
	\text{$\epsilon \, \mathscr H^1 \restrict X_2 \leq \psi_2$
		in case of \eqref{item:few_special_points:3}}, \\
	\text{$\mathscr H^{m-p} ( X_1 ) = 0$ in case
		of \eqref{item:few_special_points:p}
		or \eqref{item:few_special_points:3}}, \quad \text{$\mathscr
		H^{m-p} ( X_2 ) = 0$ in case
		of \eqref{item:few_special_points:p}}.
	\end{gather*}
	Clearly, we have $Q_2 = \varnothing$ in case
	of \eqref{item:few_special_points:2}.  Applying \cite[2.10]{MR2537022}
	with $\epsilon$, $\Gamma$, and $s$ replaced by $(2 \boldsymbol
	\upgamma(2))^{-1}$, $2^4 \boldsymbol \upgamma(2)$, and $1$, we obtain
	\begin{equation*}
	\text{$\mathscr H^1 ( A_1 \cap Q_1 \without X_1) = 0$ in case
		of \eqref{item:few_special_points:2}}, \quad
	\text{$\mathscr H^1 ( A_2 \cap Q_2 \without X_2) = 0$ in case
		of \eqref{item:few_special_points:3}}.
	\end{equation*}
	According to \cite[2.11]{MR2537022}, there holds
	\begin{equation*}
	\text{$\mathscr H^{m-p} ( A_2 \cap Q_2 ) = 0$ in case
		of \eqref{item:few_special_points:p}}.
	\end{equation*}
	Moreover, we obtain
	\begin{equation*}
	\text{$A_1 \cap Q_1 \subset X_1 \cup Q_2$ in case of
		\eqref{item:few_special_points:p} and
		\eqref{item:few_special_points:3}};
	\end{equation*}
	in fact, whenever $\sup \{ n, 1/\delta \} \leq M < \infty$ and $a \in
	Q_1 \without ( X_1 \cup Q_2 )$, all sufficiently small $r>0$ satisfy
	\begin{gather*}
		\mathbf B (a,r) \subset U, \qquad \| V_1 \| \, \mathbf B (a,s)
		\geq M^{-1} \boldsymbol \upalpha (m) s^m \quad \text{for $0 <
		s < r $}, \\
		( \boldsymbol \upalpha (m) r^m )^{1-1/p} \psi_1 ( \mathbf U
		(a,r))^{1/p} + \| V_2 \| \, \mathbf U (a,r) \leq
		\Gamma_{\textup{\ref{lemma:density_lower_bound}}} (M)^{-1}
		r^{m-1},
	\end{gather*}
	whence we infer $\| V_1 \| \, \mathbf U (a,r) \geq (1-M^{-1})
	\boldsymbol \upalpha (m) r^m$ by \ref{lemma:density_lower_bound} and
	Hölder's inequality.  Next, we verify
	\begin{equation*}
		\spt \| V_1 \| \subset Q_1 \cup X_1 \cup Q_2 \cup X_2;
	\end{equation*}
	in fact, this follows from
	\ref{lemma:lower_density_ratio_bound} and Hölder's inequality in case
	of \eqref{item:few_special_points:p},
	from \ref{lemma:lower_density_ratio_bound} alone in case
	of \eqref{item:few_special_points:3}, and
	from \ref{lemma:classic_swing} in case
	of \eqref{item:few_special_points:2}.  Therefore, we obtain
	\begin{gather*}
	\text{$A_1 \cap A_2 \subset (A_2 \cap Q_2) \cup X_1 \cup X_2$
		in case of \eqref{item:few_special_points:p} or
		\eqref{item:few_special_points:3}}, \\
	\text{$A_1 \subset ( A_1 \cap Q_1 ) \cup X_1$ in case
		of \eqref{item:few_special_points:2}},
	\end{gather*}
	whence the main conclusion follows.
	
	Since $\lambda_2 \, \mathscr H^{m-1} \restrict U \without A_2 \leq \|
	V_2 \|$ by \cite[2.10.19\,(3)]{MR41:1976} and
	\begin{equation*}
	\mathscr H^m \, \{ x \with 0 < \boldsymbol \Uptheta^m ( \| V_1
	\|, x ) < 1 \} = 0,
	\end{equation*}
	the main conclusion yields $\boldsymbol \Uptheta^m ( \| V_1 \|, x ) \geq
	1$ for $\mathscr H^m$ almost all $x \in \spt \| V_1 \|$ and the
	postscript follows.
\end{proof}

\begin{remark} \label{remark:few_but_not_too_few}
	The exponent of the Hausdorff measure $\mathscr H^{m-p}$ in
	\eqref{item:few_special_points:p} may not be replaced by any smaller
	number determined by $m$ and $p$ even if $V_2 = 0$ by
	\cite[6.1]{arXiv:2209.05955v2} and
	\ref{thm:optimality_exceptional_set}.  Similarly, if $m \geq 3$, the
	hypothesis $m-1 \leq p$ in \eqref{item:few_special_points:p} may not
	be replaced by $m-1-\epsilon \leq p$ for any $0 < \epsilon \leq 1$
	determined by $m$ and $p$ even if $V_2 = 0$ by
	\cite[6.1]{arXiv:2209.05955v2} and \ref{thm:example_p_small}.
\end{remark}

\begin{remark} \label{remark:last_remark}
	The case $m=1$, not treated here, was studied
	in \cite[4.8]{MR3528825}; similarly, results on the case $p = m$ and
	$V_2 = 0$ are summarised in \cite[7.6]{MR3528825}.
\end{remark}

\section{Geodesic diameter} \label{section:diameter_control}

In this section, we establish (see \ref{thm:diameter_bound} and
\ref{corollary:smooth_diameter_bound}) Theorem
\ref{Thm:varifold-diameter-estimate} and Corollary
\ref{Final-corollary:diameter-bound-immersions}.  For this purpose, we
firstly study and characterise the geodesic diameter of
closed subsets of Euclidean space (see
\ref{def:geodesic_distance}--\ref{lemma:diameter}).  Then, we deduce (see
\ref{thm:diameter_bound}--\ref{remark:comparison_Topping}) the bounds on the
geodesic diameter in the varifold setting. As corollaries,
we treat the cases of immersions (see \ref{def:ck_space}--\ref{remark:zemas}),
submanifolds (see
\ref{corollary:poincare-inequality}--\ref{example:cylinder}), and
$\lambda$-minimising currents (see
\ref{example:lambda-minimising-currents}--\ref{remark:Duzaar-Steffen-Fuchs}).

\begin{definition} [see \protect{\cite[6.6]{MR3626845}}]
	\label{def:geodesic_distance}
	Whenever $X$ is a boundedly compact metric space, the \emph{geodesic
	distance} on $X$ is the pseudometric on $X$ whose value at $(a,x) \in
	X \times X$ equals the infimum of the set of numbers
	\begin{equation*}
		\mathbf V_{\inf I}^{\sup I} C
	\end{equation*}
	corresponding to all continuous maps $C : \mathbf R \to X$ such that
	$C ( \inf I ) = a$ and $C ( \sup I ) = x$ for some compact non-empty
	subinterval $I$ of $\mathbf R$.  Moreover, the diameter with respect
	to the geodesic distance is termed \emph{geodesic diameter}.
\end{definition}

\begin{remark} [see \protect{\cite[6.3]{MR3626845}}]
	\label{remark:geodesic_distance}
	The same definition results if one considers maps $C : \{ y \with 0
	\leq y \leq b \} \to X$ with $\Lip C \leq 1$ and $b = \mathbf V_0^b C$
	corresponding to $0 \leq b < \infty$.  (In fact, if it is finite, the
	infimum is attained by some such $C$.)
\end{remark}

\begin{lemma} \label{lemma:diameter}
	Suppose $X$ is a closed subset of $\mathbf R^n$ and $d$ denotes the
	geodesic diameter of $X$.
	
	Then, there holds
	\begin{equation*}
	d = \sup \{ \diam f [X] \with \textup{$0 \leq f \in \mathscr
		D ( \mathbf R^n, \mathbf R )$ and $| \Der f(x) | \leq 1$
		for $x \in X$} \}.
	\end{equation*}
\end{lemma}

\begin{proof}
	In view of \cite[2.9.20]{MR41:1976} and
	\ref{remark:geodesic_distance}, the supremum does not exceed $d$.
	
	To prove the converse inequality, we define pseudometrics
	$\sigma_\delta : X \times X \to \overline{\mathbf R}$ by letting
	$\sigma_\delta (a,x)$, for $(a,x) \in X \times X$ and $0 < \delta \leq
	1$, denote the infimum of the set of numbers
	\begin{equation*}
	\sum_{i=1}^j |x_i-x_{i-1}|
	\end{equation*}
	corresponding to all finite sequences $x_0, x_1, \ldots, x_j \in X$
	with $x_0 = a$, $x_j = x$, and $|x_i-x_{i-1}| \leq \delta$ for $i = 1,
	\ldots, j$.  One readily verifies that
	\begin{equation*}
	\sigma_\delta ( \chi, a ) \leq \sigma_\delta (x,a) + |x-\chi|
	\quad \text{whenever $a,x,\chi \in X$ and $|x-\chi| \leq
		\delta$};
	\end{equation*}
	in particular, $\Lip ( \sigma_\delta (\cdot,a) | \mathbf B
	(x,\delta) ) \leq 1$ in case $\sigma_\delta (a,x) < \infty$.  Denoting
	by $\varrho$ the geodesic distance on $X$, we have
	\begin{equation*}
	|a-x| \leq \sigma_\delta (a,x) \leq \varrho (a,x) \quad
	\text{and} \quad \lim_{\delta \to 0+} \sigma_\delta (a,x) =
	\varrho (a,x) \quad \text{for $a,x \in X$}
	\end{equation*}
	by \cite[6.3]{MR3626845}.  Consequently, one readily verifies%
	\begin{footnote}
		{In fact, as $\sigma_\delta$ is real valued in case $d <
			\infty$, we have
			\begin{equation*}
			d = \sup \{ \diam \im \sup \{ s - \sigma_\delta
			(\cdot,a),0 \} \with \text{$0 \leq s < \infty$, $0 <
				\delta \leq 1$, and $a \in X$} \}.
			\end{equation*}}
	\end{footnote}%
	\begin{equation*}
	d \leq \sup \{ \diam \im \sup \{ s - \sigma_\delta (\cdot,a),0
	\} \with \text{$0 \leq s < \infty$, $0 < \delta \leq 1$, and
		$a \in X$} \}.
	\end{equation*}
	
	This estimate implies that the conclusion is a consequence of the
	following assertion: if $\epsilon > 0$, $0 \leq s < \infty$,
	$0 < \delta \leq 1$, $a \in X$, and $\zeta = \sup \{ s - \sigma_\delta
	( \cdot, a ), 0 \}$, then there exists a nonnegative function $Z \in
	\mathscr D ( \mathbf R^n , \mathbf R)$ such that
	\begin{equation*}
	| Z(x)-\zeta(x) | \leq \epsilon \quad \text{and} \quad |
	\Der Z(x) | \leq 1 \quad \text{whenever $x
		\in X$}.
	\end{equation*}
	
	To prove this assertion, we first observe that $\zeta$ is a real
	valued function with $\Lip ( \zeta | \mathbf B (x,\delta) ) \leq 1$
	for $x \in X$.  Moreover, since $\sup \im \zeta < \infty$, it is
	sufficient to prove the assertion with $| \Der Z(x) | \leq 1$ replaced
	by $| \Der Z(x) | \leq 1+\epsilon$.  For this purpose, we will
	employ the partition of unity given in \cite[3.1.13]{MR41:1976}; in
	particular, let $V_1$ be the number constructed there, $\kappa = \sup
	\{ 1, V_1 \}$, and define
	\begin{gather*}
	\Phi = \{ \mathbf U ( \chi, \delta ) \with \chi \in X \cap
	\mathbf U (a,s+\delta) \} \cup \{ \mathbf R^n \without \mathbf
	B (a,s) \}, \quad U = {\textstyle \bigcup \Phi}, \\
	h(x) = {\textstyle \frac 1{20}} \sup \big \{ \inf \{ \dist (x,
	\mathbf R^n \without T), 1 \} \with T \in \Phi \big \} \quad
	\text{for $x \in U$}.
	\end{gather*}
	Employing a Lipschitzian extension (see \cite[2.10.44]{MR41:1976}) and
	convolution, we construct, for each $T \in \Phi$, a nonnegative
	function $g_T \in \mathscr E ( \mathbf R^n, \mathbf R )$ satisfying
	\begin{equation*}
	| g_T (x) - \zeta (x) | \leq (129)^{-n} (20\kappa)^{-1} \delta
	\epsilon \quad \text{for $x \in T$}, \qquad | \Der g_T (x)
	| \leq 1 \quad \text{for $x \in \mathbf R^n$},
	\end{equation*}
	where we may assume that $g_T = 0$ if $T = \mathbf R^n \without
	\mathbf B (a,s)$, since $\spt \zeta \subset \mathbf B (a,s)$.  Taking
	$S$, $S_x$, and $v_s$, for $s \in S$, as in \cite[3.1.13]{MR41:1976}
	and choosing $\tau : S \to \Phi$ such that $\spt v_s \subset \tau (s)$
	for $s \in S$, we define $G = \sum_{s \in S} v_s g_{\tau (s)}$.
	Clearly, we have $|G(x)-\zeta(x)| \leq \epsilon$ for $x \in X$ and
	$\spt G \subset \mathbf U (a,s+2\delta)$.  Noting $h(x) \geq \frac
	\delta{20}$ for $x \in X$, we furthermore estimate
	\begin{equation*}
		\left | \sum_{s \in S} \Der v_s (x) g_{\tau(s)} (x) \right |
		\leq \sum_{s \in S_x} | \Der v_s (x) | | g_{\tau(s)} (x) -
		\zeta (x) | \leq \epsilon, \quad | \Der G(x) | \leq 1 +
		\epsilon
	\end{equation*}
	for $x \in X$.  Applying \cite[3.16]{arXiv:2206.14046v2} with $U$,
	$E_0$, and $E_1$ replaced by $\mathbf R^n$, $\mathbf R^n \without U$,
	and $X$ to obtain $f$ with the properties listed there, we may
	take $Z \in \mathscr D ( \mathbf R^n, \mathbf R )$ defined by $Z(x) =
	f(x) G(x)$ for $x \in U$ and $Z(x) = 0$ for $x \in \mathbf R^n
	\without U$.
\end{proof}

Next, we turn to Theorem
\ref{Thm:varifold-diameter-estimate}, the general a priori estimate of the
geodesic diameter in the varifold setting with boundary.

\begin{theorem} \label{thm:diameter_bound}
	Suppose $m$ and $n$ are integers, $2 \leq m \leq n$, $V_1 \in \mathbf
	V_m ( \mathbf R^n )$ and $V_2 \in \mathbf V_{m-1} ( \mathbf R^n )$
	satisfy the conditions of \ref{miniremark:new_situation} with $U =
	\mathbf R^n$, $V_1$ is indecomposable of type $\mathscr D ( \mathbf
	R^n, \mathbf R )$, $( \| V_1 \| + \| V_2 \| ) ( \mathbf R^n ) <
	\infty$,
	\begin{gather*}
	V_2 = 0 \quad \text{if $m=2$}, \qquad \| \updelta V_1 \| \leq
	\| V_1 \| \restrict | \mathbf h ( V_1, \cdot ) | + \| V_2 \|
	\quad \text{if $m>2$}, \\
	\text{$\| \updelta V_2 \|$ is absolutely continuous with
		respect to $\| V_2 \|$} \quad \text{if $m > 3$},
	\end{gather*}
	$\phi_i$ are associated with $V_i$ as in \ref{miniremark:new_situation},
	for $i \in \{ 1,2 \}$, and $d$ denotes the geodesic
	diameter of $\spt \| V_1 \|$.
	
	Then, there holds, for some positive finite number $\Gamma$ determined
	by $m$,
	\begin{equation*}
	d \leq \Gamma ( \phi_1 + \phi_2 ) ( \mathbf R^n ).
	\end{equation*}
\end{theorem}

\begin{proof}
	The isoperimetric inequality
	and Hölder's inequality yield
	\begin{equation*}
	\| V_2 \| ( \mathbf R^n )^{1/(m-1)} \leq \boldsymbol \upgamma
	(m-1)^{m-2} \, \phi_2 ( \mathbf R^n ).
	\end{equation*}
	We will show
	\begin{equation*}
	\| V_1 \| ( \mathbf R^n )^{1/m} \leq (2 \boldsymbol \upgamma (m)
	)^{m-1} \, \phi_1 ( \mathbf R^n ) + ( 2 \boldsymbol \upgamma (m)
	)^{1/(m-1)} \boldsymbol \upgamma (m-1)^{m-2} \, \phi_2 ( \mathbf
	R^n);
	\end{equation*}
	in fact, if $m=2$, then $\| V_1 \| ( \mathbf R^n )^{1/2} \leq
	\boldsymbol \upgamma (2) \, \phi_1 ( \mathbf R^n )$ by the isoperimetric
	inequality, and, if $m > 2$, then we may assume $\| V_1 \| ( \mathbf
	R^n)^{1-1/m} > 2 \boldsymbol \upgamma (m) \, \| V_2 \| ( \mathbf R^n )$,
	in which case the isoperimetric inequality may be used to obtain
	\begin{equation*}
	{\textstyle \| V_1 \| ( \mathbf R^n )^{1-1/m} \leq 2
		\boldsymbol \upgamma (m) \int | \mathbf h (V_1,x)| \ud \| V_1 \|
		\, x},
	\end{equation*}
	whence the asserted inequality follows by Hölder's inequality.
	
	Next, suppose $X = \spt \| V_1 \|$ and $f$ satisfies the conditions
	of \ref{lemma:diameter}.  Then, $f \in \mathbf T_{\varnothing} ( V_i)$
	and $\| V_i \|_{(\infty)} ( V_i \, \mathbf D f ) \leq 1$ for $i \in \{
	1,2 \}$ by
	\cite[4.6\,(1)]{MR3777387} and \cite[9.2]{MR3528825}.
	Hence,
	\ref{thm:new_sobolev_inequality} and \ref{remark:sobolev_indecomposable}
	yield
	\begin{equation*}
	\diam \spt f_\# \| V_1 \| \leq \Delta ( \phi_1 + \phi_2 )
	( \mathbf R^n ),
	\end{equation*}
	where $\Delta = \Gamma_{\textup{\ref{thm:new_sobolev_inequality}}} (m)
	\big ( 1 + \boldsymbol \upgamma (m-1)^{m-2} ( 1 + ( 2 \boldsymbol
	\upgamma (m))^{1/(m-1)}) + ( 2 \boldsymbol \upgamma (m))^{m-1} \big
	)$.  Finally, we notice $f [X] \subset \spt f_\# \| V_1 \|$ as $f$ is
	continuous.%
	\begin{footnote}
		{In fact, as $f$ is closed, we have $f[X] = \spt f_\# \| V_1
			\|$.}
	\end{footnote}%
\end{proof}

\begin{remark} \label{remark:fourth_question}
	The preceding theorem answers the fourth question posed
	in \cite[Section A]{scharrer:MSc}.
\end{remark}

\begin{remark} \label{remark:clamped-Willmore-II}
	By \cite[10.21]{arXiv:2209.05955v2}, integral varifolds satisfying the
	hypotheses with $m = 2$ and $n = 3$ occur in the minimisation of the
	Willmore energy with \emph{clamped boundary condition} amongst
	connected surfaces; see \cite[Theorem 4.1]{MR4141858}.
\end{remark}

In the case without boundary, somewhat more explicit constants may be obtained
by using \ref{corollary:sobolev_inequality} instead
of \ref{thm:new_sobolev_inequality}.

\begin{corollary}
	Suppose $V$ and $\psi$ are as in \ref{miniremark:situation} with $U =
	\mathbf R^n$, $\| V \| ( \mathbf R^n ) < \infty$, $V$ is
	indecomposable of type $\mathscr D ( \mathbf R^n, \mathbf R )$, $d$
	denotes the geodesic diameter of $\spt \| V \|$, and
	$\Gamma = 2^{m+4} m \boldsymbol \upgamma ( m )^m$.
	
	Then, there holds
	\begin{equation*}
		d \leq \Gamma \, \psi ( \mathbf R^n ).
	\end{equation*}
\end{corollary}

\begin{proof}
	With a possibly larger number $\Gamma$, this follows from
	\ref{thm:diameter_bound} with $V_1 = V$  and $V_2 = 0$.  We verify the
	eligibility of the present number $\Gamma$ by noting that
	\begin{equation*}
	\diam \spt f_\# \| V \| \leq 2^{m+3} m \boldsymbol
	\upgamma (m) \big ( \| V \| ( \mathbf R^n )^{1/m} + \boldsymbol
	\upgamma (m)^{m-1} \, \psi ( \mathbf R^n ) \big ) \leq \Gamma \,
	\psi ( \mathbf R^n )
	\end{equation*}
	by \ref{corollary:sobolev_inequality} in conjunction with the
	isoperimetric inequality and Hölder's inequality, whenever
	$f$ satisfies the conditions of \ref{lemma:diameter} with $X = \spt \|
	V \|$.
\end{proof}

\begin{remark} \label{remark:two-convex-MCF}
	Each component (see \cite[6.12]{MR3528825}
	and \cite[5.1, 7.2]{arXiv:2209.05955v2}) of a
	varifold occurring in the \emph{level-set mean curvature flow of
	two-convex submanifolds} of dimension $m$ in $\mathbf R^{m+1}$
	satisfies the hypotheses, see \cite[Corollary 1.1]{MR4176547}.
\end{remark}

\begin{remark} \label{remark:comparison_Topping}
	Here, we compare our proof with that of P.\
	Topping's analogous result for immersions (see \cite[Theorem
	1.1]{MR2410779}).  The principal geometric
	idea---suitable smallness of mean curvature implies
	lower density ratio bounds in balls---is the same.  Our formulation
	(see \ref{lemma:density_ratio} and \ref{lemma:classic_swing}) may be
	traced back to \cite[8.3]{MR0307015}, whereas
	his formulation (see \cite[Lemma 1.2]{MR2410779})
	was inspired by his local non-collapsing result for
	Ricci flow (see \cite[Theorem 4.2]{MR2216151}).
	To implement this geometric idea for varifolds, one
	faces the difficulty that one cannot---a
	priori---assume the existence of either geodesics or lower density
	bounds; the latter are employed to obtain \cite[Lemma 1.2]{MR2410779}.
	Instead, our proof avoids these tools---though, lower
	density bounds are independently proven in
	\ref{thm:few_special_points}---and proceeds through the
	characterisation of geodesic diameter (see \ref{lemma:diameter}) and
	the Sobolev-Poincaré inequality (see \ref{thm:new_sobolev_inequality})
	in conjunction with basic properties from the study of
	indecomposability (see \cite[7.12]{arXiv:2209.05955v2}).
\end{remark}

To prepare for the use of the Whitney-type approximation results
in \ref{corollary:smooth_diameter_bound}, we firstly define the appropriate
topological function space.

\begin{definition} \label{def:ck_space}
	Suppose $k$ is a positive integer, $M$ is a compact
	manifold-with-boundary of class $k$, and $Y$ is a Banach space.
	
	Then, $\mathscr C^k (M, Y )$ is defined (see \cite[2.4]{MR3626845}) to
	be the locally convex space of all maps from $M$ into $Y$ of class $k$
	topologised by the family of all seminorms, that correspond to charts
	$\phi$ of $M$ of class $k$ and compact subsets $K$ of $\dmn \phi$, and
	have value
	\begin{equation*}
		\sup \big ( \{ 0 \} \cup \big \{ \| \Der^l ( F \circ
		\phi^{-1}) (x) \| \with x \in K \without \phi [ \partial M ],
		l = 0, \ldots, k \big \} \big )
	\end{equation*}
	at $F \in \mathscr C^k (M,Y)$.
\end{definition}

\begin{remark} \label{remark:hirsch}
	Choosing a positive integer $\iota$ and charts $\phi_i$ of $M$ of
	class $k$ with compact subsets $K_i$ of $\dmn \phi_i$, for $i = 1,
	\ldots, \iota$, satisfying $M = \bigcup_{i=1}^\iota \Int K_i$, the
	topology of the locally convex space $\mathscr C^k (M,Y)$ is induced
	by the norm $\nu$ on $\mathscr C^k ( M,Y )$ whose value at $F \in
	\mathscr C^k (M,Y)$ equals
	\begin{equation*}
		\sup \big ( \{ 0 \} \cup \big \{ \| \Der^j ( F \circ
		\phi_i^{-1}) (x) \| \with \text{$x \in K_i \without \phi_i [
		\partial M]$, $i= 0, \ldots, \iota$, $j = 0, \ldots, k$} \big
		\} \big);
	\end{equation*}
	in fact, each seminorm occurring in \ref{def:ck_space} is bounded by a
	finite multiple of $\nu$ by the general formula for the differentials
	of a composition, see \cite[3.1.11]{MR41:1976}.  Similarly, we see
	that the topology of $\mathscr C^k ( M, \mathbf R^n )$ agrees with
	that of the space named ``$C_W^k ( M, \mathbf R^n )$'' in
	\cite[p.\,35]{MR1336822}.  Consequently, if $n > 2\dim M$, the set of
	embeddings of $M$ into $\mathbf R^n$ of class $2$ is dense
	in $\mathscr C^2 ( M, \mathbf R^n )$ by \cite[2.1.0]{MR1336822}.
\end{remark}

Next, we present Corollary
\ref{Final-corollary:diameter-bound-immersions}, the a priori estimate of the
geodesic diameter of immersions of compact manifolds-with-boundary.

\begin{corollary} \label{corollary:smooth_diameter_bound}
	Suppose $m$ and $n$ are positive integers, $2 \leq m \leq n$, $M$ is
	a compact connected $m$ dimensional manifold-with-boundary of
	class $2$, the map $F : M \to \mathbf R^n$ is an immersion
	of class $2$, $g$ is the Riemannian metric on $M$ induced
	by $F$, and $\sigma$ is the Riemannian distance associated
	with $(M,g)$.
	
	Then, there holds
	\begin{equation*}
		\diam_\sigma M \leq \Gamma_{\textup{\ref{thm:diameter_bound}}}
		(m) \big ( {\textstyle\int_M} | \mathbf h ( F, \cdot ) |^{m-1}
		\ud \mathscr H^m_\sigma + {\textstyle\int_{\partial M}} |
		\mathbf h ( F | \partial M, \cdot ) |^{m-2} \ud \mathscr
		H^{m-1}_\sigma \big );
	\end{equation*}
	here $0^0 = 1$.
\end{corollary}

\begin{proof}
	First, the \emph{special case}, that $F$ is an embedding, will be
	treated; in this case, $F$ induces an isometry between
	$\sigma$ and the geodesic distance on $F[M]$ by \cite[2.10.13,
	3.2.3\,(1)]{MR41:1976} and \ref{remark:geodesic_distance}. We define
	$V_1 \in \mathbf V_m ( \mathbf R^n )$ to be associated with
	$(F,\mathbf R^n)$ and $V_2 \in \mathbf V_{m-1} ( \mathbf R^n )$ to be
	$0$ if $m=2$ and to be associated with $(F| \partial M, \mathbf R^n
	)$ if $m>2$.  Hence, \cite[6.14, 10.3]{arXiv:2209.05955v2} yield
	\begin{gather*}
		\boldsymbol \Uptheta^{\dim V_i} ( \| V_i \|, x ) \geq 1 \quad
		\text{for $\| V_i \|$ almost all $x$ and $i \in \{ 1,2 \}$},
		\\
		\| V_1 \| = F_\# \mathscr H^m_\sigma, \quad \| \updelta V_1 \|
		= \| V_1 \| \restrict | \mathbf h ( F[M \without \partial M],
		\cdot ) | + F_\# ( \mathscr H^{m-1}_\sigma \restrict \partial
		M ), \\
		\text{$\| V_2 \| = F_\# ( \mathscr H^{m-1}_\sigma \restrict
		\partial M)$ if $m>2$}, \quad \text{$\| \updelta V_2 \| = \|
		V_2 \| \restrict | \mathbf h ( F[ \partial M], \cdot ) |$ if
		$m>2$}.
	\end{gather*}
	Since $V_1$ is indecomposable of type $\mathscr D ( U, \mathbf R )$ by
	\cite[7.9\,(1)\,(4)]{arXiv:2209.05955v2}, the special case now follows
	from \ref{thm:diameter_bound}.

	In the general case, we assume $n > 2m$ and obtain from
	\ref{remark:hirsch} a sequence of embeddings $F_i : M \to \mathbf R^n$
	of class $2$ converging to $F$ in $\mathscr C^2 ( M, \mathbf R^n)$ as
	$i \to \infty$; in particular, $\mathbf h(F_i, x) \to \mathbf h
	(F,x)$, uniformly for $x \in M$, as $i \to \infty$ by
	\cite[6.11]{arXiv:2209.05955v2} and \ref{remark:hirsch}.  Moreover,
	denoting by $g_i$ the Riemannian metrics on $M$ induced by $F_i$ and
	by $\sigma_i$ the Riemannian distance of $(M,g_i)$, we observe that,
	given $1 < \lambda < \infty$, all sufficiently large $i$ satisfy
	\begin{equation*}
		\lambda^{-2} \langle (w,w), g (z) \rangle \leq \langle (w,w),
		g_i (z) \rangle \leq \lambda^2 \langle (w,w), g(z) \rangle
	\end{equation*}
	whenever $z \in M$ and $w$ belongs to the tangent space of $M$ at $z$,
	whence we infer $\lambda^{-1} \sigma \leq \sigma_i \leq \lambda
	\sigma$ and $\lambda^{-k} \mathscr H^k_\sigma \leq \mathscr
	H^k_{\sigma_i} \leq \lambda^k \mathscr H^k_\sigma$ for $0 \leq k <
	\infty$.  Therefore, the conclusion follows from the special
	case applied with $F$ replaced by $F_i$.
\end{proof}

\begin{remark}
	In the case $m = 2$ with $\diam_\sigma f[M]$ replaced by $\diam f[M]$
	a better constant is obtained in \cite[Theorem
	1.1]{MR4553536}.
\end{remark}

\begin{remark} \label{remark:zemas}
	The integral $\int_{\partial M} | \mathbf h ( F|\partial M, \cdot)
	|^{m-2} \ud \mathscr H_\sigma^{m-1}$ equals $\mathscr H_\sigma^1 (
	\partial M )$ if $m = 2$ but it may not be replaced by $\mathscr
	H_\sigma^{m-1} ( \partial M )^{1/(m-1)}$ if $m>2$; 
	it suffices to take $m=n$.
\end{remark}

Turning to the case of submanifolds, we state the immediate
corollary and then discuss the impossiblity of a seemingly natural sharpening
of the result.

\begin{corollary} \label{corollary:poincare-inequality}
	Suppose $m$ and $n$ are integers, $2 \leq m \leq n$, $M$ is a compact
	connected $m$ dimensional submanifold-with-boundary of class $2$ of
	$\mathbf R^n$, $f : M \to \mathbf R$ is of class $1$ relative to $M$,
	and $\kappa = \sup \{ | \Der f(x) | \with x \in M \without \partial M
	\}$.

	Then, there holds (here, $0^0=0$)
	\begin{equation*}
		\diam f[M] \leq \Gamma_{\textup{\ref{thm:diameter_bound}}} (m)
		\big ( {\textstyle\int_M} | \mathbf h ( M,\cdot ) |^{m-1} \ud
		\mathscr H^m + {\textstyle\int_{\partial M}} | \mathbf h (
		\partial M, \cdot ) |^{m-2} \ud \mathscr H^{m-1} \big )
		\kappa.
	\end{equation*}
\end{corollary}

\begin{proof}
	We combine \cite[2.9.20]{MR41:1976}, \cite[10.3]{arXiv:2209.05955v2},
	and \ref{corollary:smooth_diameter_bound}.
\end{proof}

\begin{remark} \label{remark:poincare-inequality}
	In contrast to other Poincaré inequalities, the derivative $\Der f$
	may not be measured with respect to $( \mathscr H^m \restrict M
	)_{(q)}$ for any $m < q < \infty$, see \ref{example:cylinder}.
\end{remark}

\begin{example} \label{example:cylinder}
	Whenever $m$ is an integer and $2 \leq m < q < \infty$, the infimum of
	the set of numbers
	\begin{equation*}
		\gamma (M)^{1-m/q} \cdot \big ( {\textstyle\int_M} | \Der f
		|^q \ud \mathscr H^m \big)^{1/q},
	\end{equation*}
	where $\gamma(M) = {\textstyle\int_M} | \mathbf h ( M,z ) |^{m-1} \ud
	\mathscr H^m \, z + {\textstyle\int_{\partial M}} | \mathbf h
	(\partial M,z)|^{m-2} \ud \mathscr H^{m-1} \, z$, corresponding to all
	compact connected $m$ dimensional submanifolds-with-boundary of class
	$\infty$ of $\mathbf R^{m+1}$ and functions $f : M \to \mathbf R$ of 
	class $1$ relative to $M$ with $\diam f[M]=1$, equals $0$; in fact,
	the same holds if we require $\partial M = \varnothing$.
\end{example}

\begin{proof}
	Whenever $1 \leq r < \infty$, we consider
	\begin{equation*}
		M = \mathbf S^{m-1} \times \{ y \with 0 \leq y \leq r \}
		\subset \mathbf R^m \times \mathbf R \simeq \mathbf R^{m+1},
	\end{equation*}
	and $f : M \to \mathbf R$ such that $f(x,y) = y/r$ for $(x,y) \in M$;
	hence, $\diam f[M] = 1$.  Noting $| \Der f(x,y)|=1/r$ for $(x,y) \in
	M \without \partial M$, we use \cite[3.2.23]{MR41:1976} to compute
	\begin{align*}
		{\textstyle\int_M} | \mathbf h (M, \cdot )|^{m-1} \ud \mathscr
		H^m & = \mathscr H^{m-1} ( \mathbf S^{m-1} ) (m-1)^{m-1} r, \\
		{\textstyle\int_{\partial M}} | \mathbf h (\partial M, \cdot
		)|^{m-2} \ud \mathscr H^{m-1} & = 2 \mathscr H^{m-1} ( \mathbf
		S^{m-1} ) (m-1)^{m-2}, \\
		{\textstyle\int_M} | \Der f |^q \ud \mathscr H^m & = \mathscr
		H^{m-1} ( \mathbf S^{m-1} ) r^{1-q}.
	\end{align*}
	The principal assertion follows and the postscript may be obtained by
	adding half-spheres with boundary $\partial M$ to the cylinder $M$
	and approximation.
\end{proof}

Finally, we shall introduce the setting of
$\lambda$-minimising currents, apply our general diameter estimate to it, and
discuss its significance.

\begin{example} \label{example:lambda-minimising-currents}
	Suppose $m$ and $n$ are integers, $2 \leq m \leq n$, $0 \leq \lambda <
	\infty$, and, following \cite{MR1243155}, the integral current $Q \in
	\mathbf I_m ( \mathbf R^n )$ satisfies
	\begin{equation*}
		\| Q \| ( \mathbf R^n ) \leq \| Q + \boundary R \| ( \mathbf
		R^n ) + \lambda \| R \| ( \mathbf R^n ) \quad \text{whenever
		$R \in \mathbf I_{m+1} ( \mathbf R^n )$}.
	\end{equation*}
	Noting \cite[4.1.28]{MR41:1976} and \cite[3.5\,(1c)]{MR0307015}, we
	take $V \in \mathbf{IV}_m ( \mathbf R^n )$ characterised by $\| V \| =
	\| Q \|$.  Then,
	\begin{equation*}
		\| \updelta V \| \leq \lambda \| V \| + \| \boundary Q \|
	\end{equation*}
	by \cite[(2.3)]{MR1243155}, hence, noting $\| \boundary Q \|_{\| Q \|}
	= 0$ by \cite[4.1.28]{MR41:1976}, we obtain
	\begin{equation*}
		\| \updelta V \|_{\| V \|} \leq \lambda \| V \|, \quad \|
		\updelta V \| - \| \updelta V \|_{\| V \|} \leq \| \boundary Q
		\|;
	\end{equation*}
	in fact, \cite[2.9.2]{MR41:1976} applied with $\psi = \| \boundary Q
	\|$ and $\phi = \| V \|$ yields a Borel set $B$ such that
	\begin{equation*}
		\| \boundary Q \| ( B ) = 0 \quad \text{and} \quad \| V \| (
		\mathbf R^n \without B ) = 0,
	\end{equation*}
	whence we infer $\| \updelta V \|_{\| V \|} \leq \| \updelta V \|
	\restrict B \leq \lambda \| V \|$ and thus $\| \updelta V \| \restrict
	B \leq \| \updelta V \|_{\| V \|}$.  Recalling $\| \updelta V \|_{\| V
	\|} = \| V \| \restrict | \mathbf h (V,\cdot) |$ from
	\cite[3.21]{arXiv:2209.05955v2}, we deduce
	\begin{equation*}
		\| V \|_{(\infty)} ( \mathbf h (V,\cdot) ) \leq \lambda \quad
		\text{and} \quad \| \updelta V \| \leq \| V \| \restrict |
		\mathbf h (V,\cdot) | + \| \boundary Q \|.
	\end{equation*}
\end{example}

\begin{corollary} \label{corollary:diameter_bound_lambda_minimising_currents}
	Suppose $m$ and $n$ are integers, $2 \leq m \leq n$, $0 \leq \lambda <
	\infty$, $Q \in \mathbf I_m (\mathbf R^n)$ is indecomposable,
	\begin{equation*}
		\| Q \|( \mathbf R^n ) \leq \| Q + \boundary R \| ( \mathbf
		R^n ) + \lambda \| R \| ( \mathbf R^n ) \quad \text{for $R \in
		\mathbf I_{m+1} ( \mathbf R^n )$},
	\end{equation*}
	$d$ denotes the geodesic diameter of the set $\spt Q$,
	and $\gamma = \Gamma_{\textup{\ref{thm:diameter_bound}}} (m)$; if $m >
	2$, then suppose $W \in \mathbf V_{m-1} ( \mathbf R^n )$, $\| \updelta
	W \|$ is a Radon measure, $\boldsymbol \Uptheta^{m-1} (\| W \|, x)
	\geq 1$ for $\| W \|$ almost all $x$, and $\| \boundary Q \| \leq \| W
	\|$; if $m > 3$, then suppose $\| \updelta W \|$ is absolutely
	continuous with respect to $\| W \|$ and $\mathbf h (W, \cdot) \in
	\mathbf L_{m-2}^{\textup{loc}} ( \| W \|, \mathbf R^n )$.

	Then, $\mathscr H^m \restrict \spt Q \leq \| Q \|$, the set $\spt Q$
	is approximately differentiable of order $2$ at $\| Q \|$ almost all
	$a$, and the following three statements hold.
	\begin{enumerate}
		\item If $m = 2$, then $d \leq \gamma \big ( \int | \ap
		\mathbf h ( \spt Q, \cdot ) | \ud \| Q \| + \| \boundary Q \|
		( \mathbf R^n ) \big )$.
		\item If $m = 3$, then $d \leq \gamma \big ( \int | \ap
		\mathbf h ( \spt Q, \cdot ) |^2 \ud \| Q \| + \| \updelta W \|
		( \mathbf R^n ) \big )$.
		\item If $m > 3$, then $d \leq \gamma \big ( \int | \ap
		\mathbf h ( \spt Q, \cdot ) |^{m-1} \ud \| Q \| + \int |
		\mathbf h (W,\cdot) |^{m-2} \ud \| W \| \big )$.
	\end{enumerate}
\end{corollary}

\begin{proof}
	We associate $V \in \mathbf{IV}_m ( \mathbf R^n )$ with $Q$ as in
	\ref{example:lambda-minimising-currents}.  Noting
	\cite[5.1]{arXiv:2206.14046v2} and \cite[6.7]{arXiv:2209.05955v2}, the
	varifold $V$ is indecomposable of type $\mathscr D ( \mathbf R^n,
	\mathbf R )$ by \cite[10.9]{arXiv:2209.05955v2}. Therefore, applying
	\ref{thm:few_special_points} and \ref{thm:diameter_bound} with $(V_1,
	V_2 )$ replaced by $(V,0)$ if $m =2$ and $(V,W)$ if $m>2$ yields
	$\mathscr H^m \restrict \spt Q \leq \| Q \|$ and that the last three
	conclusions hold with $\ap \mathbf h ( \spt Q, \cdot )$ replaced by
	$\mathbf h (V,\cdot)$.  Finally, \cite[4.8]{MR3023856}, in conjunction
	with \cite[2.10.19\,(4)]{MR41:1976} and \cite[3.22]{MR3978264}, shows
	that the set $\spt Q$ is approximately differentiable of order $2$
	with $\ap \mathbf h ( \spt Q, a ) = \mathbf h (V,a)$ at $\| Q \|$
	almost all $a$.
\end{proof}

\begin{remark} \label{remark:diameter_bound_lambda_minimising_currents}
	For any $Q \in \mathbf I_m ( \mathbf R^n )$ satisfying the condition
	\begin{quote}
		``There exists no $R \in \mathbf I_m ( \mathbf R^n )$ with
		$\boundary R = 0$ such that $R \neq 0 \neq Q-R$ and $\| R \| +
		\| Q-R \| = \| Q \|$.''
	\end{quote}
	indecomposability of $\boundary Q$ implies that of $Q$.  For $Q$ as in
	\ref{example:lambda-minimising-currents}, this condition is guaranteed
	by
	\begin{equation*}
		\lambda^m \| Q \| ( \mathbf R^n ) \leq \boldsymbol \alpha
		(m+1) (m+1)^{m+1}
	\end{equation*}
	via the optimal isoperimetric inequality for integral currents of
	\cite[\S\,10]{MR855173}.
\end{remark}

\begin{remark} \label{remark:Duzaar-Steffen-Fuchs}
	Both \cite{MR1243155} and the preceding theorem are tailored to apply
	to the integral currents with prescribed mean curvature vector
	studied in \cite{MR1037996,MR1200740,MR1220033} and to the
	codimension-one area minimising integral currents with prescribed
	volume and boundary constructed in \cite{MR1156438}.  The condition
	discussed in \ref{remark:diameter_bound_lambda_minimising_currents} is
	particularly natural in this context as exemplified by \cite[1.2,
	1.3]{MR1156438} and \cite[2.3]{MR1220033}; the same holds for the mass
	bound guaranteeing it in view of \cite[6.1]{MR1037996}.  For $n-m=1$,
	its usage to obtain connectedness of the regular part of $\spt Q$
	appears in \cite[p.\,358]{MR1383909}.
\end{remark}

\section{Plateau problems}

In this section, we demonstrate how to apply our geodesic
diameter estimate (see \ref{thm:diameter_bound}) to solutions of Plateau
problems.  Firstly, we consider the setting of integral chains with
coefficients in a complete normed commutative group (see
\ref{lemma:first-variation-minimiser}--\ref{remark:nonexistence}); in
particular, we obtain Theorem \ref{Thm:Plateau:chains}.  Then, we study (see
\ref{lemma:density_of_variation_II}%
--\ref{corollary:Reifenberg-Plateau-diameter-sets}) natural classes of
varifolds with conditions on their first variation away from the boundary;
this results (see \ref{thm:estimate_variation_boundary} and
\ref{thm:Reifenberg-Plateau-diameter}) in Theorems \ref{FinalThm:geom-var-I}
and \ref{FinalThm:geom-var-II}.  Drawing from the literature, our theory
finally becomes applicable to the Plateau problem for sets based on Čech
homology yielding (see \ref{remark:Reifenberg-Plateau}) Theorem
\ref{FinalThm:Reifenberg-Plateau}.

\begin{lemma} [classical]
\label{lemma:first-variation-minimiser}
	Suppose $m$ and $n$ are integers, $1 \leq m \leq n$, $G$ is a complete
	normed commutative group, $S \in \mathbf I_m ( \mathbf R^n, G )$,
	\begin{equation*}
		\| S \| ( \mathbf R^n ) \leq \| T \| ( \mathbf R^n ) \quad
		\text{whenever $T \in \mathbf I_m ( \mathbf R^n, G )$ and
		$\boundary_G T = \boundary_G S$},
	\end{equation*}
	and $V \in \mathbf{RV}_m ( \mathbf R^n )$ is characterised by $\| V \|
	= \| S \|$.

	Then, there holds
	\begin{equation*}
		| ( \updelta V ) ( g ) | \leq {\textstyle\int} | \Nor^{m-1} (
		\| \boundary_G S \|, x )_\natural (g(x)) | \ud \| \boundary_G
		S \| \, x \quad \text{for $g \in \mathscr D ( \mathbf R^n,
		\mathbf R^n)$}.
	\end{equation*}
\end{lemma}

\begin{proof}
	Suppose $g \in \mathscr D ( \mathbf R^n, \mathbf R^n )$ and $\lambda =
	\Lip g$.  We define $h : \mathbf R \times \mathbf R^n \to \mathbf R^n$
	and $h_t : \mathbf R^n \to \mathbf R^n$ such that $h(t,x) = h_t(x) =
	x+tg(x)$ whenever $(t,x) \in \mathbf R \times \mathbf R^n$; we
	abbreviate $I_t = \spt \boldsymbol [ 0,t \boldsymbol ]$ for $t \in
	\mathbf R$ and $\phi = \| \boundary_G S \|$.
	Identifying $\boldsymbol [ 0,t \boldsymbol ] \in \mathbf
	I_1 ( \mathbf R )$ with $\iota_{\mathbf R,1} ( \boldsymbol [0,t
	\boldsymbol ] ) \in \mathbf I_1 ( \mathbf R, \mathbf Z )$, see
	\cite[5.1]{arXiv:2206.14046v2}, we notice
	\begin{equation*}
		\boldsymbol [ 0,t \boldsymbol ] \times ( \boundary_G S) \in
		\mathscr R_m ( \mathbf R \times \mathbf R^n, G ) \quad
		\text{with} \quad \| \boldsymbol [ 0, t \boldsymbol ] \times
		\boundary_G S \| = ( \mathscr L^1 \restrict I_t) \times \phi
	\end{equation*}
	whenever $t \in \mathbf R$ by \cite[4.7]{arXiv:2206.14046v2}.
	Employing \cite[5.13\,(6), 5.16, 5.13\,(4), 4.5,
	4.6]{arXiv:2206.14046v2}, we then obtain
	\begin{gather*}
		(h_t)_\# S \in \mathbf I_m ( \mathbf R^n, G ), \quad h_\# (
		\boldsymbol [ 0,t \boldsymbol ] \times S ) \in \mathbf I_{m+1}
		( \mathbf R^n, G ), \\ 
		h_\# ( \boldsymbol [ 0,t  \boldsymbol ]
		\times \boundary_G S) \in \mathbf I_m ( \mathbf R^n, G ), \\
		(h_t)_\# S - S = \boundary_G h_\# ( \boldsymbol [ 0,t
		\boldsymbol ] \times S ) + h_\# ( \boldsymbol [ 0,t
		\boldsymbol ] \times \boundary_G S ), \\
		\quad \boundary_G \big (
		(h_t)_\# S - h_\# ( \boldsymbol [ 0,t \boldsymbol ] \times
		\boundary_G S ) \big ) = \boundary_G S, \\
		\| S \| ( \mathbf R^n ) - \| h_\# ( \boldsymbol [ 0,t
		\boldsymbol ] \times \boundary_G S ) \| ( \mathbf R^n ) \leq
		\| (h_t)_\# S \| ( \mathbf R^n ), \\
		\| h_\# ( \boldsymbol [ 0,t \boldsymbol ] \times \boundary_G
		S) \| ( \mathbf R^n ) \leq {\textstyle\int_{I_t \times \mathbf
		R^n}} \big \| {\textstyle\bigwedge_m} ( \mathscr L^1 \times
		\phi, m ) \ap \Der h \big \| \ud \mathscr L^1 \times \phi
	\end{gather*}
	whenever $t \in \mathbf R$; in case $|t| \lambda < 1$, we also observe
	that $h_t$ is a diffeomorphism with $\im h_t = \mathbf R^n$, hence $
	\| (h_t)_\# V \| ( \mathbf R^n ) = \| (h_t)_\# S \| ( \mathbf R^n )$
	by \cite[3.27, 3.28, 4.6]{arXiv:2206.14046v2}.  For $\mathscr L^1
	\times \phi$ almost all $(s,x) \in I_t \times \mathbf R^n$, recalling
	\begin{equation*}
		\Tan^m ( \mathscr L^1 \times \phi, (s,x)) = \mathbf R \times
		\Tan^{m-1} ( \phi,x )
	\end{equation*}
	from \cite[4.4, 4.7]{arXiv:2206.14046v2}, we finally estimate
	\begin{align*}
		\big \| {\textstyle\bigwedge_m} ( \mathscr L^1 \times \phi, m)
		\ap \Der h (s,x) \big \| \leq (1+|t|\lambda)^{m-1} \big |
		\Nor^{m-1} ( \phi, x)_\natural (g(x)) \big |\phantom , & \\
		+ (m-1) |t| \lambda ( 1+|t| \lambda )^{m-2} \big |
		\Tan^{m-1} ( \phi, x )_\natural (g(x)) \big |; &
	\end{align*}
	thus, \cite[4.1]{MR0307015} yields the conclusion.
\end{proof}

\begin{remark} \label{remark:connected-integral-minimisers}
	The case $G = \mathbf Z$ appears in \cite[4.8\,(4)]{MR0307015} whose
	method we employ.
\end{remark}

\begin{remark} \label{remark:connected-minimisers}
	It follows that if such $S$ is indecomposable, then $\spt \| S \|$ is
	connected by \cite[7.7, 10.9]{arXiv:2209.05955v2}.
\end{remark}
	
Theorem \ref{Thm:Plateau:chains} corresponds to the case $G = \mathbf Z$ in
the next theorem, see \cite[5.1]{arXiv:2206.14046v2}.

\begin{theorem} \label{theorem:diameter_bound_mass_minimising_chains}
	Suppose $m$ and $n$ are integers, $2 \leq m \leq n$, $G$ is a complete
	normed commutative group, $S \in \mathbf I_m( \mathbf R^n, G)$ is
	indecomposable,
	\begin{equation*}
		\| S \|( \mathbf R^n ) \leq \| T \| ( \mathbf R^n ) \quad
		\text{whenever $T \in \mathbf I_m ( \mathbf R^n, G)$ and
		$\boundary_G T = \boundary_G S$},
	\end{equation*}
	$\boldsymbol \Uptheta^m(\|S\|,x) \geq 1$ for $\| S \|$ almost all $x$,
	$d$ denotes the geodesic diameter of $\spt \| S \|$, and $\gamma =
	\Gamma_{\textup{\ref{thm:diameter_bound}}} (m)$; if $m > 2$, then
	suppose $W \in \mathbf V_{m-1} ( \mathbf R^n )$, $\| \updelta W \|$ is
	a Radon measure, $\boldsymbol \Uptheta^{m-1} (\| W \|, x) \geq 1$ for
	$\| W \|$ almost all $x$, and $\| \boundary_G S \| \leq \| W \|$; if
	$m > 3$, then suppose $\| \updelta W \|$ is absolutely continuous with
	respect to $\| W \|$.

	Then, the following three statements hold.
	\begin{enumerate}
		\item If $m = 2$, then $d \leq \gamma \, \| \boundary_G S \| (
		\mathbf R^n )$.
		\item If $m = 3$, then $d \leq \gamma \, \| \updelta W \| (
		\mathbf R^n )$.
		\item If $m > 3$, then $d \leq \gamma \int | \mathbf h (W,x)
		|^{m-2} \ud \| W \| \, x$.
	\end{enumerate}
\end{theorem}

\begin{proof}
	Let $V \in \mathbf{RV}_m ( \mathbf R^n )$ be characterised by $\| V \|
	= \| S \|$.  From \ref{lemma:first-variation-minimiser}, we obtain $\|
	\updelta V \| \leq \| \boundary_G S \|$; in particular, $\mathbf
	h(V,\cdot) = 0$ by \cite[3.5\,(1b)]{MR0307015}.  In view of
	\cite[10.9]{arXiv:2209.05955v2}, the conclusion then follows by
	applying \ref{thm:diameter_bound} with $(V_1,V_2)$ replaced by $(V,0)$
	if $m=2$ and $(V,W)$ if $m>2$.
\end{proof}

\begin{remark} \label{remark:diameter_bound_mass_minimising_currents}
	For any $S \in \mathbf I_m ( \mathbf R^n, G )$ satisfying
	\begin{equation*}
		\mathbf \| S \| ( \mathbf R^n ) \leq \mathbf \| T \| ( \mathbf
		R^n ) \quad \text{whenever $T \in \mathbf I_m ( \mathbf R^n, G
		)$ and $\boundary_G T = \boundary_G S$},
	\end{equation*}
	indecomposability of $\boundary_G S$ implies indecomposability of $S$.
\end{remark}

\begin{remark} \label{remark:classic-Abelian-groups}
	We recall \cite[5.1, 6.2]{arXiv:2206.14046v2} and suppose $G = \mathbf
	Z$ or $G = \mathbf Z / d \mathbf Z$ for some positive integer $d$.
	Then, whenever $B \in \mathbf I_{m-1} ( \mathbf R^n, G )$ with
	$\boundary_G B = 0$, there exists $S \in \mathbf I_m ( \mathbf R^n,
	G)$ with $\boundary_G S = B$ satisfying
	\begin{equation*}
		\mathbf \| S \| ( \mathbf R^n ) \leq \mathbf \| T \| ( \mathbf
		R^n ) \quad \text{whenever $T \in \mathbf I_m ( \mathbf R^n, G
		)$ and $\boundary_G T = \boundary_G S$};
	\end{equation*}
	in fact, in view of \cite[4.6, 5.13\,(6)]{arXiv:2206.14046v2}, we
	combine \cite[4.1.11, 4.1.16, 4.2.17\,(2), 4.2.26]{MR41:1976}.  For
	these $G$, the density hypotheses are redundant.  For general $G$, the
	validity of a rectifiability theorem analogous to
	\cite[4.2.16\,(3)]{MR41:1976}---which is central to this minimisation
	process---has been characterised in \cite[7.1]{MR1715323} in the
	context of the flat $m$-chains over $G$ introduced in \cite{MR185084};
	see \cite[p.\,7]{arXiv:2206.14046v2} regarding the pending comparison
	to the present concept of $G$ chains.
\end{remark}

\begin{remark} \label{remark:earlier-results}
	Here, we discuss the novelty of our results in the case that $G =
	\mathbf Z$, $B$ is a connected orientable compact $m-1$ dimensional
	submanifold of $\mathbf R^n$,  and $\| \boundary_{\mathbf Z} S \| =
	\mathscr H^{m-1} \restrict B$.  For this purpose, we abbreviate $\spt
	\| S \| = A$ and recall \cite[4.1.21, 4.1.31\,(2)]{MR41:1976}.  If $B$
	is of class $4$, then connectedness of $A \without B$, and thus of
	$A$, is a consequence of the profound regularity results leading to
	\cite[5.23]{MR1777737}; that $A$ is necessarily path-connected---in
	fact, must have finite geodesic diameter---is new.  If $B$ is of class
	$2$ and $m = n-1$, then finiteness of the geodesic diameter of $A$ can
	be deduced from earlier results---it suffices to combine \cite[Theorem
	6.8\,(1)]{MR3626845} with \cite[11.2\,(1)\,(3)]{MR554379}---whereas
	the a priori bound on the geodesic diameter in terms of $B$ and $m$ is
	new.
\end{remark}

\begin{remark} \label{remark:nonexistence}
	As in \cite[Section 3]{MR4553536}, the preceding theorem implies
	nonexistence of indecomposable solutions to the Plateau problem
	in the case that the diameter of $\spt \|
	\boundary_G S \|$ strictly exceeds the upper bound for $d$ in the
	conclusion by \cite[5.13\,(5)]{arXiv:2206.14046v2}.
\end{remark}

Now, we return to the setting of varifolds.

\begin{lemma} \label{lemma:density_of_variation_II}
	Suppose $m$ and $n$ are positive integers, $m \leq n$, $B$ is an $m-1$
	dimensional submanifold of $\mathbf R^n$ of class $2$, $0 < s < R$, $b
	\in B$, $V \in \mathbf V_m ( \mathbf U (b,s) )$, $b \in \spt \| V \|$,
	$\| V \| ( B \cap \mathbf U (b,s) ) = 0$,
	\begin{gather*}
		\big | \Nor ( B,z )_\natural ( z-y ) \big | \leq (2R)^{-1}
		|z-y|^2 \quad \text{for $y,z \in B$}, \\
		\mathbf U (b,s/(1-s/R)) \cap ( \Clos B ) \without B =
		\varnothing,
	\end{gather*}
	$\spt \updelta V \subset B$, and $\boldsymbol \Uptheta^m ( \| V \|, x )
	\geq 1$ for $\| V \|$ almost all $x$.

	Then, the following two statements hold.
	\begin{enumerate}
		\item \label{item:density_of_variation:variation} If $0 < r <
		s/2$, then
		\begin{equation*}
			\| \updelta V \| \, \mathbf B (b,r) \leq \big ( r^{-1}
			+ m/(R-s) \big ) \| V \| \, \mathbf B (b,2r);
		\end{equation*}
		in particular, $\boldsymbol \Uptheta^{\ast m-1} ( \| \updelta V
		\|, b ) \leq 2^m \boldsymbol \upalpha (m-1)^{-1} \boldsymbol
		\upalpha (m) \boldsymbol \Uptheta^m ( \| V \|, b )$.
		\item \label{item:density_of_variation:mass} The function
		mapping $0 < r \leq s$ onto
		\begin{equation*}
			\frac{\| V \| \, \mathbf U (b,r)}{r^m} \exp \left (
			\frac{3mR}{2(R-s)^2} r \right )
		\end{equation*}
		is nondecreasing.  Moreover, there holds $1/2 \leq \boldsymbol
		\Uptheta^m ( \| V \|, b ) < \infty$.
	\end{enumerate}
\end{lemma}

\begin{proof}
	We notice that \cite[2.2\,(3)\,(4), 3.1, 3.5\,(1)\,(2)]{MR0397520}
	remain valid if the submanifold $B$ therein is required to be of class
	$2$, instead of class $\infty$.  From \cite[2.2\,(4b)]{MR0397520}, we
	thus infer that for $a \in \mathbf U (b,s)$ there exists a unique $y
	\in B$ with $|y-a| = \dist (a,B) < R$.  Therefore, \cite[2.2\,(3),
	3.1\,(1)\,(2)]{MR0397520} yield
	\begin{equation*}
		\| \updelta V \| ( \psi ) \leq {\textstyle\int} | \Der \psi
		(x) \circ S_\natural | \ud V \, (x,S) + m/(R-s)
		{\textstyle\int} \psi \ud \| V \|
	\end{equation*}
	for $0 \leq \psi \in \mathscr D ( \mathbf U (b,s), \mathbf R )$; in
	particular, $\| \updelta V \|$ is a Radon measure and hence $V \in
	\mathbf{RV}_m ( \mathbf U (b,s))$ by \cite[5.5\,(1)]{MR0307015}.
	Moreover, we have
	\begin{equation*}
		1/2 \leq \boldsymbol \Uptheta^m ( \| V \|, b) < \infty
	\end{equation*}
	by \cite[4.8\,(2)\,(4)]{MR3528825} if $m=1$ and by
	\cite[3.5\,(1)\,(2)]{MR0397520} if $m \geq 2$.

	To verify \eqref{item:density_of_variation:variation}, we define $f :
	\mathbf U (b,s) \to \mathbf R$ by
	\begin{equation*}
		f(x) = \sup \big \{ 0, 1 - r^{-1} \dist (x, \mathbf B(b,r) )
		\big \} \quad \text{for $x \in \mathbf U (b,s)$},
	\end{equation*}
	hence $0 \leq f \leq 1$, $f(x) = 1$ for $x \in \mathbf B (b,r)$, $\Lip
	f = r^{-1}$, and $\spt f = \mathbf B (b,2r)$.  Recalling
	\cite[3.5\,(1b)]{MR0307015} and \cite[8.7]{MR3528825}, and
	approximating $f$ by nonnegative members of $\mathscr D ( \mathbf U
	(b,s), \mathbf R)$ for instance by means of \cite[3.7]{MR3626845}, we
	conclude
	\begin{equation*}
		\| \updelta V \| ( f ) \leq {\textstyle\int} | V \weakD f |
		\ud \| V \| +  m/(R-s) {\textstyle\int} f \ud \| V \|
	\end{equation*}
	and infer \eqref{item:density_of_variation:variation}, as $| V \weakD
	f (x) | \leq \Lip f$ for $x \in \dmn V \weakD f$.

	Finally, the first conclusion of
	\eqref{item:density_of_variation:mass} may be obtained by adapting the
	case $\alpha = 0$ of \cite[3.4\,(2)]{MR0397520} as follows:%
	\begin{footnote}%
		{This presupposes the following typographical corrections to
		\cite[3.4\,(2)]{MR0397520}: in the statement, replace
		``$m(s)^{1/k} /(R-s)$'' by ''$k m(s)^{1/k}/(R-s)<1$'' on line
		3 thereof; on page 427, replace ``(6)'' by ``(b)'' on line 7,
		``$\mathscr D ( \mathbf R )$'' by ``$\mathscr E^0 ( \mathbf
		R)$'', ``$\Phi$'' by ``$\varphi$'', and ``near $0$'' by ``near
		$0$ with $\sup \spt \varphi < s$'' all on line 19, and
		``$\xi$'' by ``$\zeta$'' on line 27.}
	\end{footnote}%
	Recalling that \cite[2.2\,(3)\,(4), 3.1\,(3)]{MR0397520} remain valid
	for $B$ therein required to be of class $2$ instead of class $\infty$,
	in the statement of \cite[3.4\,(2)]{MR0397520}, omit the definition of
	$\mu$ and modify the definition of $\Phi(r)$ to $\Phi(r) =
	\frac{3kR}{2(R-s)^2} r$, and, in its proof, omit the first paragraph,
	replace $\alpha$ by $0$ throughout, replace ``$g \in \mathscr X (
	\mathbf U(0,s) )$'' by ``$g : \mathbf U (0,s) \to \mathbf R^n$ is of
	class $1$ with compact support'', and omit the summand involving $\mu$
	in the last equation.
\end{proof}

\begin{remark}
	The method of proof of \eqref{item:density_of_variation:variation} is
	adapted from \cite[3.4\,(1)]{MR0397520}.
\end{remark}

\begin{remark}
	If $m = 1$, then the sharp estimate $\boldsymbol \Uptheta^0 ( \|
	\updelta V \|, b ) \leq 2 \boldsymbol \Uptheta^1 ( \| V \|, b )$ holds
	as we could apply \cite[3.1\,(2)]{MR0397520} after the first paragraph
	of the proof.
\end{remark}

The preceding lemma readily yields Theorem
\ref{FinalThm:geom-var-I}.

\begin{theorem} \label{thm:estimate_variation_boundary}
	Suppose $m$ and $n$ are positive integers with $m \leq n$, $B$ is a
	nonempty
	compact $m-1$ dimensional submanifold of class $2$ of $\mathbf R^n$,
	 $R = \reach (B)$, $V \in \mathbf V_m ( \mathbf
	R^n)$, $\| V \| ( B ) = 0$, $\spt \updelta V \subset B \subset \spt \|
	V \|$, $\boldsymbol \Uptheta^m ( \| V \|, x ) \geq 1$ for $\| V \|$
	almost all $x$,
	\begin{equation*}
		M = R^{-m} \sup \{ \| V \| \, \mathbf U (b,R/2) \with b \in B
		\},
	\end{equation*}
	and $\Gamma = 4^m \boldsymbol \upalpha (m-1)^{-1} \exp ( 3m )$.

	Then, there holds $0 < R < \infty$, $\Gamma M \geq 2^{m-1} \boldsymbol
	\upalpha (m) \boldsymbol \upalpha (m-1)^{-1}$, and
	\begin{equation*}
		\| \updelta V \| \leq \Gamma M \mathscr H^{m-1} \restrict B.
	\end{equation*}
\end{theorem}

\begin{proof}
	Since $\reach (B,b) > 0$ for $b \in B$ by \cite[4.12]{MR0110078},
	there holds $R > 0$ by \cite[4.2]{MR0110078}.  Moreover, observing
	that $B$ cannot be convex, we obtain $R < \infty$ from
	\cite[4.2]{MR0110078}; in particular, the first conclusion holds.
	Next, \cite[4.18]{MR0110078} yields
	\begin{equation*}
		| \Nor (B,b)_\natural (y-b) | \leq (2R)^{-1} |y-b|^2 \quad
		\text{whenever $y,b \in B$}.
	\end{equation*}
	Applying \ref{lemma:density_of_variation_II}%
	\,\eqref{item:density_of_variation:variation} with $s = R/2$ and
	\ref{lemma:density_of_variation_II}%
	\,\eqref{item:density_of_variation:mass} with $r = s = R/2$, we obtain
	\begin{equation*}
		\boldsymbol \Uptheta^{\ast m-1} ( \| \updelta V \|, b ) \leq
		2^m \boldsymbol \upalpha (m-1)^{-1} \boldsymbol \upalpha (m)
		\boldsymbol \Uptheta^m ( \| V \|, b ) \leq \Gamma M \quad
		\text{for $b \in B$};
	\end{equation*}
	in particular, the second conclusion holds because $\boldsymbol
	\Uptheta^m ( \| V \|, b ) \geq 1/2$ for $b \in B$ by
	\ref{lemma:density_of_variation_II}%
	\,\eqref{item:density_of_variation:mass}.  As $\boldsymbol
	\Uptheta^{m-1} ( \mathscr H^{m-1} \restrict B, b) = 1$ by
	\cite[3.1.23, 3.2.17]{MR41:1976}, we infer
	\begin{equation*}
		\limsup_{r \to 0+} \frac{\| \updelta V \| \, \mathbf B (b,r)}{
		( \mathscr H^{m-1} \restrict B ) \, \mathbf B (b,r)} \leq
		\Gamma M \quad \text{for $b \in B$}.
	\end{equation*}
	Since $\| \updelta V \| ( \mathbf R^n \without B ) = 0$, this implies
	firstly that $\| \updelta V \|$ is absolutely continuous with respect
	to $\mathscr H^{m-1} \restrict B$ by \cite[2.9.2, 2.9.15]{MR41:1976}
	and then the last conclusion by \cite[2.8.18, 2.9.7]{MR41:1976}.
\end{proof}

Boundary connectedness may be exploited with the following
lemma.

\begin{lemma} \label{lemma:plateau_connectedness}
	Suppose $m$ and $n$ are positive integers, $m \leq n$, $U$ is an open
	subset of $\mathbf R^n$, $X \in \mathbf V_m ( U )$, $\| X \| ( U ) <
	\infty$, $\beta = \infty$ if $m=1$, $\beta = m/(m-1)$ if $m>1$,
	\begin{equation*}
		\sup \{ ( \updelta X ) ( \theta ) \with \theta \in \mathscr D
		( U, \mathbf R^n ), \| X \|_{(\beta)} ( \theta ) \leq 1 \} <
		\boldsymbol \gamma (m)^{-1},
	\end{equation*}
	and $\boldsymbol \Uptheta^m ( \| X \|, x ) \geq 1$ for $\| X \|$
	almost all $x$.

	Then, the closure of every connected component of $\spt \| X \|$ meets
	$\Bdry U$; in particular, if $\Bdry U \cap \Clos \spt \| X \|$ is
	connected, then so is $\Clos \spt \| X \|$.
\end{lemma}

\begin{proof}
	Let $i : U \to \mathbf R^n$ be the inclusion map and denote by $\Phi$
	the family of connected components of $\spt \| X \|$.  Suppose $C \in
	\Phi$, hence $\spt ( \| X \| \restrict C ) = C$ and $\| \updelta ( X
	\restrict C \times \mathbf G (n,m) ) \| = \| \updelta X \| \restrict
	C$ by \cite[6.14\,(4)]{MR3528825} and \cite[5.1]{arXiv:2209.05955v2}.
	Defining
	\begin{equation*}
		V = i_\# ( X \restrict C \times \mathbf G (n,m))
		\in \mathbf V_m (  \mathbf R^n ),
	\end{equation*}
	we note $0 < \| V \| ( \mathbf R^n ) < \infty$ and infer $\spt \| V \|
	= \Clos C$ and
	\begin{equation*}
		\spt \updelta V \subset \spt \| V \| \subset \Clos U, \quad \|
		\updelta V \| \restrict U = i_\# ( \| \updelta X \| \restrict
		C ).
	\end{equation*}
	We conclude $\| \updelta V \| ( \Bdry U ) > 0$ because, using
	\cite[3.3, 3.4]{arXiv:2209.05955v2}, we may estimate
	\begin{align*}
		\boldsymbol \gamma (m)^{-1} \| V \| ( \mathbf R^n )^{1-1/m}
		& \leq \| \updelta V \| ( \mathbf R^n ) \leq \| \updelta X \|
		( U ) + \| \updelta V \| ( \Bdry U ) \\
		& < \boldsymbol \gamma(m)^{-1} \| V \| ( \mathbf R^n )^{1-1/m}
		+ \| \updelta V \| ( \Bdry U );
	\end{align*}
	hence, $\Bdry U \cap \spt \| V \| \neq \varnothing$ and the principal
	conclusion follows.  Accordingly, if $B = \Bdry U \cap \Clos \spt \| X
	\|$ is connected, then so are $B \cup \bigcup \{ \Clos C \with C \in
	\Phi \}$ and its closure $\Clos \spt \| X \|$.
\end{proof}

\begin{remark}
	If $m=1$, then $\upgamma(1)=2^{-1}$ by \cite[3.8, 3.9]{MR3777387} and
	$\upgamma(1)^{-1}$ may not be replaced by any larger value in the
	preceding lemma.
\end{remark}

\begin{remark} \label{remark:sharp-values}
	If $m=2$ and $V$ is integral, then $\upgamma(m)^{-1}$ may be replaced
	by the larger value $(\int_{\mathbf S^m} | \mathbf h ( \mathbf S^m,
	\cdot ) |^m \ud \mathscr H^m )^{1/m}$ by \cite[A.18]{MR2119722}; the
	latter value is evidently sharp. $\big($Based on
	\cite[24.1]{arXiv:2310.01754v1}, the same replacement is feasible in
	case $2 \leq m = n-1$ and $V$ is integral.$\big)$
\end{remark}

\begin{remark} \label{remark:Brakke-flow}
	To illustrate the preceding lemma, let $B = \Bdry U \cap \Clos \spt \|
	W \|$ be the \emph{fixed boundary} of either a \emph{Brakke flow}
	$V_t$, $0 \leq t < \infty$, or a stationary integral varifold
	$V_\infty$, resulting from it as subsequential limit as $t \to
	\infty$; one is assured of the existence of such objects under broad
	conditions by \cite[Theorem 2.2, Corollary 2.4]{MR4204569}.  Thus, if
	this boundary $B$ is connected, so must be $\Clos \spt \| V_t \|$,
	whenever $V_t$ satisfies the first variation condition, and $\Clos
	\spt \| V_\infty \|$.
\end{remark}

In combination with results from \cite{arXiv:2209.05955v2}, Theorem
\ref{FinalThm:geom-var-II} now follows.

\begin{theorem} \label{thm:Reifenberg-Plateau-diameter}
	Suppose $m$ and $n$ are integers, $2 \leq m \leq n$, $B$ is a compact
	connected $m-1$ dimensional submanifold of class $2$ of $\mathbf R^n$,
	$V \in \mathbf V_m ( \mathbf R^n)$, $\| V \| ( \mathbf R^n ) <
	\infty$, $1 \leq \lambda < \infty$,
	\begin{equation*}
		B \subset \spt \| V \|, \quad \| \updelta V \| \leq \lambda
		\mathscr H^{m-1} \restrict B,
	\end{equation*}
	$\boldsymbol \Uptheta^m ( \| V \|, x ) \geq 1$ for $\| V \|$ almost
	all $x$, and $d$ is the geodesic diameter of $\spt \| V \|$.

	Then, there holds
	\begin{equation*}
		d \leq \Gamma_{\textup{\ref{thm:diameter_bound}}} (m) \lambda
		{\textstyle\int_B} | \mathbf h ( B, b ) |^{m-2} \ud \mathscr
		H^{m-1} \, b;
	\end{equation*}
	here, $0^0=1$.
\end{theorem}

\begin{proof}
	We note $V \in \mathbf{RV}_m ( \mathbf R^n )$ by
	\cite[5.5\,(1)]{MR0307015}; in particular, $\| V \| ( B ) = 0$ by
	\cite[3.5\,(1b)]{MR0307015}.  Thus, taking $X = V | \mathbf 2^{(
	\mathbf R^n \without B ) \times \mathbf G (n,m)}$, the set
	\begin{equation*}
		\spt \| V \| = \Clos \spt \| X \|
	\end{equation*}
	is connected by \ref{lemma:plateau_connectedness}.  We then apply
	\cite[10.16]{arXiv:2209.05955v2} in conjunction with
	\cite[9.2]{arXiv:2209.05955v2} to conclude that $V$ is indecomposable
	of type $\mathscr D ( \mathbf R^n, \mathbf R )$.  We
	let $W \in \mathbf V_{m-1} ( \mathbf R^n )$ be
	defined by
	\begin{equation*}
		W(k) = \lambda {\textstyle\int_B} k (b,\Tan(B,b)) \ud \mathscr
		H^{m-1} \, b \quad \text{for $k \in \mathscr K ( \mathbf R^n
		\times \mathbf G (n,m-1))$}
	\end{equation*}
	so that
	\begin{equation*}
		\| \updelta V \| \leq \| W \|.
	\end{equation*}
	The conclusion now follows from \ref{thm:diameter_bound} applied with
	$(V_1,V_2)$ replaced by $(V,0)$ if $m=2$ and $(V,W)$ if $m > 2$,
	respectively.
\end{proof}

\begin{remark} \label{remark:min-max}
	Our hypotheses have been arranged so as to include with $\lambda = 1$
	the varifolds furnished by the min-max methods of \cite[Theorem
	2.6\,(b)]{MR3893761} and \cite[Theorem 1.3]{MR4160865} when, in these
	theorems, the prescribed boundary is connected and the ambient
	Riemannian manifold is a convex body in $\mathbf R^n$.
\end{remark}

The preceding two theorems combine into two corollaries.

\begin{corollary} \label{corollary:Reifenberg-Plateau-diameter}
	Suppose $m$ and $n$ are integers, $2 \leq m \leq n$, $B$ is a nonempty
	compact connected $m-1$ dimensional submanifold of class $2$ of
	$\mathbf R^n$, $R = \reach (B)$, $V \in \mathbf V_m ( \mathbf R^n)$,
	$\| V \| ( \mathbf R^n ) < \infty$, $\| V \| ( B ) = 0$,
	\begin{equation*}
		\spt \updelta V \subset B \subset \spt \| V \|,
	\end{equation*}
	$\boldsymbol \Uptheta^m ( \| V \|, x ) \geq 1$ for $\| V \|$ almost all
	$x$, $M = R^{-m} \sup \{ \| V \| \, \mathbf U (b,R/2) \with b \in B
	\}$, and $d$ is the geodesic diameter of $\spt \| V \|$.

	Then, there holds
	\begin{equation*}
		d \leq \Gamma M {\textstyle\int_B} | \mathbf h ( B, b )
		|^{m-2} \ud \mathscr H^{m-1} \, b,
	\end{equation*}
	where $\Gamma$ is a positive, finite number determined by $m$; here,
	$0^0=1$.
\end{corollary}

\begin{proof}
	Let $\kappa = 2^{m-1} \boldsymbol \upalpha (m) \boldsymbol \upalpha
	(m-1)^{-1}$.  We will show that one may take
	\begin{equation*}
		\Gamma = \sup \big \{ \kappa^{-1}, 1 \big \}
		\Gamma_{\textup{\ref{thm:diameter_bound}}} (m)
		\Gamma_{\textup{\ref{thm:estimate_variation_boundary}}} (m).
	\end{equation*}
	From \ref{thm:estimate_variation_boundary}, we obtain
	$\Gamma_{\textup{\ref{thm:estimate_variation_boundary}}} (m) M \geq
	\kappa$ and
	\begin{equation*}
		\| \updelta V \| \leq
		\Gamma_{\textup{\ref{thm:estimate_variation_boundary}}} (m) M
		\mathscr H^{m-1} \restrict B;
	\end{equation*}
	hence, we may apply \ref{thm:Reifenberg-Plateau-diameter} with
	$\lambda = \sup \{ 1,
	\Gamma_{\textup{\ref{thm:estimate_variation_boundary}}} (m) M \}$.
\end{proof}

\begin{remark} \label{remark:Brakke-flow-II}
	The preceding corollary is applicable to the image in $\mathbf R^n$ of
	the varifolds furnished by subsequential limits as time approaches
	$\infty$ of the Brakke flow with fixed boundary in \cite[Corollary
	2.4]{MR4204569} when the boundary in the cited source is a connected
	$m-1$ dimensional submanifold of class $2$ of $\mathbf R^n$; in this
	case, $M$ is bounded a priori by the initial conditions of the flow.
\end{remark}

\begin{corollary} \label{corollary:Reifenberg-Plateau-diameter-sets}
	Suppose $m$ and $n$ are integers, $2 \leq m \leq n$, $B$ is a nonempty
	compact connected $m-1$ dimensional submanifold of class $2$ of
	$\mathbf R^n$, $R = \reach (B)$, $S$ is a $(\mathscr H^m,m)$
	rectifiable subset of $\mathbf R^n$,
	\begin{gather*}
		B \subset S, \quad S = \spt ( \mathscr H^m \restrict S ), \\
		{\textstyle\int_S \Tan^m ( \mathscr H^m \restrict S, x )
		\bullet \Der \theta (x) \ud \mathscr H^m \, x = 0} \quad
		\text{for $\theta \in \mathscr D ( \mathbf R^n \without B,
		\mathbf R^n)$}, \\
		M = R^{-m} \sup \{ \mathscr H^m ( S \cap  \mathbf U (b,R/2) )
		\with b \in B \},
	\end{gather*}
	and $d$ is the geodesic diameter of $S$.

	Then, there holds
	\begin{equation*}
		d \leq
		\Gamma_{\textup{\ref{corollary:Reifenberg-Plateau-diameter}}}
		(m) M {\textstyle\int_B} | \mathbf h ( B, b ) |^{m-2} \ud
		\mathscr H^{m-1} \, b;
	\end{equation*}
	here, $0^0=1$.
\end{corollary}

\begin{proof}
	We take $V \in \mathbf{RV}_m ( \mathbf R^n )$ such that $\| V \| =
	\mathscr H^m \restrict S$ in
	\ref{corollary:Reifenberg-Plateau-diameter}.
\end{proof}

Finally, we include the derivation of Theorem
\ref{FinalThm:Reifenberg-Plateau}.

\begin{remark} \label{remark:Reifenberg-Plateau}
	Adapting the terminology of
	\cite[\S\,12]{MR3800850} and drawing from the literature,%
	\begin{footnote}%
		{The authors did not verify the results of
		\cite{MR0050886,MR0420406,MR3998213,MR4489608} employed.}
	\end{footnote}%
	we will deduce the following corollary: \emph{If $G$ is a commutative
	group, $L$ is a subgroup of the $(m-1)$-th Čech homology group of $B$
	with coefficients in~$G$, $\mathscr{\check C} (B,L,G)$ denotes the
	family of compact subsets of $\mathbf R^n$ spanning $L$,
	\begin{equation*}
		E \in \mathscr{\check C} (B,L,G), \quad \mathscr H^m (E) =
		\inf \big \{ \mathscr H^m (F) \with F \in \mathscr{\check C}
		(B,L,G) \big \},
	\end{equation*}
	$S = \spt ( \mathscr H^m \restrict E )$, then $\mathscr H^m (E) \leq
	\mathscr H^{m-1} (B) \diam(B)/m$ and the geodesic diameter of $S$ does
	not exceed
	\begin{equation*}
		\Gamma_{\textup{\ref{corollary:Reifenberg-Plateau-diameter}}}
		(m) \reach(B)^{-m} \mathscr H^m (E) {\textstyle\int_B} |
		\mathbf h (B,b) |^{m-2} \ud \mathscr H^{m-1} \, b.
	\end{equation*}}%
	We assume $L \neq \{ 0 \}$ because $L = \{ 0 \}$ implies $\mathscr H^m
	(E) = 0$.  Denoting by $C$ the convex hull of $B$ and verifying
	\begin{equation*}
		Y_c = \{ tb+(1-t)c \with b \in B, 0 \leq t \leq 1 \} \in
		\mathscr{\check C} (B,L,G) \quad \text{for $c \in C$}
	\end{equation*}
	by means of \cite[Chapter 9, Theorems 3.4, 4.4, and 5.1]{MR0050886},
	the estimate of $\mathscr H^m (E)$ follows from
	\cite[3.2.20]{MR41:1976}; thus, noting \cite[2.1.3]{MR4489608}, $E
	\without B$ is either empty or $(1, \infty)$ restricted with respect
	to $B$ in the sense of \cite[II.1]{MR0420406} so that $E$ is
	$(\mathscr H^m,m)$ rectifiable by \cite[II.3\,(9)]{MR0420406} and
	\begin{gather*}
		{\textstyle\int_E \Tan^m ( \mathscr H^m \restrict E, x
		)_\natural \bullet \Der \theta (x) \ud \mathscr H^m \, x = 0}
		\quad \text{for $\theta \in \mathscr D ( \mathbf R^n \without
		B, \mathbf R^n)$}
	\end{gather*}
	by \cite[4.1]{MR0307015}.  Inferring $S \in \mathscr{\check C}
	(B,L,G)$, from \cite[2.2.1]{MR4489608} applied with $E_k$, $E$, and
	$d$ replaced by $E$, $S$, and $m$, and concluding $B \subset S$ from
	\cite[1.2.2]{MR4489608}, the preceding corollary becomes applicable.
	We also record that, noting \cite[2.10.11, 4.1.16]{MR41:1976} and
	\cite[2.1.3]{MR4489608}, the existence of such a set $E$, additionally
	contained $C$, is guaranteed for instance by \cite[3.2.2]{MR4489608}
	applied with $\Gamma = B$, $d=m$, and $\mathscr I = \mathscr H^m$.
	Finally, the isoperimetric inequality \cite[Theorem 2]{MR3998213}
	yields
	\begin{equation*}
		\mathscr H^m (E) \leq \Delta \, \mathscr H^{m-1}
		(B)^{m/(m-1)},
	\end{equation*}
	where $\Delta$ is a finite number determined by $n$; in fact, in view
	of \cite[2.1.3, 2.2.1]{MR4489608}, the estimate is entailed by
	applying the inequality with $A$, $Y$, and $L$, replaced by $B$,
	$Y_c$, and $\mathscr H^{m-1} (B)^{1/(m-1)}$ for some $c \in C$ because
	$Y_c \in \mathscr{\check C} (B,L,G)$.
\end{remark}

\bibliographystyle{myalphaurl}
\bibliography{ohne_duplikate}

\medskip \noindent \textsc{Affiliations}

\medskip \noindent Ulrich Menne \smallskip \newline
Department of Mathematics \\
National Taiwan Normal University \\
No.88, Sec.4, Tingzhou Rd. \\
Wenshan Dist., \textsc{Taipei City 116059 \\
        Taiwan(R.\ O.\ C.)}

\medskip \noindent Christian Scharrer \smallskip \newline Institute for Applied Mathematics
\newline University of Bonn
\newline Endenicher Allee 60 \newline
\textsc{53115 Bonn} \\ \textsc{Germany}

\medskip \noindent \textsc{Email addresses}

\medskip \noindent
\href{mailto:Ulrich.Menne@math.ntnu.edu.tw}{Ulrich.Menne@math.ntnu.edu.tw}
\quad
\href{mailto:Scharrer@iam.uni-bonn.de}{Scharrer@iam.uni-bonn.de}

\end{document}